\documentclass[reqno,10pt]{amsart}

\usepackage{hyperref}

\usepackage{amssymb,graphicx,color}
\usepackage[latin1]{inputenc}

\theoremstyle{plain}
\newtheorem{thm}{Theorem}[section]
\newtheorem{lem}{Lemma}[section]
\newtheorem{prop}{Proposition}[section]

\newtheorem{defs}{Definition}[section]
\newtheorem{prob}{Problem}[section]
\theoremstyle{definition}
\newtheorem{rmk}{Remark}[section]

\newcommand{\NN}{{\mathbb{N}}}
\newcommand{\ZZ}{{\mathbb{Z}}}
\newcommand{\RR}{{\mathbb{R}}}

\newcommand{\bu}{\mathbf{u}}
\newcommand{\bv}{\mathbf{v}}
\newcommand{\bw}{\mathbf{w}}
\newcommand{\be}{\mathbf{e}}

\newcommand{\bb}{\mathbf{b}}
\newcommand{\bg}{\mathbf{g}}

\newcommand{\bx}{\mathbf{x}}
\newcommand{\by}{\mathbf{y}}

\newcommand{\bbf}{\mathbf{f}}

\newcommand{\bA}{\mathbf{A}}
\newcommand{\bB}{\mathbf{B}}

\newcommand{\bX}{\mathbf{X}}

\newcommand{\bvarphi}{\boldsymbol{\varphi}}

\newcommand{\bxi}{\boldsymbol{\xi}}
\newcommand{\bnabla}{\boldsymbol{\nabla}}

\newcommand{\bfeta}{\boldsymbol{\eta}}

\newcommand{\bgamma}{{\boldsymbol{\gamma}}}

\newcommand{\define}{\stackrel{\text{\rm def}}{=}}
\newcommand{\Ccal}{{\mathcal C}}
\newcommand{\Acal}{{\mathcal A}}

\newcommand{\Vcal}{{\mathcal V}}

\newcommand{\loc}{{\text{\rm loc}}}

\newcommand{\Dcal}{{\mathcal D}}

\newcommand{\rd}{{\text{\rm d}}}

\newcommand{\sfA}{\mathsf{A}}
\newcommand{\sfB}{\mathsf{B}}
\newcommand{\sfD}{\mathsf{D}}
\newcommand{\sfF}{\mathsf{F}}
\newcommand{\sfQ}{\mathsf{Q}}
\newcommand{\sfT}{\mathsf{T}}
\newcommand{\sfI}{\mathsf{I}}
\newcommand{\sfL}{\mathsf{L}}
\newcommand{\sfM}{\mathsf{M}}
\newcommand{\sfS}{\mathsf{S}}
\newcommand{\sfsigma}{\mathsf{\sigma}}

\newcommand{\tr}{\text{\rm tr}}

\newcommand{\Tr}{\operatorname{Tr}}
\newcommand{\essinf}{\operatorname*{ess\,inf}}

\newcommand{\dual}[1]{\langle{#1}\rangle}

\begin{document}
\numberwithin{equation}{section}


\title[On a viscoelastic Stokes problem for salt rocks]{On the existence, uniqueness and regularity of solutions of a viscoelastic Stokes problem modelling salt rocks}

\author{R. A. Cipolatti}
\author{I.-S. Liu}
\author{L. A. Palermo}
\author{M. A. Rincon}
\author{R. M. S. Rosa}

\address[Rolci A. Cipolatti]{Instituto de Matem\'atica, Universidade Federal do Rio de Janeiro, Brazil}
\address[I-Shih Liu]{Instituto de Matem\'atica, Universidade Federal do Rio de Janeiro, Brazil}
\address[Luiz A. Palermo]{CENPES/Petrobr\'as, Rio de Janeiro, Brasil}
\address[Mauro A. Rincon]{Instituto de Matem\'atica, Universidade Federal do Rio de Janeiro, Brazil}
\address[Ricardo M. S. Rosa]{Instituto de Matem\'atica, Universidade Federal do Rio de Janeiro, Brazil}

\email[R. A. Cipolatti]{cipolatti@im.ufrj.br}
\email[I.-S. Liu]{liu@im.ufrj.br}
\email[L. A. Palermo]{luizpalermo@petrobras.com.br}
\email[M. A. Rincon]{rincon@dcc.ufrj.br}
\email[R. M. S. Rosa]{rrosa@im.ufrj.br}

\date{January 25, 2017}

\thanks{All the authors acknowledge the financial support of CENPES/PETROBR\'AS. The last author was also partly supported by CNPq, Bras\'{\i}lia, Brazil.}


\begin{abstract}
A Stokes-type problem for a viscoelastic model of salt rocks is considered, and existence, uniqueness and regularity are investigated in the scale of $L^2$-based Sobolev spaces. The system is transformed into a generalized Stokes problem, and the proper conditions on the parameters of the model that guarantee that the system is uniformly elliptic are given. Under those conditions, existence, uniqueness and low-order regularity are obtained under classical regularity conditions on the data, while higher-order regularity is proved under less stringent conditions than classical ones. Explicit estimates for the solution in terms of the data are given accordingly.
\end{abstract}

\maketitle

\section{Introduction}
We are interested in the viscoelastic Stokes-type problem 
\begin{equation}
  \label{viscoelasticStokesorg}
   \begin{cases}
      - \mu_1 \Delta \bv - \mu_2 \bnabla \cdot (\sfD(\bv)\sfB + \sfB\sfD(\bv)) \\ 
      \qquad \qquad \qquad -\mu_ 3\bnabla \cdot (\sfD(\bv)\sfB^{-1} + \sfB^{-1} \sfD(\bv)) + \bnabla p = \bbf, & \text{in } \Omega,\\
      \bnabla \cdot \bv = 0, & \text{in } \Omega,\\
      \bv = 0, & \text{on } \partial \Omega,
    \end{cases}
\end{equation}
on a three-dimensional bounded, smooth domain $\Omega\subset \RR^3$, where $\bx=(x_1,x_2,x_3)\in \RR^3$ is the spatial variable and the unknowns are the scalar field $p=p(\bx)$ and the three-dimensional vector field $\bv=\bv(\bx)$, with coordinates $\bv=(v_1,v_2,v_3)$. In this system of equations, $\sfD=\sfD(\bv)$ is the symmetric gradient tensor given by
\begin{equation}
  \label{straintensor}
  \sfD(\bv) = \frac{1}{2}(\bnabla \bv + (\bnabla \bv)^\tr),
\end{equation}
where the superscript ``${}\tr$'' denotes the transpose of a matrix; $\sfB=\sfB(\bx)$ is a given symmetric, positively defined tensor with coordinates $\sfB=(b_{ij})_{ij}$; and the parameters $\mu_1, \mu_2, \mu_3$ are scalar fields.

Our aim is to obtain conditions on the parameters $\mu_1, \mu_2$, and $\mu_3$, and on the tensor $\sfB$ that guarantee that the problem \eqref{viscoelasticStokesorg} is uniformly elliptic and that we have existence, uniqueness and regularity properties of the solutions. We also aim to give explicit estimates for the regularity of the solutions depending on the regularity of $\sfB$, $\bbf$, and the parameters $\mu_1, \mu_2, \mu_3$. 

Problem \eqref{viscoelasticStokesorg} does not seem to have been treated directly before. We rewrite \eqref{viscoelasticStokesorg} into a form that fits previously considered general system and verify the necessary conditions that guarantee the existence, uniqueness and the desired regularity up to $H^2$ in $\bv$ and $H^1$ in $p$. For higher-order regularity, the known general results require too much regularity from $\sfB$ and we present a proof that yields the desired regularity upon less stringent conditions. Then we work out the estimates in detail. These are particularly important for the application of this result to the time-dependent problem mentioned below and addressed in a follow-up work.

The parameters $\mu_1, \mu_2$, and $\mu_3$ need not be constant, they are assumed to be general scalar fields. In applications, this dependence is typically on invariants of the model, such as the trace  of $\sfB$, or of powers of $\sfB$, and they may also depend on the temperature distribution of the medium. The actual dependence is important when analysing the corresponding evolutionary problem coupled with  equations for $\sfB$ and for the temperature. Here, however, we consider $\mu_1, \mu_2$, and $\mu_3$ as given general scalar fields. 

Problem \eqref{viscoelasticStokesorg} arises as the dissipative part for the linear momentum equation of a homogeneous and incompressible model of salt rocks, for which the constitutive law for the stress tensor $\sfsigma$ is of the form
 \begin{equation}
  \label{sigmaconstitutive}
   \sfsigma = -p\sfI + 2\mu_1 \sfD + s_1 \sfB + s_{-1} \sfB^{-1} + \mu_2 (\sfD\sfB + \sfB\sfD) + \mu_3(\sfD\sfB^{-1} + \sfB^{-1}\sfD).
\end{equation}
In \eqref{sigmaconstitutive}, the tensor $\sfI$ is the identity in three-dimension; $p$ is the kinematic pressure; $\bv$ is the velocity field; $\bnabla \bv$ is the velocity gradient; $\sfD$ is the strain-rate tensor defined in \eqref{straintensor}, which is the symmetric part of the velocity gradient;  $\sfB=\sfF\sfF^\tr$ is the left Cauchy-Green tensor; $\sfF$ is the deformation gradient tensor; $s_1$ and $s_{-1}$ are Mooney-Rivlin elastic parameters; the $\mu_1$ term is a Newtonian viscous stress; and $\mu_2$ and $\mu_3$ are Mooney-Rivlin-type parameters for the combined viscoelastic terms in the constitutive law (see \cite{truesdellnoll2004, haupt2002, Liu2002, temammiranville2005} for the fundamentals of continuum mechanics). 

The material model given by the stress tensor form \eqref{sigmaconstitutive} was introduced in 
\cite{LiuCipolattiRinconPalermo2014} and was referred to as \emph{Mooney-Rivlin type isotropic viscoelastic solid}. Notice that the purely elastic part of the constitutive equation is the same as that for Mooney-Rivlin isotropic elastic solids, while the viscous part is linear in the strain-rate tensor $\sfD$ and depends on the Cauchy-Green stress tensor $\sfB$ in a form similar to the purely elastic part.  This constitutive equation was successfully used to model the salt migration and the associated formation of salt diapirs (or salt domes) in a two-layer sediment-salt body problem, in which a denser elastic sediment material lies on top of a lighter layer of viscoelastic salt rock. The influence of the temperature in the formation of these salt domes was then investigated in \cite{teixeiraliurinconcipolattipalermo2014}. For other models and more information on this important geophysical problem, see the literature referenced in these two articles.

The associated evolutionary system of equations are written in the form
\begin{equation}
  \label{LiueulerB}
    \begin{cases}
      \sfB_t + (\bv\cdot \bnabla) \sfB = (\bnabla \bv) \sfB + ((\bnabla \bv)\sfB)^\tr, \\
      \rho(\bv_t + (\bv\cdot \bnabla) \bv) + \bnabla p =  \nabla \cdot \sfT(\sfB,\sfD(\bv)) - \rho g \be_3, \\
      \theta_t + (\bv\cdot\bnabla) \theta = \bnabla \cdot (\sfQ(\sfB) \bnabla \theta),  \\
      \nabla \cdot \bv = 0, 
    \end{cases}
\end{equation}
on the domain $\Omega$ and time interval $[0,T]$, with $T>0$, along with appropriate boundary conditions on $\partial \Omega$ and initial conditions at $t=0$. In the system \eqref{LiueulerB}, $g$ is the acceleration of gravity, $\be_3=(0,0,1)$ is the vertical direction, the viscoelastic stress tensor $\sfT(\sfB,\sfD)$ is given according to $\sfT = \sfsigma + p\sfI$, where $\sfsigma$ is as is \eqref{sigmaconstitutive}, the thermal conductivity tensor is given by $\sfQ(\sfB) = \kappa_1 \sfI + \kappa_2 \sfB + \kappa_3 \sfB^{-1}$, and the viscoelastic parameters $s_1$, $s_{-1}$, $\mu_1$, $\mu_2$ and $\mu_3$ and the thermal conductivity coefficients $\kappa_1, \kappa_2$, and $\kappa_3$ may depend on invariants of the left Cauchy-Green tensor $\sfB$ and on the temperature field $\theta$.

Notice that in the formulation \eqref{LiueulerB} the left Cauchy-Green stress tensor $\sfB$ becomes an unknown. This is the trick to write an Eulerian formulation for the model, instead of the classical Lagrangian formulation. 

The formulation \eqref{LiueulerB} is inspired by a similar formulation given in \cite{linliuzhang2005, linzhang2008} for a particular type of Oldroyd-B viscoelastic model (see also \cite{guillopesaut1990, lionsmasmoudi2000, cheminmasmoudi2001, liuwalkington2001, sideristhomases2005} for Eulerian formulations of more general types of Oldroyd-B models). In the model considered in \cite{linliuzhang2005, linzhang2008} the stress tensor is of the form $\sigma = -p\sfI + 2\mu\sfD + s\sfF\sfF^\tr$, so that the linear part is simply the classical Stokes problem, and the unknown is actually the deformation tensor $\sfF$. Our formulation, besides including a more complicated constitutive law, uses instead $\sfB$ as an unknown, which seems to have some computational advantages. The Lagrangian position $\bx(t,\bX)$ at time $t$ of a particle which initially was at $\bx(0,\bX) =\bX=(X_1,X_2,X_3)$ can be recovered from the vector field $\bv$, and then it can be showed that $\sfB$ is precisely $\sfF\sfF^\tr$, for $\sfF=\sfF(t,\bx)$ given according to $\sfF(t,\bx(t,\bX)) = D_\bX \bx(t,\bX)=(\partial x_i/\partial_{X_j}(t,\bX))_{ij}$ (see \cite[Lemma 1.1]{linliuzhang2005} for the corresponding result in the context of the model in \cite{linliuzhang2005}, in which the unkown is $\sfF$ itself). 

The details for the initial-value problem \eqref{LiueulerB} will be given in a forthcoming paper. The current work is an important step towards that, since proving existence, uniqueness, and regularity properties of the viscoelastic Stokes problem \eqref{viscoelasticStokesorg} is critical for proving local well-posedness of the evolutionary problem \eqref{LiueulerB}. In particular, we will exploit the explicit estimates obtained here to derived suitable a~priori estimates for the solutions of the evolutionary system.

The model of salt dynamics that we are interested in corresponds to particular combinations of the parameters $\mu_1, \mu_2$, and $\mu_3$, which may be different for different types of salts (see \cite{LiuCipolattiRinconPalermo2014} and \cite{teixeiraliurinconcipolattipalermo2014}). Nevertheless, we analyse the mathematical properties of the viscoelastic Stokes problem \eqref{viscoelasticStokesorg} for more general choices of the parameters, for possibly other types of materials.

In the Eulerian formulation \eqref{LiueulerB}, since the left Cauchy-Green tensor $\sfB$ is taken to be an unknown variable of the system, it is also of fundamental importance for us to derive the explicit dependence of the Sobolev estimates for the solution $\bv$ of problem \eqref{viscoelasticStokesorg} in terms not only on $\bbf$ but also on $\sfB$. This is part of the difficulty of the problem.

Regarding the notation, the operator $\bnabla = (\partial_{x_i})_i$ acts to yield the vector field $\bnabla p = (\partial_{x_i}p)_i$, the tensor $(\bnabla \bv)_{ij}=(\partial_{x_j}v_i)_{ij}$, the scalar field $\bnabla \cdot \bv = \sum_i \partial_{x_i} v_i$, and the vector field $\bnabla\cdot \sfA = (\sum_j\partial_{x_j}a_{ij})_i$, when $p$ is a scalar field, $\bv=(v_i)_i$ is a vector field, and $\sfA = (a_{ij})_{ij}$ is a tensor.  The Laplace operator can be written as $\Delta = \bnabla \cdot \bnabla$. For vectors $\bu, \bv$ and matrices $\sfA, \sfB$, we write $\bu\cdot\bv = \sum_i u_iv_i$, $\bu\cdot \sfA = (\sum_j u_j a_{ij})_{i}$, and $\sfA : \sfB = \sum_{ij} a_{ij}b_{ij}$, as well as $\sfA\sfB = (a_{ik}b_{kj})_{ij}$ and $\bu\otimes \bv = (u_iv_j)_{ij}$.

Now, concerning the system \eqref{viscoelasticStokesorg}, we address it in the following way. First, notice that, due to the divergence-free condition, we have
\[ \bnabla \cdot (\bnabla \bv)^\tr = \left(\sum_i \partial_{x_i} \partial_{x_j} v_i \right)_j = \bnabla (\bnabla \cdot \bv) = 0,
\]
so that
\[ \Delta \bv = \bnabla \cdot \bnabla \bv = \bnabla \cdot ((\bnabla \bv) + (\bnabla \bv)^\tr)  = 2\bnabla \cdot \sfD(\bv) = \bnabla \cdot (\sfD(\bv) \sfI + \sfI \sfD(\bv)).
\]
Hence, we write
\begin{multline*}  \mu_1 \Delta \bv + \mu_2 \bnabla \cdot (\sfD(\bv)\sfB + \sfB\sfD(\bv)) +\mu_ 3\bnabla \cdot (\sfD(\bv)\sfB^{-1} + \sfB^{-1} \sfD(\bv)) \\
  = \bnabla \cdot (\sfD(\bv)(\mu_1\sfI + \mu_2 \sfB + \mu_3 \sfB^{-1}) + (\mu_1\sfI + \mu_2 \sfB + \mu_3 \sfB^{-1})\sfD(\bv)).
\end{multline*}
In order to simplify the notation and also the analysis of the problem, we introduce the operator
\begin{equation}
  \label{defAcal}
  \Acal(\sfB) = \mu_1 \sfI + \mu_2 \sfB + \mu_3 \sfB^{-1},
\end{equation} 
and we write
\[  \mu_1 \Delta \bv + \mu_2 \bnabla \cdot (\sfD(\bv)\sfB + \sfB\sfD(\bv)) +\mu_ 3\bnabla \cdot (\sfD(\bv)\sfB^{-1} + \sfB^{-1} \sfD(\bv))
  = \sfD(\bv)\Acal(\sfB) + \Acal(\sfB)\sfD(\bv).
\]

This allows us to rewrite the viscoelastic Stokes problem \eqref{viscoelasticStokesorg} in the form
\begin{equation}
  \label{viscoelasticStokes}
   \begin{cases}
      - \bnabla \cdot (\sfD(\bv)\Acal(\sfB)+ \Acal(\sfB)\sfD(\bv)) + \bnabla p = \bbf, & \text{in } \Omega,\\
      \bnabla \cdot \bv = 0, & \text{in } \Omega,\\
      \bv = 0, & \text{on } \partial \Omega.
    \end{cases}
\end{equation}
We then separate the analysis of problem \eqref{viscoelasticStokes} into two parts. First we consider the generalized Stokes problem
\begin{equation}
  \label{generalizedStokes}
   \begin{cases}
      - \bnabla \cdot (\sfD(\bv)\sfA + \sfA\sfD(\bv)) + \bnabla p = \bbf, & \text{in } \Omega,\\
      \bnabla \cdot \bv = 0, & \text{in } \Omega,\\
      \bv = 0, & \text{on } \partial \Omega.
    \end{cases}
\end{equation}
for a given tensor $\sfA$, and investigate the existence, uniqueness and regularity of solutions of problem \eqref{generalizedStokes} in the case that $\sfA$ is an essentially bounded and uniformly positive definite symmetric tensor. 

The existence and uniqueness of solutions of the weak form of \eqref{generalizedStokes} together with the regularity $\bv\in H_0^1(\Omega)$, $p\in L^2(\Omega)$ follows from a classical application of Lax-Milgram Theorem, once the ellipticity is guaranteed (see Theorem \ref{weaksolutionweakgeneralizedStokeswithpressure}). The regularity $\bv\in H^2(\Omega)$, $p\in H^1(\Omega)$ follows by applying a result of Mach\'a \cite[Theorem 4.2 with $k=0$]{macha2011} for a slightly more general system, once all the conditions are verified, and along with that we obtain more explicit estimates in terms of all the data of the problem (see Theorem \ref{h2regularity}). For higher-order regularity, the case $k\geq 1$ in \cite[Theorem 4.2]{macha2011} is not optimal, and neither is any other result that we are aware of (the main reason is that the conditions on $\sfA$ are imposed independently of the space dimension and the Sobolev embeddings for particular dimensions are not explored; see Remark \ref{dontneedsomuchregularityonA}). We, therefore, work out the details of the proof for higher-order regularity and obtain the desired result under less stringent regularity assumptions on $\sfA$, along with the associated explicit estimates (see Theorem \ref{hkplustworegularity}).

Once the generalized Stokes problem \eqref{generalizedStokes} has been understood, we study the transformation \eqref{defAcal}, i.e. $\sfB \longmapsto \Acal(\sfB)$, and investigate the conditions on $\sfB$ and on $\mu_1, \mu_2, \mu_3$ that yield the proper conditions on $\Acal(\sfB)$ to ensure the regularity and the uniform ellipticity of \eqref{generalizedStokes} with $\sfA = \Acal(\sfB)$. More precisely, we give conditions on the parameters $\mu_1$, $\mu_2$, $\mu_3$ and $\sfB$ that guarantee that $\Acal(\sfB)$ is uniformly positive definite (Propositions \ref{propessinfglambdacondition} and \ref{propmuscenarios}) and sufficiently regular (Propositions \ref{proplinftyboundonacalb}, \ref{propwoneinftyboundonacalb}, \ref{propw23boundonacalb}, and \ref{prophmboundonacalb}). This yields a condition for \eqref{viscoelasticStokesorg} to be uniformly elliptic and for the existence, uniqueness and regularity results to hold for the viscoelastic Stokes problem as well (Theorems \ref{weaksolutionweakviscoelasticdStokeswithpressure}, \ref{h2regularityforviscoelasticStokes}, and \ref{hkplustworegularityforviscoelasticStokes}).

\section{Function spaces}

In this work, we consider a bounded domain $\Omega\subset \RR^3$, i.e. an open and connected bounded set in $\RR^3$. We denote by $L^p(\Omega)$ the classical Lebesgue spaces with $1\leq p \leq \infty$, with norm $\|\cdot \|_{L^p(\Omega)}$, and by $W^{k,p}(\Omega)$ the classical Sobolev spaces with $k\in \NN$ and $1\leq p \leq \infty$, of functions with derivatives up to order $k$ in $L^p(\Omega)$, with norm $\|\cdot\|_{W^{k,p}(\Omega)}$. In the case $p=2$, we denote $H^k(\Omega) = W^{k,2}(\Omega)$, with norm $\|\cdot\|_{H^{k}(\Omega)}$. 

We consider the three-dimensional vector-valued versions $L^p(\Omega)^3$, $W^{k,p}(\Omega)^3$, and $H^k(\Omega)^3$ of those spaces, although for notational simplicity we keep the notation for the norms and the inner products the same. For instance, the $L^p$ norm, $1\leq p < \infty$, of a vector field is such that
\[ \|\bv\|_{L^p(\Omega)} = \left(\sum_{i=1}^3 \|v_i\|_{L^p(\Omega)}^p\right)^{1/p},
\]
while the $W^{k,p}$ norm, $k\in \NN$, $1\leq p \leq \infty$ is such that
\[ \|\bv\|_{W^{k,p}(\Omega)} = \left( \sum_{i=1}^3 \|v_i\|_{W^{k,p}(\Omega)}^p \right)^{1/p}.
\]
The vector valued version of the space of test functions and distributions and their duality product is also considered. Similarly, we consider the tensor-valued version of the spaces $L^p(\Omega)$, $W^{k,p}(\Omega)$, and $H^k(\Omega)$, but keep the notation for the norms and inner products the same as those for scalars and vector fields.

We identify the dual of $L^2(\Omega)$ with itself and denote the dual space of $H_0^1(\Omega)$ by $H^{-1}(\Omega)$. For the sake of notational simplicity, we write the duality product as an integral even when the functional does not belong to $L^2(\Omega)$, i.e.
\[ \dual{f,\varphi}_{H^{-1}(\Omega),H_0^1(\Omega)} = \int_\Omega f \varphi \;\rd\bx.
\]

For higher-order derivatives, we consider a multi-index $\bgamma = (\gamma_1, \gamma_2, \gamma_3)\in \NN_0^3$, where each $\gamma_i\in \NN_0$ is a non-negative integer. The order of the multi-index is the number $|\bgamma| = \gamma_1 + \gamma_2 + \gamma_3$. The associated derivative $D^\bgamma$ of a scalar function $\varphi$ is 
\[ D^\bgamma \varphi = \partial_{x_1}^{\gamma_1}\partial_{x_2}^{\gamma_2}\partial_{x_3}^{\gamma_3}\varphi = \frac{\partial^{|\bgamma|}\varphi}{\partial_{x_1}^{\gamma_1}\partial_{x_2}^{\gamma_2}\partial_{x_3}^{\gamma_3}}. 
\]
Similarly for the $D^\bgamma$ derivative of vector fields and tensors. Notice that we can write $\bgamma = \gamma_1\be_1 + \gamma_2\be_2 + \gamma_3\be_3$, where $\be_1=(1,0,0)$, $\be_2=(0,1,0)$, and $\be_3=(0,0,1)$ form the canonical basis of $\RR^3$ and can also be regarded as the multi-indices of the first-order derivatives. Combining all the derivatives of a certain order $k$ yields the tensor 
\[  D^k\varphi = (D^\bgamma \varphi)_{|\bgamma| = k}, 
\]
where the index $\bgamma$ runs through all the $\NN_0^3$ indices with order $|\bgamma|=k$.

For dimensional reasons, we define the equivalent norms
\begin{equation}
  \label{dimensionalhigherordernorm}
  \| \varphi\|_{W^{k,p}_\lambda(\Omega)} = \sum_{j=0}^k \lambda^{(k-j)/2}\|D^j\varphi\|_{L^p(\Omega)},
\end{equation}
where $\lambda$ is a parameter with the physical dimensional of $(\text{length})^{-2}$ (such as the constant in the Poincar\'e inequality mentioned in \eqref{poincareineq} below). Notice that $\|\varphi\|_{W^{k,p}_\lambda(\Omega)}$ has the same physical dimension (unit) as $\|D^k\varphi\|_{L^p(\Omega)}$. Similarly for the vector-valued and tensor-valued versions of these spaces. In the case $p=2$, we write $W^{k,2}_\lambda(\Omega)=H^k_\lambda(\Omega)$. For negative powers, $H^{-k}(\Omega)$, for $k\in \NN$, is the dual of $H_0^k(\Omega)$, and $H^{-k}_{\lambda}(\Omega)$ is the dual norm associated with the norm $H^k_{\lambda}(\Omega)$ in the subspace $H_0^k(\Omega)$, i.e. 
\[ \| \xi \|_{H^{-k}_{\lambda}(\Omega)} = \sup_{\varphi\in H_0^k(\Omega), \;\|\varphi\|_{H^k_\lambda(\Omega)=1}}\dual{\xi,\varphi}, \quad \forall \xi\in H^{-k}_{\lambda}(\Omega),
\]
where $\dual{\cdot,\cdot}$ is the corresponding duality product.

A useful inequality is the following interpolation inequality, balancing different orders of derivatives:
\begin{equation}
  \label{interpolationineq}
  \|\varphi\|_{H^m_\lambda(\Omega)} \leq c(\Omega)\|\varphi\|_{H^j_\lambda(\Omega)}^\theta\|\varphi\|_{H^k_\lambda(\Omega)}^{(1-\theta)},
\end{equation}
for any $\varphi\in H^k_\lambda(\Omega)$ and for any $j\leq m\leq k$, where $c(\Omega)$ is a non-dimensional constant that only depends on the shape of $\Omega$, and $\theta\in ]0,1[$ is given by
\begin{equation}
  \label{interpolationtheta}
  m = \theta j + (1-\theta)k.
\end{equation}
The same inequality holds with $\varphi$ replaced by a vector field or a tensor.

Concerning the constant that appears in \eqref{interpolationineq} and in many other inequalities involving functional norms given in this work, we call a ``universal constant'' any term that is independent of any data of the problem, and by a ``shape constant'' any term that may depend only on the shape of $\Omega$, i.e. it is independent of translations, rotations and dilations of the spatial domain. The universal constants will be denoted by $c$, while shape constants will be denoted by $c(\Omega)$. From a physical point of view, they are non-dimensional constants. All the estimates in this work have been properly dimensionalized, in the sense that the constants are non-dimensional, so that all the dimensional quantities appear explicitly and with the proper dimensionalization (e.g. the exact power).

As usual, the pressure term in \eqref{viscoelasticStokes} or in \eqref{generalizedStokes} satisfies an equation of the form $\bnabla p = \bg$ and is, therefore, obtained up to a constant. In this regard, it is useful to consider the quotient space $L^2(\Omega)/\RR$, which is the space $L^2(\Omega)$ quotient equivalence up to translation by a constant. More precisely, two functions $\varphi$ and $\psi$ in $L^2(\Omega)$ are equivalent up to translations by a constant if $\varphi-\psi$ is constant in $\Omega$. In particular, $\varphi$ is in the same equivalent class as its zero mean representative $\varphi - \int_\Omega \varphi$. Such a quotient space is also a Banach space. Given a function $\varphi$ in $L^2(\Omega)$, this function belongs to an equivalence class in $L^2(\Omega)/\RR$ which with a certain abuse of notation we still denote by $\varphi$, for simplicity, and is such that its norm in the equivalence class is given by
\[ \|\varphi\|_{L^2(\Omega)/\RR} = \inf_{c\in \RR} \|\varphi + c\|_{L^2(\Omega)} = \left\|\varphi - \int_\Omega \varphi\right\|_{L^2(\Omega)}.
\]

We denote by $\Ccal^0(\bar\Omega)$ the space of continuous functions defined on $\bar\Omega$ and by $\Ccal^k(\bar\Omega)$, the subspace of $\Ccal^0(\bar\Omega)$ with derivatives up to order $k$ defined and continuous on $\Omega$ and with well defined extensions up to the boundary. We define the space $\Ccal^{k,\alpha}(\bar\Omega)$, $k\in \ZZ$, $k\geq 0$, $0<\alpha\leq 1$, as the subspace of $\Ccal^k(\bar\Omega)$ such that the derivatives of order $k$ are H\"older-continuous function with H\"older exponente $\alpha$, or a Lipschitz function, in the case $\alpha=1$.

We follow \cite[page 77]{adn2} (see also \cite{adamsfournier2003} for related and more general notions) and say that the domain $\Omega$ has a boundary of class $\Ccal^k$, $k\in \ZZ$, $k\geq 0$, when its boundary $\partial \Omega$ can be covered by a finite number of open sets $U_m$, $m=1,\ldots, M$, $M\in \NN$, such that each portion $\overline{U_m}\cap\overline{\Omega}$ of the domain is homeomorphic to a hemisphere, the portion $\overline{U_m}\cap\partial \Omega$ is the image of the flat face of the hemisphere, and the homeomorphism and its inverse are of class $\Ccal^k$, with all derivatives up to order $k$ bounded by a constant independent of $m$. Similarly, we say that $\Omega$ has a boundary of class $\Ccal^{k,\alpha}$, $k\in \ZZ$, $k\geq 0$, $0<\alpha\leq 1$, when $\Omega$ has a boundary of class $\Ccal^k$ and the $k$-th derivatives of the associated homeomorphism and of its inverse are H\"older continuous with exponent $\alpha$ and with modulus of H\"older continuity bounded independently of $m$. In the case $\Ccal^{0,1}$, we say that $\Omega$ has a Lipschitz boundary.

When $\Omega$ has a boundary of class $\Ccal^1$, the space $W^{1,\infty}(\Omega)$ can be identified with the space of Lipschitz functions $\Ccal^{0,1}(\bar\Omega)$, and the spaces $W^{1,p}(\Omega)$, $p>3$, can be identified with the space of H\"older functions $\Ccal^{0,1 - 3/p}(\bar\Omega)$ (see \cite[Theorem 5.6.5, pg 283]{evans2010}).\footnote{If $\varphi\in W^{1,\infty}(\Omega)$, let $\varphi_\varepsilon$ be a mollification of $\varphi$, which converges to $\varphi$ in any $W^{1,p}_\loc(\Omega)$, for $1\leq p < \infty$. Thus, a subsequence of $\varphi_\varepsilon$ converges to $\varphi$ and is such that $\bnabla\varphi_\varepsilon$ converges to $\bnabla\varphi$, almost everywhere in $\Omega$. Since $|\delta_i^h\varphi_\varepsilon|\leq \sup_\Omega|\bnabla\varphi_\varepsilon|$, we find at the limit that $\|\delta_i^h\varphi\|_{L^\infty(\Omega_0)} \leq \|\partial_{x_i}\varphi\|_{L^\infty(\Omega)}$, for every subdomain $\Omega_0$ compactly included in $\Omega$ and for every $h\neq 0$ sufficiently small such that $|h|\leq \textrm{dist}(\Omega_0,\partial\Omega)$.}

For the sake of the boundary conditions, we also consider the Hilbert space $H_0^1(\Omega)$, which is the completion, under the norm $H^1(\Omega)$, of the space of test functions $\Dcal(\Omega)=\Ccal_0^\infty(\Omega)$ of compactly supported, infinitely-many-times continuously differentiable functions on $\Omega$.  The dual $\Dcal(\Omega)'$ is the space of distributions, and the duality product between them is denoted $\dual{\cdot,\cdot}_{\Dcal(\Omega)', \Dcal(\Omega)'}$.

Since $\Omega$ is assumed to be bounded, we have the Poincar\'e inequality
\begin{equation}
  \label{poincareineq}
  \int_\Omega \varphi^2(\bx) \;\rd \bx \leq \frac{1}{\lambda_1} \int_\Omega |\bnabla \varphi(\bx)|^2 \;\rd\bx, \quad \forall \varphi \in H_0^1(\Omega).
\end{equation}
The notation comes from the fact that $\lambda_1$ ends up being the first eigenvalue of the Laplace operator. The Poincar\'e inequality extends, with the same constant, to the vector-valued functions in $H_0^1(\Omega)^3$.

In view of the Poincar\'e inequality \eqref{poincareineq}, the norm in $H_0^1(\Omega)$ is taken to be 
\begin{equation}
  \label{norminH01}
  \|\varphi\|_{H_0^1(\Omega)} = \|\bnabla \varphi\|_{L^2(\Omega)}, 
\end{equation}
and similarly in the vector-valued case $H_0^1(\Omega)^3$. For higher-order derivatives, we have, by induction and using integration by parts and Cauchy-Schwarz inequality, that
\begin{equation}
  \label{higherorderpoincare}
  \|D^j\varphi\|_{L^2(\Omega)}^2 \leq \frac{1}{\lambda_1^{k-j}} \|D^k\varphi\|_{L^2(\Omega)}^2, \quad \forall \varphi\in H_0^1(\Omega)\cap H^k(\Omega),
\end{equation}
for any integers $0\leq j \leq k$. In particular, the $H_{\lambda_1}^k$ norm in the subspace $H_0^1(\Omega)\cap H^k(\Omega)$ is equivalent to the $D^k$-norm, i.e.
\[ \|D^k\varphi\|_{L^2(\Omega)} \leq \|\varphi\|_{H_{\lambda_1}^k(\Omega)} \leq (k+1)\|D^k\varphi\|_{L^2(\Omega)}, \qquad \forall \varphi \in H_0^1(\Omega)\cap H^k(\Omega).
\]

In our case of a three-dimensional space, we have the embeddings $H^2(\Omega)\subset L^\infty(\Omega)$ (see \cite[Theorem 3.9]{agmon1965}) and $H^1(\Omega)\subset L^6(\Omega)$. In the particular case of vanishing boundary conditions, and using the Poincar\'e inequality \eqref{poincareineq}, we write, with the proper dimensionalization,
\begin{equation}
  \label{agmoninequality} 
  \|D^j\varphi\|_{L^\infty} \leq \frac{c_j(\Omega)}{\lambda_1^{1/4}}\|D^{j+2}\varphi\|_{L^2(\Omega)}, \quad \forall \varphi\in H_0^1(\Omega) \cap H^{j+2}(\Omega)
\end{equation}
and
\begin{equation}
  \label{l6inequality} 
  \|D^j\varphi\|_{L^6} \leq c_j(\Omega)\|D^{j+1}\varphi\|_{L^2(\Omega)}, \quad \forall \varphi\in H_0^1(\Omega)\cap H^{j+1}(\Omega),
\end{equation}
for any integer $j\geq 0$, where $c_j(\Omega)$ is a shape constant, which may be different for different $j$. Interpolating between \eqref{higherorderpoincare} and \eqref{l6inequality} we also find the inequality
\begin{equation}
  \label{l3inequality} 
  \|D^j\varphi\|_{L^3} \leq \frac{c_j(\Omega)}{\lambda_1^{1/4}}\|D^{j+1}\varphi\|_{L^2(\Omega)}, \quad \forall \varphi\in H_0^1(\Omega) \cap H^{j+1}(\Omega),
\end{equation}

In the case that $\varphi$ does not vanish on the boundary, a version of \eqref{higherorderpoincare} holds with the full norm,
\begin{equation}
  \label{higherorderscaledembedding}
  \|\varphi\|_{H^j_\lambda(\Omega)}^2 \leq \frac{c_{j,k}(\Omega)}{\lambda^{k-j}}\|\varphi\|_{H^k_\lambda(\Omega)}^2,
\end{equation}
for any $\varphi\in H^k_\lambda(\Omega)$ and any $j\leq k$, where $c_{j,k}(\Omega)$ is a shape constant.

For divergence-free vector fields, we also consider the space
\[ \Vcal(\Omega) =  \{ \bv \in \Ccal_0^\infty(\Omega); \; \bnabla \cdot \bv = 0 \}.
\]
The completion of the space $\Vcal(\Omega)$ under the norm of $L^2(\Omega)$ is denoted by $H(\Omega)$ and the completion of the space $\Vcal(\Omega)$ under the norm of $H_0^1(\Omega)^3$ is denoted by $V(\Omega)$. The norms and inner products in $H(\Omega)$ and in $V(\Omega)$ are those inherited from $L^2(\Omega)^3$ and $H_0^1(\Omega)^3$, respectively. The Poincar\'e inequality also holds in $V(\Omega)$.




If we denote
\begin{equation}
   \sfD(\bv) = (d_{ij})_{ij} = \frac{1}{2}\left( \frac{\partial v_i}{\partial x_j} + \frac{\partial v_j}{\partial x_i}\right)_{ij}, \quad d_{ij}= d_{ji},
\end{equation} 
then, 
\[ \sum_{ij} d_{ij}^2 = \frac{1}{4} \sum_{ij} \left( \frac{\partial v_i}{\partial x_j} + \frac{\partial v_j}{\partial x_i}\right)^2 = \frac{1}{4} \sum_{ij} \left( \left(\frac{\partial v_i}{\partial x_j}\right)^2 + 2 \frac{\partial v_i}{\partial x_j}\frac{\partial v_j}{\partial x_i} + \left(\frac{\partial v_j}{\partial x_i}\right)^2 \right).
\]
Thus, 
\begin{multline*}
   \|\sfD(\bv)\|_{L^2(\Omega)}^2 = \sum_{ij} \int_\Omega d_{ij}^2\;\rd \bx = \frac{1}{4} \sum_{ij} \int_\Omega\left( \left(\frac{\partial v_i}{\partial x_j}\right)^2 + 2 \frac{\partial v_i}{\partial x_j}\frac{\partial v_j}{\partial x_i} + \left(\frac{\partial v_j}{\partial x_i}\right)^2 \right) \;\rd\bx \\
   =  \frac{1}{2}\sum_{ij} \int_\Omega \left(\frac{\partial v_i}{\partial x_j}\right)^2 \;\rd\bx +  \frac{1}{2}\sum_{ij} \int_\Omega \frac{\partial v_i}{\partial x_j}\frac{\partial v_j}{\partial x_i} \;\rd \bx =  \frac{1}{2}\|\bnabla\bv\|_{L^2(\Omega)}^2 + \frac{1}{2}\sum_{ij} \int_\Omega \frac{\partial v_i}{\partial x_j}\frac{\partial v_j}{\partial x_i} \;\rd \bx.
  \end{multline*} 

If $\bv\in \Dcal(\Omega)^3$, then by two integration by parts and using that $\bv$ vanishes on the boundary we find that
\[ \sum_{ij} \int_\Omega \frac{\partial v_i}{\partial x_j}\frac{\partial v_j}{\partial x_i} \;\rd \bx = \sum_{ij} \int_\Omega \frac{\partial v_i}{\partial x_i}\frac{\partial v_j}{\partial x_j} \;\rd \bx = \int_\Omega \left(\bnabla \cdot \bv\right)^2 \;\rd \bx = \|\bnabla \cdot \bv\|_{L^2(\Omega)}^2.
\]
Thus, by the density of $\Dcal(\Omega)^3$ in $H_0^1(\Omega)^3$, we find that (c.f. Korn's inequality \cite[Theorem 6.3-3]{ciarlet1988})
\begin{equation}
  \label{Dandgradvinhzero1}
  \|\sfD(\bv)\|_{L^2(\Omega)}^2 = \frac{1}{2}\|\bnabla\bv\|_{L^2(\Omega)}^2 + \frac{1}{2}\|\bnabla \cdot \bv\|_{L^2(\Omega)}^2, \qquad \forall \bv\in H_0^1(\Omega)^3.
\end{equation}
In particular,
\begin{equation}
  \label{Dandgradvcoerciveinhzero1}
  \frac{1}{2}\|\bnabla\bv\|_{L^2(\Omega)}^2 \leq \|\sfD(\bv)\|_{L^2(\Omega)}^2 \leq \|\bnabla\bv\|_{L^2(\Omega)}^2, \qquad \forall \bv\in H_0^1(\Omega)^3.
\end{equation}

If, moreover, $\bv\in V(\Omega)$, then $\bnabla \cdot \bv = 0$ and we are left with
\begin{equation}
  \label{Dandgradv}
  \|\sfD(\bv)\|_{L^2(\Omega)}^2 = \frac{1}{2}\|\bnabla\bv\|_{L^2(\Omega)}^2, \qquad \forall \bv\in V(\Omega).
\end{equation}

For the proof of higher-order regularity, we use the finite-difference operators
\begin{equation}
  \label{finitedifference}
  \delta_i^h\bw(\bx) = \frac{\bw(\bx + h\be_i) - \bw(\bx)}{h},
\end{equation}
for $i=1,2,3$ and $h>0$. A useful estimate related to these operators is the following:
  \begin{align} 
    \label{finitedifferenceintestimate}
    \|\delta_i^h \varphi\|_{L^2(\Omega_0)} & \leq \|\partial_{x_i} \varphi\|_{L^2(\Omega)}, \quad \forall \varphi\in H^1(\Omega),
  \end{align}
for $h\neq 0$ sufficiently small such that $|h| < \textrm{dist}(\Omega_0, \partial\Omega)$, for any subdomain $\Omega_0$ compactly included in $\Omega$ (see \cite[Lemma 15.1]{friedman1969}). By extending the function $\varphi$ to zero outside $\Omega$, the extended function $\tilde\varphi$ satisfies
  \begin{align} 
    \label{finitedifferenceintestimate2}
    \|\delta_i^h \tilde\varphi\|_{L^2(\Omega)} & \leq \|\partial_{x_i} \varphi\|_{L^2(\Omega)}, \quad \forall \varphi\in H^1(\Omega),
  \end{align}
for any $h\neq 0$. Of course, these estimates extend to vector fields and tensors. 

In the case $\varphi\in W^{1,\infty}(\Omega)$ (so that $\varphi$ is Lipschitz), then it is not difficult to see that $\delta_i^h\varphi$ is uniformly bounded in the interior, i.e.
\begin{equation}
  \label{finitedifferenceintestimateinfty}
  \|\delta_i^h\varphi\|_{L^\infty(\Omega_0)} \leq \|\partial_{x_i}\varphi\|_{L^\infty(\Omega)},
\end{equation}
for any $h\neq 0$ sufficiently small such that $|h| < \textrm{dist}(\Omega_0, \partial\Omega)$, for any subdomain $\Omega_0$ compactly included in $\Omega$. For the estimates near the boundary, we localize and rectify the domain, so that we may assume that a portion of the boundary is locally flat and given by $\{x_3=0\}$. More precisely, assume there is a ``rectified domain'' $\tilde\Omega$ and a ball $B$ such that $B\cap \tilde\Omega = \{\bx\in B; x_3<0\}$ and extend $\varphi$ to a function $\tilde\varphi$ that vanishes on $B\cap (\RR^3\setminus \tilde\Omega)$. Then, we bound any ``tangential'' finite difference by the corresponding ``tangential derivative'', i.e.
\begin{equation}
  \label{finitedifferencebdryestimateinfty}
  \|\delta_i^h\varphi\|_{L^\infty(B_0)} \leq \|\partial_{x_i}\varphi\|_{L^\infty(B\cap\Omega)},
\end{equation}
for every $i=1,2$ and every $h\neq 0$ such that $|h| < \textrm{dist}(B_0, \partial B)$, where $B_0$ is a ball concentric to $B$ and with strictly smaller radius.

\section{Weak formulations}

If $\bv$ and $p$ are smooth solutions of \eqref{generalizedStokes} and $\bw$ is another smooth vector field vanishing at the boundary, and assuming that $\sfA$ and $\bbf$ are sufficiently smooth, then by multiplying scalar-wise the equation by $\bw$ and integrating that by parts, using the boundary conditions, we arrive at
\[ \int_\Omega (\sfD(\bv)\sfA + \sfA\sfD(\bv)) : \bnabla \bw \;\rd \bx - \int_\Omega p \bnabla \cdot \bw \;\rd \bx = \int_\Omega \bbf \cdot \bw \;\rd \bx.
\]
The above equation still makes sense for less regular vector fields and tensors, which leads us to the following weak formulation for the generalized Stokes problem \eqref{generalizedStokes}.

\begin{prob}[Weak formulation for the generalized Stokes problem with pressure]
  Given a tensor $\sfA\in L^\infty(\Omega)^{3\times 3}$ and a vector field $\bbf\in H^{-1}(\Omega)^3$, find a vector field $\bv\in V(\Omega)$ and a scalar field $p\in L^2(\Omega)$ such that
\begin{equation}
    \label{weakgeneralizedStokeswithpressure}
    \int_\Omega (\sfD(\bv)\sfA + \sfA\sfD(\bv)) : \bnabla \bw \;\rd \bx - \int_\Omega p \bnabla \cdot \bw \;\rd \bx = \int_\Omega \bbf \cdot \bw \;\rd \bx, \qquad \forall \bw\in H_0^1(\Omega)^3.
  \end{equation}  
\end{prob}

If we further restrict the test functions $\bw$ to the space $V(\Omega)$ of divergence-free vector fields, then the pressure term disappears and we arrive at the following weak formulation.

\begin{prob}[Weak formulation for the generalized Stokes problem]
  Given a tensor $\sfA\in L^\infty(\Omega)^{3\times 3}$ and a vector field $\bbf\in H^{-1}(\Omega)^3$, find a vector field $\bv\in V(\Omega)$ such that
  \begin{equation}
    \label{weakgeneralizedStokes}
    \int_\Omega (\sfD(\bv)\sfA + \sfA\sfD(\bv)) : \bnabla \bw \;\rd \bx = \int_\Omega \bbf \cdot \bw \;\rd \bx, \qquad \forall \bw\in V(\Omega).
  \end{equation}
\end{prob}

In the case $\sfA$ is replaced by $\Acal(\sfB)$, we have the following weak formulation for the viscoelastic Stokes problem \eqref{viscoelasticStokes} with the pressure term (see Proposition \ref{proplinftyboundonacalb} to justify the hypotheses on $\sfB$).

\begin{prob}[Weak formulation for the viscoelastic Stokes problem with pressure]
  Given a tensor $\sfB\in L^\infty(\Omega)^{3\times 3}$ with $(\det \sfB)^{-1} \in L^\infty(\Omega)$, coefficients $\mu_1, \mu_2, \mu_3$ in $L^\infty(\Omega)$, and a vector field $\bbf\in H^{-1}(\Omega)^3$, find a vector field $\bv\in V(\Omega)$ and a scalar field $p\in L^2(\Omega)$ such that
  \begin{equation}
    \label{weakviscoelasticStokeswithpressure}
    \int_\Omega (\sfD(\bv)\Acal(\sfB) + \Acal(\sfB)\sfD(\bv)) : \bnabla \bw \;\rd \bx - \int_\Omega p \bnabla \cdot \bw \;\rd \bx  = \int_\Omega \bbf \cdot \bw \;\rd \bx, \qquad \forall \bw\in H_0^1(\Omega)^3.
  \end{equation}
\end{prob}

Similarly, we also have the following weak formulation in $V(\Omega)$.

\begin{prob}[Weak formulation for the viscoelastic Stokes problem]
  Given a tensor $\sfB\in L^\infty(\Omega)^{3\times 3}$ with $(\det \sfB)^{-1} \in L^\infty(\Omega)$ and given a vector field $\bbf\in H^{-1}(\Omega)^3$, find a vector field $\bv\in V(\Omega)$ such that
  \begin{equation}
    \label{weakviscoelasticStokes}
    \int_\Omega (\sfD(\bv)\Acal(\sfB) + \Acal(\sfB)\sfD(\bv)) : \bnabla \bw \;\rd \bx = \int_\Omega \bbf \cdot \bw \;\rd \bx, \qquad \forall \bw\in V(\Omega).
  \end{equation}
\end{prob}

Of course, if $\sfB = \sfF\sfF^\tr$, where $\sfF$ is the deformation tensor, then $\det \sfB=1$ thanks to the incompressibility condition. Nevertheless, the previous problems are stated for a general tensor $\sfB$.

In order to characterize the problems above as elliptic problems and to obtain suitable existence, uniqueness and continuous dependence on the data, an important condition is that of uniform ellipticity. In that regard, we make the following definition, suitable in our context.
\begin{defs}
  \label{Stokesuniformlyelliptic}
  We say that, for a given tensor $\sfA$, the generalized Stokes problem \eqref{weakgeneralizedStokes} is \textbf{uniformly elliptic} in $V(\Omega)$, or, more generally, in $H_0^1(\Omega)^3$, if there exist $\delta' \geq \delta >0$ such that
  \begin{multline}
    \label{probgenStokesunifellip}
    \delta \int_\Omega |\bnabla \bv|^2\;\rd\bx \leq \int_\Omega (\sfD(\bv)\sfA + \sfA\sfD(\bv)) : \bnabla \bv \;\rd \bx \\ \leq \delta' \int_\Omega |\bnabla \bv|^2\;\rd\bx, \quad \forall \bv \in V(\Omega), \text{ or, respectively, } \forall \bv \in H_0^1(\Omega)^3.
  \end{multline}  
\end{defs}
\medskip

As we will see in the Section \ref{secgenstokes}, the key condition to show the uniform ellipticity of the generalized Stokes problem is that the tensor $\sfA$ be essentially bounded and uniformly positive definite (see Proposition \ref{probgenStokesisunifelliptic}). With that purpose in mind, we make the following definition.
\begin{defs}
  \label{defuniformpositivedef}
  We say that a symmetric tensor $\sfA$ on $\Omega$ is \textbf{uniformly positive definite} if there exists 
\begin{equation}
  \label{alpha}
  \alpha >0
\end{equation}
such that, for almost every $\bx\in \Omega$,
\begin{equation}
  \label{unifellip}
  \sfA(\bx)\bxi \cdot \bxi \geq \alpha |\bxi|^2, \qquad \forall \bxi \in \RR^3.
\end{equation}
\end{defs}
\medskip

Notice that, if \eqref{unifellip} holds, then, by taking $\bxi$ to be one of the vectors of the canonical basis, we find that $\alpha$ is bounded by every element of the diagonal of the tensor $\sfA$. In particular,
\begin{equation}
  \label{alphaboundedLinfty}
  \alpha \leq \|\sfA\|_{L^\infty}.
\end{equation}

Similarly, concerning the viscoelastic Stokes problem, we have the following definition of uniform ellipticity:
\begin{defs}
  \label{viscoelasticuniformlyelliptic}
  We say that, for a given tensor $\sfB$, the viscoelastic Stokes problem \eqref{weakviscoelasticStokes} is \textbf{uniformly elliptic} in $V(\Omega)$, or, more generally, in $H_0^1(\Omega)^3$, if there exist $\delta'\geq \delta>0$ such that
  \begin{multline}
    \label{probviscoelasticStokesunifellip}
    \delta \int_\Omega |\bnabla \bv|^2\;\rd\bx \leq \int_\Omega (\sfD(\bv)\Acal(\sfB) + \Acal(\sfB)\sfD(\bv)) : \bnabla \bv \;\rd \bx \\
    \leq \delta' \int_\Omega |\bnabla \bv|^2\;\rd\bx, \quad \forall \bv \in V(\Omega), \text{ or, respectively, } \forall \bv \in H_0^1(\Omega)^3.
  \end{multline}  
\end{defs}
\medskip

The condition on $\sfB$ that ensures the uniform ellipticity of the viscoelastic Stokes problem is investigated in Section \ref{conditionsellipticity}. For the moment, we only make the following definition, for the sake of notational simplification.

\begin{defs}
  \label{defproper}
  A positive definite symmetric tensor $\sfB=\sfB(\bx)$ is called $\Acal$-\textbf{positive} when $\sfA=\Acal(\sfB)$ is a uniformly positive definite symmetric tensor in the sense of Definition \ref{defuniformpositivedef}.
\end{defs}
\medskip

Hence, it will follow from Propositions \ref{probgenStokesisunifelliptic} and \ref{proplinftyboundonacalb} that every $\Acal$-positive tensor $\sfB$ in $L^\infty(\Omega)^{3\times 3}$ yields a uniformly elliptic viscoelastic Stokes problem, and the question then is to characterize the tensors $\sfB$ which are $\Acal$-positive. As we just mentioned, this is done in Section \ref{conditionsellipticity}.

\begin{rmk}[Entropy condition]
The ellipticity conditions in Definitions \ref{viscoelasticuniformlyelliptic} and \ref{defproper} are closely related to the entropy principle of the second law of thermodynamics (see \cite{Liu2008}), which concerns the viscous part of the stress tensor and, in our model, reads
\[ (\sfD\Acal(\sfB) + \Acal(\sfB)\sfD) \cdot \sfD \geq 0.
\]
The only difference is that the condition of uniform ellipticity requires a strict inequality, uniformly on the material.
\end{rmk}

Once we find a weak solution, we may recover the pressure using the following classical result (see \cite[Proposition 1.1.1]{temam}).

\begin{prop}
  \label{proprecoverpressure}
  Let $\Omega\subset \RR^d$ be open and $\bg\in \Dcal'(\Omega)^d$, $d\in \NN$. A necessary and sufficient condition for having $\bg = \bnabla p$, in the distribution sense, for some $p\in \Dcal'(\Omega)$ is that 
  \[ \dual{\bg,\bv}_{\Dcal'(\Omega)^d,\Dcal(\Omega)^d} = 0, \quad \forall \bv\in \Vcal(\Omega).
  \]
\end{prop}
\medskip

The regularity of $p$ can be inferred from the regularity of $\bg=\bnabla p$ as follows (see \cite[Proposition 1.1.2]{temam} and \cite{necas}).

\begin{prop}
  \label{regpressurel2}
  Let $\Omega\subset\RR^d$ be an open bounded Lipschitz domain, with $d\in \NN$. Suppose $p\in \Dcal(\Omega)'$ is a distribution with $\bnabla p \in H^{-1}(\Omega)^3$. Then, $p\in L^2(\Omega)$ and
  \begin{equation}
    \label{pressureestl2}
    \|p\|_{L^2(\Omega)/\RR} \leq c(\Omega) \|\bnabla p\|_{H^{-1}(\Omega)},
  \end{equation}
  where $c(\Omega)$ is a shape constant.
\end{prop}
\medskip

Once $\bnabla p$ is more regular, higher regularity for $p$ itself follows directly. In fact, if $\bnabla p$ belongs to $L^2(\Omega)^3$, then clearly $p\in H^1(\Omega)$, and so on for higher regularities of $\bnabla p$.

\section{The generalized Stokes problem}
\label{secgenstokes}

In this section, we study problem \eqref{generalizedStokes} for a given tensor $\sfA$ in $L^\infty(\Omega)^{3\times 3}$, which is assumed to be symmetric and uniformly positive definite in the sense of Definition \ref{defuniformpositivedef}. As we will see later on (Section \ref{subsecunifelliptic}), this ensures that the generalized Stokes problem \eqref{weakgeneralizedStokes} is uniformly elliptic in the sense of Definition \ref{Stokesuniformlyelliptic}. We then use this to obtain the existence and uniqueness of the solutions of this problem (Section \ref{subsecexistuniqgenStokes}) and higher-order regularity (Sections \ref{subsecreggenStokes} and \ref{subsechighreggenStokes}).

\subsection{Uniform ellipticity of the generalized Stokes problem}
\label{subsecunifelliptic}

We start with the following useful lemma.
\begin{lem}
  \label{lempositivityofL}
  Let $\sfA\in L^\infty(\Omega)^{3\times 3}$ be a symmetric tensor on $\Omega$ which is uniformly positive definite and let $\alpha>0$ be the corresponding constant such that \eqref{unifellip} holds. For $\sfM\in L^\infty(\Omega)^{3\times 3}$, consider the operator
  \begin{equation}
    \label{defopL}
    \sfL(\sfM) = \sfS(\sfM)\sfA + \sfA\sfS(\sfM),
  \end{equation} 
  where $\sfS$ is the symmetrization operator
  \begin{equation}
    \label{defopS}
    \sfS(\sfM) = \frac{1}{2}(\sfM + \sfM^\tr).
  \end{equation}
  Then, $\sfL$ satisfies
  \begin{equation}
    \label{symposopL}
    \sfL(\sfM): \sfM \geq 2\alpha \sfS(\sfM):\sfS(\sfM), \quad \forall \sfM\in L^\infty(\Omega)^{3\times 3},
  \end{equation}
  everywhere on $\Omega$. 
\end{lem}

\begin{proof}
  Since $\sfA$ and $\sfS(\sfM)$ are symmetric, so is $\sfL(\sfM)$. Thus, omitting $\bx\in \Omega$ for the sake of simplicity,
  \[ \sfL(\sfM) : \sfM = \frac{1}{2}\sfL(\sfM):\sfM + \frac{1}{2}\sfL(\sfM)^\tr: \sfM = \frac{1}{2}\sfL(\sfM):\sfM + \frac{1}{2}\sfL(\sfM): \sfM^\tr = \sfL(\sfM): \sfS(\sfM).
  \]
  Write $\sfS(\sfM) = (s_{ij})_{ij}$ and $\sfA=(a_{ij})_{ij}$. Then,
\[ \sfL(\sfM) = \sfS(\sfM)\sfA+\sfA\sfS(\sfM) = \left( \sum_l (s_{il}a_{lj} + a_{il}s_{lj})\right)_{ij}.
\]
Using again that $\sfA$ and $\sfS(\sfM)$ are symmetric,
\begin{multline*}
  \sfL(\sfM): \sfM = \sfL(\sfM):\sfS(\sfM) = \sum_{ijl} (s_{il}a_{lj} + a_{il}s_{lj}) s_{ij} \\
  = \sum_{ijl} (a_{lj}s_{il}s_{ij} + a_{li} s_{lj} s_{ij}) = \sum_{ijl} (a_{lj}s_{il}s_{ij} + a_{li} s_{jl} s_{ji}).
\end{multline*}
Switching $i$ with $j$ in the second term in the right hand side above and using the uniform strict positivity of $\sfA$ we find that
\[ \sfL(\sfM): \sfM = 2 \sum_{ijl} a_{lj}s_{il}s_{ij} \geq 2\alpha \sum_{ij} s_{ij}^2 = 2\alpha \sfS(\sfM): \sfS(\sfM),
\]
which completes the proof.
\end{proof}
\medskip

\begin{rmk}
  \label{rmkopLnotpositive}
  Lemma \ref{lempositivityofL} says that $\sfL$ is uniformly strictly positive on the space of essentially bounded symmetric tensors, which we may denote by $L_{\textrm{sym}}^\infty(\Omega)^{3\times 3}$. Indeed, from \eqref{symposopL} and using that $\sfS(\sfM) = \sfM$ for a symmetric tensor, we have
  \[ \sfL(\sfM): \sfM  \geq 2\alpha \sfM: \sfM, \quad \forall \sfM\in L_{\textrm{sym}}^\infty(\Omega)^{3\times 3}.
  \]
  However, $\sfL$ is not strictly positive on all $L^\infty(\Omega)^{3\times 3}$. Indeed, for any $\sfM$ anti-symmetric, we have $\sfS(\sfM) = 0$, so that $\sfL(\sfM) = 0$. In this work, we are actually interested in the case that $\sfM = \bnabla \bv$, which can be written as $\sfM = (\bnabla \otimes \bv)^\tr$. If we extrapolate on this idea and look at tensors of the form $\sfM=\bxi\otimes\bfeta=(\xi_i\eta_j)_{ij},$ for vectors $\bxi$ and $\bfeta$, then we notice that
  \[ \sfS(\sfM) = \sfS(\bxi\otimes\bfeta) = \frac{1}{2}(\bxi\otimes\bfeta + \bfeta\otimes\bxi),
  \]
  and
  \[ \sfS(\sfM):\sfS(\sfM) = \frac{1}{4}(\bxi\otimes\bfeta + \bfeta\otimes\bxi): (\bxi\otimes\bfeta + \bfeta\otimes\bxi)
 = \frac{1}{2}\left(|\bxi|^2|\bfeta|^2 + (\bxi\cdot\bfeta)^2 \right) = \frac{1}{2}\sfM:\sfM + \frac{1}{2}(\tr\sfM)^2.
  \]
  This expression is directly related to the identity \eqref{Dandgradvinhzero1}. In particular,
  \[  \sfS(\sfM):\sfS(\sfM)  \geq \frac{1}{2}\sfM:\sfM.
  \]
  Thus,
  \[ \sfL(\sfM): \sfM \geq \alpha \sfM:\sfM,
  \]
  for any $\sfM=\bxi\otimes\bfeta$, so that $\sfL$ is uniformly strictly positive on the space of tensors of this form.
\end{rmk}

The motivation for considering the operators $\sfL$ and $\sfS$ above follows from the fact that, by taking $\sfM=\bnabla \bv$, we see that
\begin{equation}
  \label{DvasSnablav} 
  \sfD(\bv) = \sfS(\bnabla\bv),
\end{equation}
and
\begin{equation}
  \label{linearpartasL}
  \sfD(\bv)\sfA + \sfA\sfD(\bv) = \sfL(\bnabla \bv),
\end{equation}
so that the generalized Stokes problem \eqref{generalizedStokes} can be written as
\begin{equation}
  \label{generalizedStokeswithL}
   \begin{cases}
      - \bnabla \cdot (\sfL(\bnabla \bv)) + \bnabla p = \bbf, & \text{in } \Omega,\\
      \bnabla \cdot \bv = 0, & \text{in } \Omega,\\
      \bv = 0, & \text{on } \partial \Omega.
    \end{cases}
\end{equation}
and similarly for the elliptic problem \eqref{weakgeneralizedStokes}. With that in mind, we obtain, using Lemma \ref{lempositivityofL}, the uniform ellipticity of the Problem \eqref{weakgeneralizedStokes}.

\begin{prop}
  \label{probgenStokesisunifelliptic}
  If $\sfA\in L^\infty(\Omega)^{3\times 3}$ is a symmetric tensor on $\Omega$ which is uniformly positive definite then the associated problem \eqref{weakgeneralizedStokes} is uniformly elliptic in $H_0^1(\Omega)^3$. More precisely, if \eqref{unifellip} holds for $\sfA$ with $\alpha>0$, then \eqref{probgenStokesunifellip} holds with $\delta=\alpha$ in either $V(\Omega)$ and $H_0^1(\Omega)^3$ and with $\delta' = \sqrt{2}\|\sfA\|_{L^\infty(\Omega)}$, in $V(\Omega)$, and $\delta' = 2\|\sfA\|_{L^\infty(\Omega)}$, in $H_0^1(\Omega)^3$.
\end{prop}

\begin{proof}
   Taking \eqref{linearpartasL} and \eqref{DvasSnablav} into consideration and applying Lemma \ref{lempositivityofL} we find from \eqref{symposopL} that
  \[ (\sfD(\bv) \sfA + \sfA\sfD(\bv)) : \bnabla \bv = \sfL(\bnabla\bv): \bnabla\bv \geq 2\alpha \sfS(\bnabla\bv) : \sfS(\bnabla \bv) = 2\alpha \sfD(\bv): \sfD(\bv).
  \]
  Integrating over $\Omega$ and using the lower bound in \eqref{Dandgradvcoerciveinhzero1} we find that
  \[ \int_\Omega (\sfD(\bv)\sfA + \sfA\sfD(\bv)) : \bnabla \bv \;\rd \bx \geq  2\alpha \|\sfD(\bv)\|_{L^2(\Omega)}^2 \geq \alpha \|\bnabla \bv\|_{L^2(\Omega)}^2,
  \]
for every $\bv$ in $H_0^1(\Omega)^3$. This proves the lower bound in \eqref{probgenStokesunifellip} with $\delta = \alpha$, in $H_0^1(\Omega)^3$, and, in particular, in $V(\Omega)$.

For the upper bound in \eqref{probgenStokesunifellip}, we use H\"older's inequality and the assumption that $\sfA\in L^\infty(\Omega)^{3\times 3}$ to find that
\[ \int_\Omega (\sfD(\bv)\sfA + \sfA\sfD(\bv)) : \bnabla \bv \;\rd \bx 
  \leq 2\|\sfA\|_{L^\infty(\Omega)}\|\sfD(\bv)\|_{L^2(\Omega)}\|\bnabla \bv\|_{L^2(\Omega)},
\]
for all $\bv\in H_0^1(\Omega)^3$. Then, using the upper bound in \eqref{Dandgradvcoerciveinhzero1} for $\sfD(\bv)$ yields the upper bound in \eqref{probgenStokesunifellip}, in $H_0^1(\Omega)^3$, with $\delta'=2\|\sfA\|_{L^\infty(\Omega)}$, while the upper bound in \eqref{Dandgradv} for $\sfD(\bv)$ yields the upper bound in \eqref{probgenStokesunifellip}, in $V(\Omega)$, with $\delta'=\sqrt{2}\|\sfA\|_{L^\infty(\Omega)}$.
\end{proof}
\medskip

\subsection{Existence and uniqueness of solutions of the weak form of the generalized Stokes problem}
\label{subsecexistuniqgenStokes}

Thanks to the ellipticity property proved in Proposition \ref{probgenStokesisunifelliptic}, an application of the Lax-Milgram theorem solves the question of existence and uniqueness of the weak formulation of the generalized Stokes problem.

\begin{prop}
  \label{weaksolutionweakgeneralizedStokes}
  Let $\Omega\subset \RR^3$ be a bounded domain, let $\sfA\in L^\infty(\Omega)^{3\times 3}$ be a uniformly positive definite symmetric tensor on $\Omega$ and let $\bbf\in H^{-1}(\Omega)^3$. Then, there exists a unique $\bv\in V(\Omega)$ that solves the weak formulation \eqref{weakgeneralizedStokes} of the generalized Stokes problem. Moreover, if $\alpha>0$ is such that \eqref{unifellip} holds, then
  \begin{equation}
    \label{weaksolutionweakgeneralizedStokesestimate}
    \|\bnabla \bv\|_{L^2(\Omega)} \leq \frac{1}{\alpha} \|\bbf\|_{H^{-1}(\Omega)}.
  \end{equation}
\end{prop}

\begin{proof}
  Consider the bilinear form $a: V(\Omega)\times V(\Omega) \rightarrow \RR$ given by
  \[ a(\bv,\bw) = \int_\Omega (\sfD(\bv)\sfA + \sfA\sfD(\bv)) : \bnabla \bv \;\rd \bx, \quad \forall \bv, \bw \in V(\Omega),
  \]
  and the linear form $\ell:V(\Omega)\rightarrow \RR$ given by
  \[ \ell(\bw) = \int_\Omega \bbf\cdot \bw \;\rd\bx, \quad \forall \bw\in V(\Omega).
  \]
  Since $\sfA\in L^\infty(\Omega)^{3\times 3}$, it follows that $a(\cdot,\cdot)$ is a well-defined and continuous bilinear form on $V(\Omega)$. Moreover, thanks to Proposition \ref{probgenStokesisunifelliptic}, the estimate \eqref{probgenStokesunifellip} holds in $V(\Omega)$ with $\delta=\alpha$, and $a(\cdot,\cdot)$ is a coercive bilinear form on $V(\Omega)$. On the other hand, since $\bbf\in H^{-1}(\Omega)^3\subset V'(\Omega)$, then $\ell(\cdot)$ is a continuous linear form on $V(\Omega)$. Thus, it follows from the Lax-Milgram Theorem (see e.g. \cite{laxmilgram}) that there exists one, and only one, $\bv\in V(\Omega)$, such that $a(\bv,\bw) = \ell(\bw)$, for all $\bw\in V(\Omega)$, which proves the existence and uniqueness part of the statement. For the estimate, we use the ellipticity condition \eqref{probgenStokesunifellip} to write, with $\bw=\bv$,
  \[ \alpha \|\bnabla \bv\|_{L^2(\Omega)}^2 \leq a(\bv,\bv) = \ell(\bv) = \int_\Omega \bbf\cdot \bv \;\rd\bx \leq \|\bbf\|_{H^{-1}(\Omega)}\|\bv\|_{H_0^1(\Omega)} = \|\bbf\|_{H^{-1}(\Omega)}\|\bnabla \bv\|_{L^2(\Omega)},
  \]
  so that
  \[ \alpha \|\bnabla \bv\|_{L^2(\Omega)} \leq \|\bbf\|_{H^{-1}(\Omega)},
  \]
  which yields \eqref{weaksolutionweakgeneralizedStokesestimate}. This completes the proof.
\end{proof}
\medskip

\begin{rmk}
  \label{rmkweaksolutionweakgeneralizedStokesdualforces}
  In the case $\bbf\in L^2(\Omega)^3$ in Proposition \ref{weaksolutionweakgeneralizedStokes}, the estimate \eqref{weaksolutionweakgeneralizedStokesestimate} implies  $\alpha \|\bnabla \bv\|_{L^2(\Omega)} \leq \lambda_1^{-1/2}\|\bbf\|_{L^2(\Omega)}$, where $\lambda_1$ is the constant in the Poincar\'e inequality \eqref{poincareineq}.
\end{rmk}

Using Proposition \ref{weaksolutionweakgeneralizedStokes}, along with Propositions \ref{proprecoverpressure} and \ref{regpressurel2}, we recover the pressure and obtain the following existence and uniqueness result.
\begin{thm}[Existence and uniqueness]
  \label{weaksolutionweakgeneralizedStokeswithpressure}
  Let $\Omega\subset \RR^3$ be a bounded Lipschitz domain, $\sfA\in L^\infty(\Omega)^{3\times 3}$ be a uniformly positive definite symmetric tensor on $\Omega$, and $\bbf\in H^{-1}(\Omega)^3$. Then, there exist a unique vector field $\bv\in V(\Omega)$ and a scalar field $p\in L^2(\Omega)$ that solve the weak formulation \eqref{weakgeneralizedStokeswithpressure} of the generalized Stokes problem. The field $p$ is unique up to a constant, in $L^2(\Omega)$. Moreover, if $\alpha>0$ is such that \eqref{unifellip} holds, then
  \begin{equation}
    \label{weaksolutionweakgeneralizedStokesestimaterepeat}
    \|\bnabla \bv\|_{L^2(\Omega)} \leq \frac{1}{\alpha} \|\bbf\|_{H^{-1}(\Omega)},
  \end{equation}
  and  
  \begin{equation}
    \label{weaksolutionweakgeneralizedStokesestimatewithpressure}
    \|p\|_{L^2(\Omega)/\RR} \leq \frac{c(\Omega)}{\alpha} \|\sfA\|_{L^\infty(\Omega)}\|\bbf\|_{H^{-1}(\Omega)},
  \end{equation}
  where $c(\Omega)$ is a shape constant.
\end{thm}

\begin{proof}
  From Proposition \ref{weaksolutionweakgeneralizedStokes} we have a unique solution $\bv\in V(\Omega)$ of the generalized Stokes problem \eqref{weakgeneralizedStokes}, which can be written as
  \[ \dual{-\bnabla \cdot (\sfD(\bv)\sfA + \sfA\sfD(\bv)) - \bbf, \bw }_{\Dcal'(\Omega)^3, \Dcal(\Omega)^3} = 0, \qquad \forall \bw\in \Vcal(\Omega).
  \]
  Then, from Propositions \ref{proprecoverpressure} and \ref{regpressurel2}, there is a distribution $p\in L^2(\Omega)/\RR$, which is clearly unique up to a constant, such that
  \[ \bnabla p = -\bnabla \cdot (\sfD(\bv)\sfA + \sfA\sfD(\bv)) - \bbf,
  \]
  with
  \[ \|p\|_{L^2(\Omega)/\RR} \leq c(\Omega) \|\bnabla(\sfD(\bv)\sfA + \sfA\sfD(\bv)) + \bbf\|_{H^{-1}(\Omega)}.
  \]
  Thus,
  \[ \|p\|_{L^2(\Omega)/\RR} \leq c(\Omega) \left(\|\sfD(\bv)\sfA + \sfA\sfD(\bv)\|_{L^2(\Omega)} + \|\bbf\|_{H^{-1}(\Omega)}\right).
  \]
  In three dimensions, we estimate
  \[ \|\sfD(\bv)\sfA\|_{L^2(\Omega)} \leq \sqrt{3} \|\sfA\|_{L^\infty(\Omega)}\|\sfD(\bv)\|_{L^2(\Omega)},
  \]
  and similarly for $\sfA\sfD(\bv)$. Then, using \eqref{Dandgradv}, we find
  \[ \|p\|_{L^2(\Omega)/\RR}  \leq c(\Omega)\left(\sqrt{3} \|\sfA\|_{L^\infty(\Omega)}\|\bnabla \bv\|_{L^2(\Omega)} + \|\bbf\|_{H^{-1}(\Omega)}\right). 
  \]
  Using the previous estimate \eqref{weaksolutionweakgeneralizedStokesestimaterepeat} for $\bnabla \bv$ we obtain
  \[ \|p\|_{L^2(\Omega)/\RR} \leq c(\Omega)\left(\frac{\sqrt{3} \|\sfA\|_{L^\infty(\Omega)}}{\alpha} + 1\right)\|\bbf\|_{H^{-1}(\Omega)}.
  \]  
Using \eqref{alphaboundedLinfty} we finally arrive at \eqref{weaksolutionweakgeneralizedStokesestimatewithpressure}, with another shape constant $c(\Omega)$.
\end{proof}
\medskip

\subsection{Regularity of the solutions of the generalized Stokes problem}
\label{subsecreggenStokes}

As usual, if we have more regularity on the data, then we obtain more regularity on the solution. 

The first regularity result that we consider follows from the paper of Mach\'a \cite[Theorem 4.2]{macha2011}, about a general class of systems that includes ours, or from the result of Giaquinta and Modica \cite[Part II, Theorem 1.2 in page 198 and Remark 1.5 in page 199]{giaquintamodica1982}, about a slightly more general system, and recalling that, with smooth boundaries, we have $W^{1,\infty}=\Ccal^1$. See also the related result of Hyu and Star\'a \cite[Theorem 4.1]{huystara2006}, which unfortunately requires that the operator $\sfL$, defined in \eqref{defopL}, be strictly positive for any $\sfM$ and not only for symmetric tensors $\sfM$ (see Remark \ref{rmkopLnotpositive}), although this is clearly unnecessary as discussed in \cite[Part II, Remark 1.5]{giaquintamodica1982}. The proof of those results is based on a now classical idea of estimating the finite-difference quotients of the solution, considered for instance in \cite{adn2}, and used here in the proof of higher-regularity in Theorem \ref{hkplustworegularity} (see also \cite[Theorem 8.12, page 186]{gilbargtrudinger2001} for a result with similar hypotheses in a classical elliptic problem).

\begin{thm}[Regularity]
  \label{h2regularity}
  Let $\Omega$ be a bounded domain in $\RR^3$ with a boundary of class $\Ccal^2$ and suppose $\sfA\in W^{1,\infty}(\Omega)^{3\times 3}$ is a uniformly positive definite symmetric tensor on $\Omega$ and $\bbf \in L^2(\Omega)^3$. Then, the solution $(\bv, p)$ of the weak formulation \eqref{weakgeneralizedStokeswithpressure} of the generalized Stokes problem is such that $\bv \in V(\Omega)\cap H^2(\Omega)$ and $p\in H^1(\Omega)$, with 
  \begin{equation}
    \label{h2regularityforh2v}
    \|D^2 \bv\|_{L^2(\Omega)} \leq \frac{c(\Omega)}{\alpha}\left( \|\bbf\|_{L^2(\Omega)} + \frac{1}{\alpha}\|\sfA\|_{W^{1,\infty}_{\lambda_1}(\Omega)}\|\bbf\|_{H^{-1}(\Omega)}\right),
  \end{equation}
  and
  \begin{equation}
    \label{h2regularityforh1p}
    \|\bnabla p\|_{L^2(\Omega)} \leq \frac{c(\Omega)}{\alpha}\|\sfA\|_{L^\infty(\Omega)}\left( \|\bbf\|_{L^2(\Omega)} + \frac{1}{\alpha}\|\sfA\|_{W^{1,\infty}_{\lambda_1}(\Omega)}\|\bbf\|_{H^{-1}(\Omega)}\right),
  \end{equation}
  where $\lambda_1$ is the Poincar\'e constant and $c(\Omega)$ is a shape constant.  Moreover, $(\bv,p)$ solves the generalized Stokes problem \eqref{generalizedStokes} almost everywhere in $\Omega$, with the first two equations in \eqref{generalizedStokes} holding in $L^2(\Omega)^3$ and $H^1(\Omega)^3$, respectively, and with the last equation holding for the trace of $\bv$ on $\partial\Omega$ in the space $H^{3/2}(\partial \Omega)$.
\end{thm}

\begin{proof} 
  We apply the result \cite[Theorem 4.2 with $k=0$]{macha2011} to obtain the desired regularity. In \cite{macha2011}, the system considered is of the form \eqref{generalizedStokeswithL} (even more general, allowing the divergence to be any given function in $H^1(\Omega)$). It is assumed that $\sfL$ is weakly coercive in the sense defined in \cite{macha2011}, which is precisely the lower bound in \eqref{probgenStokesunifellip}, which, as proved in Proposition \ref{probgenStokesisunifelliptic}, holds under our hypothesis that $\sfA$ is a uniformly positive definite symmetric tensor. It is also assumed that the boundary of $\Omega$ is of class $\Ccal^2$ and that $\bbf \in L^2(\Omega)^3$. Finally, it is assumed that $\sfL\in W^{1,\infty}(\Omega)^{3\times 3}$, which follows immediately from the assumption that $\sfA\in W^{1,\infty}(\Omega)^{3\times 3}$ and the fact, which is clear from \eqref{defopL}, that $\sfL$ has the same regularity as $\sfA$. Hence, we find that $\bv\in H^2(\Omega)^3$ and $p\in H^1(\Omega)$. In particular, we deduce, from the weak formulation, that the strong form \eqref{generalizedStokes} holds almost everywhere, with the regularities as we described in the statement of this theorem.
  
  Now that we know that the solution is more regular, we address the Sobolev estimates for the solution in terms of the data.

  Since we now know that $\bv$ is in $H^2(\Omega)^3$, we have that $\bnabla \bv$ is in $H^1(\Omega)$ and hence it has sufficient regularity to be a test function in the weak formulation, but it lacks the condition of vanishing on the boundary. Nevertheless, since $\bv$ vanishes on the boundary, the tangential derivatives of $\bv$ along the boundary also vanish, i.e. $\bnabla\bv(\bx)\cdot \bxi = 0$, for every $\bxi$ tangent to $\Omega$ at $\bx\in \partial\Omega$. This will be sufficient for our purposes.  
  
  The idea is to look at the equations for the partial derivatives of $\bv$. By taking the distributional derivative of \eqref{generalizedStokes} with respect to $x_i$, $i=1,2,3$, we find the equations
\begin{equation}
  \label{generalizedStokesforderivative}
   \begin{cases}
      - \bnabla \cdot (\sfD(\partial_{x_i}\bv)\sfA + \sfA\sfD(\partial_{x_i}\bv)) -\bnabla \cdot (\sfD(\bv)(\partial_{x_i}\sfA) + (\partial_{x_i}\sfA)\sfD(\bv)) \\
      \qquad \qquad \qquad + \bnabla (\partial_{x_i}p) =  \partial_{x_i}\bbf, & \text{in } \Omega,\\
      \bnabla \cdot (\partial_{x_i}\bv) = 0, & \text{in } \Omega,
    \end{cases}
\end{equation}
valid in the distribution sense on $\Omega$. Since we already know that $\bv \in V(\Omega)\cap H^2(\Omega)$ and $p\in H^1(\Omega)$, then the first equation in \eqref{generalizedStokesforderivative} actually holds in $H^{-1}(\Omega)^3$, while the second one holds in $L^2(\Omega)$. The system \eqref{generalizedStokesforderivative} can be written as
\begin{equation}
  \label{generalizedStokesforderivativealt}
   \begin{cases}
      - \bnabla \cdot (\sfD(\bv_i)\sfA + \sfA\sfD(\bv_i)) - \bb_i + \bnabla p_i = \bbf_i, & \text{in } \Omega,\\
      \bnabla \cdot \bv_i= 0, & \text{in } \Omega,
    \end{cases}
\end{equation}
where
\[ \bv_i = \partial_{x_i}\bv, \qquad p_i = \partial_{x_i}p, \qquad \bbf_i = \partial_{x_i}\bbf, \qquad \bb_i = \bB(\partial_{x_i}\sfA,\bv),
\]
with
\begin{equation}
  \label{defbB}
  \bB(\sfM,\bw) =  \bnabla \cdot (\sfD(\bw)\sfM + \sfM\sfD(\bw)),
\end{equation}
for any tensor $\sfM$ and any vector field $\bw$.

Equation \eqref{generalizedStokesforderivativealt} has the same form of \eqref{generalizedStokes}, if viewing the additional lower order term $\bb_i$ as an extra forcing term. Thus, if we had the same no-slip boundary conditions for the derivatives $\bv_i = \partial_{x_i}\bv$, Theorem \ref{weaksolutionweakgeneralizedStokeswithpressure} would apply to \eqref{generalizedStokesforderivativealt} with the right hand side as $\bbf_i + \bb_i$, and estimates \eqref{weaksolutionweakgeneralizedStokesestimaterepeat} and \eqref{weaksolutionweakgeneralizedStokesestimatewithpressure} would yield estimates for $\|\bnabla\bv_i\|_{L^2(\Omega)}$ and $\|p_i\|_{L^2(\Omega)/\RR}$ in terms of $\alpha$, $\|\bbf_i + \bb_i\|_{H^{-1}(\Omega)}$ and $\|\sfA\|_{L^\infty(\Omega)}$. These estimates, together  with an estimate for $\bb_i = \bB(\partial_{x_i}\sfA,\bv)$ using the bound for $\partial_{x_i}\sfA$ in $L^\infty$ and the estimate \eqref{weaksolutionweakgeneralizedStokesestimaterepeat} for $\bnabla \bv$, would yield estimates similar to \eqref{h2regularityforh2v} and \eqref{h2regularityforh1p}, with the differences that $c$ above is independent of $\Omega$ and only the higher order term $\bnabla \sfA$ of $\sfA$ would appear. 

However, of course, $\bv_i = \partial_{x_i}\bv$ does not vanish on the boundary, so the above argument does not apply. Nevertheless, the essence of the argument can be recovered by localizing the equations to the interior and to the boundary of the domain and using that the derivatives of $\bv$ tangential to the boundary vanish. At the end, this will give us \eqref{h2regularityforh2v} and \eqref{h2regularityforh1p} with the lower order terms of $\sfA$ and the multiplicative shape constant $c(\Omega)$.
  
More precisely, we use a partition of unit based on a finite covering of the boundary by open sets $U_m$, $m=1, \ldots, M$, for some $M\in\NN$, and a complementing open set $U_0$ compactly included in $\Omega$, so that $\Omega\subset\bigcup_{m=0}^M U_m.$ We write $\bv_i=\sum_m \eta_m\bv_i$, where each $\eta_m$ is supported on $U_m$, with $0\leq \eta_m\leq 1$, and $\sum_m\eta_m = 1$ on $\Omega$. We then estimate each portion $\eta_m\bv_i$. 
  
The idea of the proof is that, for the component $m=0$, $\eta_0$ is supported on the interior of the domain, so that $\eta_0\bv$ vanishes on a neighborhood of the boundary, and so does $\eta_0^2\partial_i\bv$, for every $i=1,2,3$. Moreover, since we already know that $\bv\in H^2(\Omega)^3$, then $\eta_0^2\partial_i\bv$ belongs to $H_0^1(\Omega)^3$. Thus, the estimate for $\eta_0\partial_{x_i}\bv$ is obtained directly from the equation \eqref{generalizedStokesforderivative}, by testing it with $\eta_0^2\partial_i\bv$. 

For the boundary components $m=1,\ldots, M$, we first rectify the domain to a hemisphere by a $\Ccal^2$ change of variables, transforming the system into a similar system with coefficients with the same regularity as $\bA$, and then use that, in the rectified system, all the derivatives in the directions tangential to the boundary are well defined and do vanish on the boundary, which allows us to use the ellipticity of the equation and the weak formulation of \eqref{generalizedStokesforderivative} tested with a localized version of each tangential derivative to get uniform bounds for the tangential derivatives. Then, we use the differential equation itself to express the derivative in the direction normal to the boundary in terms of the derivatives in the directions tangential to the boundary and other lower order derivatives (just like we write $v_{zz} = - f - v_{xx}-v_{yy}$ over the half-plane domain $z<0$, in the case of the elliptic equation $-\Delta v = f$), hence obtaining an estimate of the same order for the derivative normal to the boundary as well. Let us go into that with more details. And we work with \eqref{generalizedStokesforderivativealt} for notational simplicity. 

For the interior regularity, we test the equation \eqref{generalizedStokesforderivativealt}, valid in $H^{-1}(\Omega)^3$, with the function $\bw = \eta_0^2 \bv_i$, which belongs to $H_0^1(\Omega)^3$, so that
  \begin{multline}
    \label{weakgeneralizedStokeswithpressureinterior} 
    \int_\Omega (\sfD(\bv_i)\sfA + \sfA\sfD(\bv_i)) \cdot \bnabla (\eta_0^2\bv_i) \;\rd \bx  + \int_{\Omega}  (\sfD(\bv)(\partial_{x_i}\sfA) + (\partial_{x_i}\sfA)\sfD(\bv))\cdot \bnabla (\eta_0^2\bv_i) \;\rd \bx \\
    - \int_\Omega  p_i \bnabla \cdot (\eta_0^2\bv_i) \;\rd \bx = \int_\Omega \bbf_i \cdot (\eta_0^2\bv_i) \;\rd \bx.
  \end{multline}

Notice that $\bnabla(\eta_0^2\bv_i) = \eta_0 \bnabla(\eta_0\bv_i) + \eta_0(\bnabla\eta_0)\otimes \bv_i$, and 
$D(\eta_0\bv_i) = \eta_0D(\bv_i) + \sfS((\bnabla\eta_0)\otimes \bv_i)$, so that the first term in \eqref{weakgeneralizedStokeswithpressureinterior} can be written as
\begin{align*}
  \int_\Omega (\sfD(\bv_i)\sfA + \sfA\sfD(\bv_i)) & \cdot \bnabla (\eta_0^2\bv_i) \;\rd \bx \\
    & = \int_\Omega \eta_0(\sfD(\bv_i)\sfA + \sfA\sfD(\bv_i)) \cdot \bnabla (\eta_0\bv_i) \;\rd \bx  \\
  & \qquad + \int_\Omega \eta_0(\sfD(\bv_i)\sfA + \sfA\sfD(\bv_i)) \cdot ((\bnabla\eta_0)\otimes \bv_i) \;\rd \bx \\
  & =  \int_\Omega (\sfD(\eta_0\bv_i)\sfA + \sfA\sfD(\eta_0\bv_i)) \cdot \bnabla (\eta_0\bv_i) \;\rd \bx \\
  & \qquad - \int_\Omega (\sfS((\bnabla\eta_0)\otimes \bv_i)\sfA + \sfA\sfS((\bnabla\eta_0)\otimes \bv_i)) \cdot \bnabla (\eta_0\bv_i) \;\rd \bx  \\
  & \qquad + \int_\Omega (\sfD(\eta_0\bv_i)\sfA + \sfA\sfD(\eta_0\bv_i)) \cdot ((\bnabla\eta_0)\otimes (\eta_0\bv_i)) \;\rd \bx \\
  & \qquad - \int_\Omega (\sfS((\bnabla\eta_0)\otimes \bv_i)\sfA + \sfA\sfS((\bnabla\eta_0)\otimes \bv_i)) \cdot ((\bnabla\eta_0)\otimes (\eta_0\bv_i)) \;\rd \bx.
\end{align*}
Since $p_i = \partial_{x_i}p$ and $\bv_i$ is divergence-free, we rewrite the pressure term as
\[ - \int_\Omega  \partial_{x_i} p \bnabla \cdot (\eta_0^2\bv_i) \;\rd \bx = 2\int_\Omega  p \partial_{x_i} ((\bnabla \eta_0) \cdot (\eta_0\bv_i)) \;\rd \bx.
\]
Using that $\bbf_i=\partial_{x_i}\bbf$, we rewrite the forcing term as
\[ \int_\Omega \bbf_i \cdot (\eta_0^2\bv_i) \;\rd \bx = - \int_\Omega \bbf \cdot \partial_{x_i}(\eta_0^2\bv_i) \;\rd \bx.
\]
Hence, we find that
  \begin{equation}
    \label{weakgeneralizedStokeswithpressureinteriorb} 
    \begin{aligned}
    \int_\Omega & (\sfD(\eta_0\bv_i)\sfA + \sfA\sfD(\eta_0\bv_i)) \cdot \bnabla (\eta_0\bv_i) \;\rd \bx \\
    & = - \int_{U_0}  \bb_\bgamma \cdot (\delta_i^{-h}(\eta_0^2\delta_i^h\bv_\bgamma)) \;\rd \bx \\
    & \qquad + \int_\Omega (\sfS((\bnabla\eta_0)\otimes \bv_i)\sfA + \sfA\sfS((\bnabla\eta_0)\otimes \bv_i)) \cdot \bnabla (\eta_0\bv_i) \;\rd \bx \\
    & \qquad - \int_\Omega (\sfD(\eta_0\bv_i)\sfA + \sfA\sfD(\eta_0\bv_i)) \cdot ((\bnabla\eta_0)\otimes (\eta_0\bv_i)) \;\rd \bx \\
  & \qquad + \int_\Omega (\sfS((\bnabla\eta_0)\otimes \bv_i)\sfA + \sfA\sfS((\bnabla\eta_0)\otimes \bv_i)) \cdot ((\bnabla\eta_0)\otimes (\eta_0\bv_i)) \;\rd \bx \\
  & \qquad - \int_{\Omega}  (\sfD(\bv)(\partial_{x_i}\sfA) + (\partial_{x_i}\sfA)\sfD(\bv))\cdot \bnabla (\eta_0^2\bv_i) \;\rd \bx \\
  & \qquad - 2\int_\Omega  p \partial_{x_i} ((\bnabla \eta_0) \cdot (\eta_0\bv_i)) \;\rd \bx \\
   & \qquad  - \int_\Omega \bbf \cdot \partial_{x_i}(\eta_0^2\bv_i) \;\rd \bx.
  \end{aligned}
  \end{equation}
  
  Now we estimate each term in \eqref{weakgeneralizedStokeswithpressureinteriorb}. Thanks to the uniform ellipticity of \eqref{probgenStokesunifellip} with $\delta=\alpha$, proved in Proposition \ref{probgenStokesisunifelliptic}, the term in the left hand side of \eqref{weakgeneralizedStokeswithpressureinteriorb} is bounded below as follows
  \begin{equation}
    \int_\Omega (\sfD(\eta_0\bv_i)\sfA + \sfA\sfD(\eta_0\bv_i)) \cdot \bnabla (\eta_0\bv_i) \;\rd \bx \geq \alpha \|\bnabla(\eta_0\bv_i)\|_{L^2(\Omega)}^2.
  \end{equation}
  
  The first term in the right hand side of \eqref{weakgeneralizedStokeswithpressureinteriorb} is bounded using H\"older's inequality and the $L^\infty$ bounds on $A$ and on $\bnabla \eta_0$:
  \begin{multline*}
    - \int_\Omega (\sfS((\bnabla\eta_0)\otimes \bv_i)\sfA + \sfA\sfS((\bnabla\eta_0)\otimes \bv_i)) \cdot \bnabla (\eta_0\bv_i) \;\rd \bx \\
    \leq 6\|\sfA\|_{L^\infty(\Omega)}\|\bnabla \eta_0\|_{L^\infty(\Omega)} \|\bnabla\bv\|_{L^2(\Omega)}\|\bnabla (\eta_0\bv_i)\|_{L^2(\Omega)}.
  \end{multline*}
  The second term is estimated using again H\"older's inequality and the $L^\infty$ bounds on $A$ and on $\bnabla \eta_0$, together with the fact that $0\leq \eta_0 \leq 1$:
  \begin{multline*}
     \int_\Omega (\sfD(\eta_0\bv_i)\sfA + \sfA\sfD(\eta_0\bv_i)) \cdot ((\bnabla\eta_0)\otimes (\eta_0\bv_i)) \;\rd \bx  \\
    \leq 6\|\sfA\|_{L^\infty(\Omega)}\|\bnabla \eta_0\|_{L^\infty(\Omega)} \|\bnabla\bv\|_{L^2(\Omega)}\|\bnabla (\eta_0\bv_i)\|_{L^2(\Omega)}.
  \end{multline*}

  The estimate on the third term also depends on the $L^\infty$ bounds on $A$ and on $\bnabla \eta_0$, using again the fact that $0\leq \eta_0 \leq 1$:
  \begin{multline*}
    - \int_\Omega (\sfS((\bnabla\eta_0)\otimes \bv_i)\sfA + \sfA\sfS((\bnabla\eta_0)\otimes \bv_i)) \cdot ((\bnabla\eta_0)\otimes (\eta_0\bv_i)) \;\rd \bx \leq 6\|\sfA\|_{L^\infty(\Omega)}\|\bnabla \eta_0\|_{L^\infty(\Omega)}^2 \|\bnabla\bv\|_{L^2(\Omega)}^2.
  \end{multline*}
  
  The fourth term in the right hand side depends on the $L^\infty$ bound on $\partial_{x_i}\sfA$:
  \begin{multline*} 
    \int_{\Omega}  (\sfD(\bv)(\partial_{x_i}\sfA) + (\partial_{x_i}\sfA)\sfD(\bv))\cdot \bnabla (\eta_0^2\bv_i) \;\rd \bx  \\
     \leq 6\|\partial_{x_i}\sfA\|_{L^\infty(\Omega)}\|\bnabla\bv\|_{L^2(\Omega)}(\|\bnabla\eta_0\|_{L^\infty(\Omega)}\|\bnabla\bv\|_{L^2(\Omega)} + \|\bnabla(\eta_0\bv_i)\|_{L^2(\Omega)}).
  \end{multline*}  
    
  The pressure term is estimated using H\"older's inequality and $L^\infty$ bounds on $\bnabla \eta_0$ and on $D^2\eta_0$, along with the fact that $0\leq \eta_0 \leq 1$:
  \begin{multline*} 
  2\int_\Omega  p \partial_{x_i} ((\bnabla \eta_0) \cdot (\eta_0\bv_i)) \;\rd \bx \leq 2\|p\|_{L^2(\Omega)} \|\bnabla((\bnabla \eta_0) \cdot (\eta_0\bv_i))\|_{L^2(\Omega)} \\
  \leq 2\|p\|_{L^2(\Omega)} \left(\|\bnabla (\bnabla\eta_0)\|_{L^\infty(\Omega)} \|\eta_0\bv_i\|_{L^2(\Omega)} + \|\bnabla\eta_0\|_{L^\infty(\Omega)} \|\bnabla (\eta_0\bv_i)\|_{L^2(\Omega)}\right) \\
  \leq 2\|p\|_{L^2(\Omega)} \left(\|D^2\eta_0\|_{L^\infty(\Omega)} \|\bnabla\bv\|_{L^2(\Omega)} + \|\bnabla\eta_0\|_{L^\infty(\Omega)} \|\bnabla (\eta_0\bv_i)\|_{L^2(\Omega)}\right).
  \end{multline*}
 
 The estimate on the forcing term in the right hand side uses H\"older's inequality and the $L^\infty$ bound on $\bnabla\eta_0$, using again that $0\leq \eta_0\leq 1$:
  \begin{multline*}
    - \int_\Omega \bbf \cdot \partial_{x_i}(\eta_0^2\bv_i) \;\rd \bx \leq \|\bbf\|_{L^2(\Omega)} \|\bnabla(\eta_0^2\bv_i)\|_{L^2(\Omega)} \\
    \leq \|\bbf\|_{L^2(\Omega)}(\|\bnabla\eta_0\|_{L^\infty(\Omega)}\|\eta_0\bv_i\|_{L^2(\Omega)} + \|\eta_0\bnabla(\eta_0\bv_i)\|_{L^2(\Omega)}) \\
    \leq \|\bbf\|_{L^2(\Omega)}(\|\bnabla\eta_0\|_{L^\infty(\Omega)}\|\bnabla\bv\|_{L^2({\Omega})} + \|\bnabla(\eta_0\bv_i)\|_{L^2(\Omega)}).    
  \end{multline*}

  Now, using Young's inequality we put all the estimates together and arrive at 
\begin{multline}
    \label{interiorestimateh2}
    \frac{\alpha}{2} \|\bnabla(\eta_0\bv_i)\|_{L^2(\Omega)}^2 \leq 6\|\sfA\|_{L^\infty(\Omega)}\|\bnabla \eta_0\|_{L^\infty(\Omega)}^2 \|\bnabla \bv\|_{L^2(\Omega)}^2 + \frac{2}{\alpha} 12^2\|\sfA\|_{L^\infty(\Omega)}^2\|\bnabla \eta_0\|_{L^\infty(\Omega)}^2 \|\bnabla \bv\|_{L^2(\Omega)}^2 \\
    + 6\|\partial_{x_i}\sfA\|_{L^\infty(\Omega)} \|\bnabla \eta_0\|_{L^\infty(\Omega)}\|\bnabla \bv\|_{L^2(\Omega)}^2 + \frac{2}{\alpha}6^2\|\partial_{x_i}\sfA\|_{L^\infty(\Omega)}^2 \|\bnabla \bv\|_{L^2(\Omega)}^2 \\
    + 2\|p\|_{L^2(\Omega)} \|\bnabla(\bnabla\eta_0)\|_{L^\infty(\Omega)}\|\bnabla\bv\|_{L^2(\Omega)} + \frac{2}{\alpha}2^2\|p\|_{L^2(\Omega)}^2 \|\bnabla\eta_0\|_{L^\infty(\Omega)}^2 \\
    + \|\bbf\|_{L^2(\Omega)}\|\bnabla\eta_0\|_{L^\infty(\Omega)}\|\bnabla\bv\|_{L^2(\Omega)} + \frac{2}{\alpha}  \|\bbf\|_{L^2(\Omega)}^2.
\end{multline}
This can be simplified to the following form:
\begin{multline}
  \label{interiorestimateh2b}
  \frac{\alpha}{2} \|\bnabla(\eta_0\bv_i)\|_{L^2(\Omega)}^2 \leq \left(\alpha \lambda_1 c_1(\eta_0) + \frac{\lambda_1c_2(\eta_0)}{\alpha}\|\sfA\|_{L^\infty(\Omega)}^2 + \frac{c}{\alpha}\|\partial_{x_i}\sfA\|_{L^\infty(\Omega)}^2\right)\|\bnabla \bv\|_{L^2(\Omega)}^2 \\
  + \frac{c}{\alpha}\|\bbf\|_{L^2(\Omega)}^2 + \frac{\lambda_1 c_1(\eta_0)}{\alpha}\|p\|_{L^2(\Omega)}^2,
\end{multline}
where $c$ is a universal constant and $c_1(\eta_0) = c\lambda_1^{-1}(\|\bnabla \eta_0\|_{L^\infty(\Omega)}^2 + \|\bnabla(\bnabla\eta_0)\|_{L^\infty(\Omega)})$, $c_2(\eta_0) = c\lambda_1^{-1}\|\bnabla \eta_0\|_{L^\infty(\Omega)}^2$. Notice that $c_1(\eta_0)$ and $c_2(\eta_0)$ are non-dimensional and can, in fact, be chosen independently of translations, rotations and dilations of $\Omega$, so they are shape constants.

Using \eqref{alphaboundedLinfty}, the inequality \eqref{interiorestimateh2b} yields
\begin{multline}
  \label{interiorestimateh2c}
  \|\bnabla(\eta_0\bv_i)\|_{L^2(\Omega)}^2 \\ \leq \frac{c(\Omega)}{\alpha^2}\left( \|\bbf\|_{L^2(\Omega)}^2 + \lambda_1\|p\|_{L^2(\Omega)}^2 + \left(\lambda_1\|\sfA\|_{L^\infty(\Omega)}^2 + \|\bnabla\sfA\|_{L^\infty(\Omega)}^2\right)\|\bnabla \bv\|_{L^2(\Omega)}^2\right),
\end{multline}
for $i=1,2,3$, where $c(\Omega)$ is a shape constant.

  As for the estimates near the boundary, they are obtained, as mentioned above, by first rectifying each portion of the covering of the boundary. For each $m=1,\ldots, M$, this amounts to a change of variables $\bvarphi: B \rightarrow \RR^3$ that is a diffeomorphism of class $\Ccal^2$ from a ball $B\subset\RR^3$ onto the open set $U_m$. The transformation $\bx = \bvarphi(\by)$ changes the coefficients of $\sfA$ and of the forcing term $\bbf$ to terms with the same regularity, and bounded by the $W^{1,\infty}(\Omega)^{3\times 3}$ norm of $\sfA$ and the $L^2(\Omega)^3$ norm of $\bbf$, respectively (see \cite[Theorem 13.49]{AB}). Thus,  the weak formulation changes to 
\begin{multline}
    \label{weakgeneralizedStokeswithpressureinhemisphere}
    \int_{B^-} (\sfD(\bv(\bvarphi(\by)))\sfA(\bvarphi(\by)) + \sfA(\bvarphi(\by))\sfD(\bv(\bvarphi(\by)))) : \bnabla \bw(\bvarphi(\by)) \det D\bvarphi(\by)\;\rd \by \\ 
    + \int_{B^-} p(\bvarphi(\by)) (\bnabla \cdot \bw)(\bvarphi(\by)) \det D\bvarphi(\by)\;\rd \by = \int_{B^-} \bbf(\bvarphi(\by)) \cdot \bw(\bvarphi(\by)) \det D\bvarphi(\by)\;\rd \by,
\end{multline}  
for all $\bw\in H_0^1(\Omega)^3$, where $D\bvarphi$ is the differential of $\bvarphi$ and $B^-= B \cap \{y_3<0\} = \bvarphi^{-1}(\Omega\cap U_m)$ is an open hemisphere (the ``southern'' hemisphere) with flat boundary $B\cap \{y_3=0\}$. 

We introduce the rectified variables 
\begin{equation}
  \label{localizedrectifiedvariables}
  \tilde\bv(\by) = \bv(\bvarphi(\by)), \qquad \tilde p(\by) = p(\bvarphi(\by)),
\end{equation}
and use that
\[ \bnabla \tilde\bv(\by) = \bnabla\bv(\bvarphi(\by))D\bvarphi(\by),
\]
and
\[ \sfD(\bv(\bvarphi(\by)) = \sfS(\bnabla\tilde\bv(\by) D\bvarphi(\by)^{-1}),
\]
where $\sfS$ is the symmetrization operator defined in \eqref{defopS}. Similarly for $\tilde p$, $\tilde \bw$, with
\[ (\bnabla \cdot \bw)(\bvarphi(\by)) = D\tilde\bw(\by) : D\bvarphi(\by)^{-tr},
\]
where $D\bvarphi(\by)^{-tr} = (D\bvarphi(\by)^{-1})^\tr$. We also define $\tilde \sfA(\by) = \sfA(\bvarphi(\by))$ and $\tilde\bbf(\by) = \bbf(\bvarphi(\by))$. Then, the weak formulation is rewritten as
\begin{multline}
    \label{weakgeneralizedStokeswithpressureinhemispherenewvar}
    \int_{B^-} (\sfS(\bnabla\tilde\bv(\by) D\bvarphi(\by)^{-1})\tilde \sfA(\by) + \tilde \sfA(\by)\sfS(\bnabla\tilde\bv(\by) D\bvarphi(\by)^{-1}) : (\bnabla \tilde\bw(\by) D\bvarphi(\by)^{-1}) \det D\bvarphi(\by)\;\rd \by \\ 
    + \int_{B^-} \tilde p(\by) (D\tilde\bw(\by) : D\bvarphi(\by)^{-tr}) \det D\bvarphi(\by))\;\rd \by = \int_{B^-}\tilde\bbf(\by) \cdot \tilde\bw(\by) \det D\bvarphi(\by)\;\rd \by,
\end{multline}  
for any ``rectified'' test function $\tilde\bw\in H_0^1(B^-)^3$.  The divergence-free condition on $\bv$ is transformed to the condition
\[ D\tilde\bv(\by) : D\bvarphi(\by)^{-tr} = 0
\]
on $\tilde\bv$.

The ellipticity of the higher order term with respect to the vector field $\bnabla\tilde\bv D\bvarphi^{-1}$ in the problem \eqref{weakgeneralizedStokeswithpressureinhemispherenewvar} follows immediately from the estimate \eqref{symposopL} in Lemma \ref{lempositivityofL} and the fact that $\bvarphi$ is a diffeomorphism. Moreover, since $\bvarphi$ is of class $\Ccal^2$, all the coefficients have the appropriate regularity to yield an estimate for the tangential derivatives $\partial_{y_1}\tilde\bv$ and $\partial_{y_2}\tilde\bv$ of the same form as the interior estimates \eqref{interiorestimateh2c}. Here, we exploit the fact that, since $\tilde\bv$ vanishes on the boundary and the boundary is flat, all the derivatives on the directions tangential to the boundary are well defined and vanish on the boundary as well, so that they can be used as test functions. In particular, this implies that the tangential derivatives also vanish on the boundary, which is expected. The estimates rely on the inequalities \eqref{finitedifferenceintestimate2} and \eqref{finitedifferencebdryestimateinfty} and are of the same form as \eqref{interiorestimateh2c}.

Finally, using \eqref{weakgeneralizedStokeswithpressureinhemispherenewvar}, we isolate the distributional derivatives $\partial_{y_3}\tilde\bv$, in the direction normal to the boundary, in terms of the derivatives $\partial_{y_1}\tilde\bv$ and $\partial_{y_2}\tilde\bv$ and also lower order terms, to obtain an estimate of the form \eqref{interiorestimateh2c} for the derivative in the normal direction too. 

With all the local estimates ready, we use the partition of unity to write
\[D^2 \bv = \bnabla(\bnabla\bv) = \bnabla\left(\sum_{m=0}^M \eta_m \bnabla\bv\right) = \sum_{m=0}^M \bnabla(\eta_m \bnabla\bv),
\]
so that, putting all the estimates together, we arrive at the following global estimate for $D^2\bv$:
\begin{equation}
  \label{estimateh2}
  \|D^2 \bv\|_{L^2(\Omega)}^2 \leq \frac{c(\Omega)}{\alpha^2}\left( \|\bbf\|_{L^2(\Omega)}^2 + \lambda_1\|p\|_{L^2(\Omega)}^2 +\left(\lambda_1\|\sfA\|_{L^\infty(\Omega)}^2 + \|\bnabla\sfA\|_{L^\infty(\Omega)}^2\right)\|\bnabla \bv\|_{L^2(\Omega)}^2\right),
\end{equation}
where $c(\Omega)$ is a shape constant.

Using the estimates \eqref{weaksolutionweakgeneralizedStokesestimaterepeat} and  \eqref{weaksolutionweakgeneralizedStokesestimatewithpressure} for the derivatives of the velocity field and the pressure, we obtain precisely  \eqref{h2regularityforh2v}, for another shape constant $c(\Omega)$.

An estimate for $\bnabla p$ is obtained directly from the equation \eqref{generalizedStokes}: 
\begin{equation}
  \label{basicpestimatefromStokes}
  \|\bnabla p \|_{L^2(\Omega)}  \leq \|\bbf\|_{L^2(\Omega)} + \|\sfA\|_{L^\infty(\Omega)}\|D^2\bv\|_{L^2(\Omega)} + \|\bnabla\sfA\|_{L^\infty(\Omega)}\|\bnabla \bv\|_{L^2(\Omega)}.
\end{equation}  
Using the estimates \eqref{weaksolutionweakgeneralizedStokesestimaterepeat} and \eqref{h2regularityforh2v}, we obtain
\begin{align*}
  \|\bnabla p \|_{L^2(\Omega)} & \leq \|\bbf\|_{L^2(\Omega)}  + \frac{1}{\alpha}\|\bnabla\sfA\|_{L^\infty(\Omega)}\|\bbf\|_{H^{-1}(\Omega)}  \\
    & \qquad + \frac{c(\Omega)}{\alpha}\|\sfA\|_{L^\infty(\Omega)}\left( \|\bbf\|_{L^2(\Omega)} + \frac{1}{\alpha}\left(\lambda_1^{1/2}\|\sfA\|_{L^\infty(\Omega)} + \|\bnabla\sfA\|_{L^\infty(\Omega)}\right)\|\bbf\|_{H^{-1}(\Omega)}\right),
\end{align*}  
Then, using \eqref{alphaboundedLinfty} and replacing $c(\Omega)$ by $c(\Omega)+1$, we find the estimate \eqref{h2regularityforh1p}, which completes the proof.
\end{proof}
\medskip

\begin{rmk}
  The proof of regularity in Theorem \ref{h2regularity} may be done directly, without resorting to \cite[Theorem 4.2 with $k=0$]{macha2011}, by working with the finite differences defined in \eqref{finitedifference}. Instead of taking the differential of the equation \eqref{generalizedStokes} to arrive at \eqref{generalizedStokesforderivative} and test the localized versions of this equation against $\eta_0^2\bv_i$, where $\bv_i = \partial_{x_i}\bv$, we test the localized versions of \eqref{generalizedStokes} against $\delta_i^{-h}(\eta_0^2\delta_i^h\bv)$ and use a change of variables that works like integration by parts. This leads to estimates like \eqref{interiorestimateh2c}, but with $\delta_i^h\bv$ at the left hand side in the place of $\bv_i$. Then, passing to the limit as $h\rightarrow 0$, we obtain the regularity and the corresponding $H^1$-bound for $\partial_{x_i}\bv$, for every $i=1,2,3$. This process is done in more details for the higher-order derivatives obtained in the following section, for which we did not find any regularity result with the desired smoothness conditions on the coefficients of the system.
\end{rmk}

\begin{rmk}
In the derivation of the localized and rectified weak formulation \eqref{weakgeneralizedStokeswithpressureinhemispherenewvar}, since $\bv$ is to be interpreted as a velocity field, say corresponding to the speed of a trajectory $\bx(t)=\bvarphi(\by(t))$, it would be more natural do defined the velocity $\tilde\bv(\by)$ following the idea that 
\[ \bv(t) = \frac{\rd \bx(t)}{\rd t} = D\bvarphi(\by(t))\frac{\rd \by(t)}{\rd t} = D\bvarphi(\by(t))\tilde\bv(t).
\]
This would lead to the rectified velocity field $\tilde\bv(\by) = (D\bvarphi(\by))^{-1}\bv(\bvarphi(\by))$, 
and the localized and rectified weak formulation 
\begin{multline}
    \label{weakgeneralizedStokeswithpressureinhemispherenewvarold}
    \int_{B^-} (D\bvarphi(\by)\sfD_{\by}(\tilde\bv(\by))D\bvarphi(\by)^{-1}\sfA(\bvarphi(\by)) + \sfA(\bvarphi(\by))D\bvarphi(\by)\sfD_{\by}(\tilde\bv(\by))D\bvarphi(\by)^{-1}) \\
      : D\bvarphi(\by)\bnabla_{\by} \tilde\bw(\by)D\bvarphi(\by)^{-1} \det D\bvarphi(\by)\;\rd \by \\ 
    + \int_{B^-} \tilde p(\by) \tr(D\bvarphi(\by)\bnabla_{\by} \tilde\bw(\by)D\bvarphi(\bv)^{-1}) \det D\bvarphi(\by)\;\rd \by = \int_{B^-} \tilde\bbf(\by) \cdot \tilde\bw(\by) \det D\bvarphi(\by)\;\rd \by,
\end{multline}
This works perfectly fine, except that we would need one more degree of regularity from the boundary. For the velocity field $\bv$ to be in $H^2$, we would need $\bvarphi$ of class $\Ccal^3$. In order to avoid this unnecessary extra regularity, we perform the change of variables according \eqref{localizedrectifiedvariables}, as done is most classical works (see e.g. \cite{adn2, giaquintamodica1982, huystara2006, macha2011}).
\end{rmk}

\subsection{Higher-order regularity of the solutions of the generalized Stokes problem}
\label{subsechighreggenStokes}

For higher-order regularity, the result of \cite[Theorem 4.2 with $k>0$]{macha2011} is not optimal, requiring too much regularity from $\sfA$ than we can afford for the well-posedness of the associated evolutionary problem \eqref{LiueulerB}. (The result \cite[Theorem 4.2 with $k>0$]{macha2011} requires $\sfA\in W^{k+1,\infty}(\Omega)$, besides $\bbf \in H^k(\Omega)^3$, to have $\bv\in H^{k+2}(\Omega)^3$ and $p\in H^{k+1}(\Omega)$). Similarly in the works \cite{giaquintamodica1982, huystara2006}.) 

The fundamental work of Agmon, Douglis and Niremberg \cite{adn2} also requires too much regularity from $\sfA$ (see the condition just before Theorem 10.5, in page 78, of \cite{adn2}, in which it is required that $\sfA\in \Ccal^{l-s_i}(\bar\Omega)$, with $l$ equals to our $k$, and where $s_i$ can be chosen, in our case, to be $s_1=s_2=s_3=0$ and $s_4=-1$, so that we need at least some coefficients of $\sfA$ to be $\Ccal^{k+1}$, and others to be $\Ccal^k$, to obtain the desired higher-order regularity for $\bv$ and $p$). 

It is worth mentioning the work of Solonnikov \cite{solonnikov2001}, which treats a general Stokes-type problem that includes ours, but only addresses H\"older-continuity, and the work of Ghidaglia \cite{ghidaglia1984}, which has a more direct approach to the classical Stokes problem, but does not apply to our system. 

For general and similar elliptic systems, but not of Stokes type, the works of Friedman \cite{friedman1969}, Lions and Magenes \cite{lionsmagenes1972} and Gilbarg and Trudinger \cite{gilbargtrudinger2001} have results that also assume more regularity from the coefficients of the linear term. 

Most of the works omit the proof of higher-order regularity, limiting themselves to say that the proof follows as in the $H^2$ regularity. The work \cite{friedman1969}, however, does give some details of the proof, and the reason for the regularity conditions becomes clear. The main reason for the high-order regularity assumption on coefficients of the linear term is that the theorems do not take into full consideration the space dimension of the system and the associated Sobolev embeddings. In our case, we exploit the fact that the space dimension is three in order to obtain a more forgiven condition on $\sfA$. See Remark \ref{dontneedsomuchregularityonA} for more discussion on this issue. With that in mind, we present, in Theorem \ref{hkplustworegularity} below, a result assuming less regularity on the coefficients than the previously mentioned works, and which will be perfect for our needs. 

\begin{thm}[Higher-order regularity]
  \label{hkplustworegularity}
    Let $\Omega$ be a bounded domain in $\RR^3$ with a boundary of class $\Ccal^{k+2}$ and suppose $\sfA\in  W^{1,\infty}(\Omega)^{3\times 3} \cap W^{2,3}(\Omega)^{3\times 3} \cap H^{k+1}(\Omega)^{3\times 3}$ is a uniformly positive definite symmetric tensor on $\Omega$ and $\bbf \in H^k(\Omega)^3$, where $k\in \NN$. Then, the solution $(\bv, p)$ of the weak formulation \eqref{weakgeneralizedStokeswithpressure} of the generalized Stokes problem is such that $\bv \in V(\Omega)\cap H^{k+2}(\Omega)$ and $p\in H^{k+1}(\Omega)$. Moreover, $(\bv,p)$ solves the generalized Stokes problem \eqref{generalizedStokes} almost everywhere in $\Omega$, with the first two equations in \eqref{generalizedStokes} holding in $H^{k-1}(\Omega)^3$ and $H^k(\Omega)$, respectively, and with the last equation holding for the trace of $\bv$ on $\partial\Omega$, in the space $H^{k+1/2}(\partial \Omega)$. Finally, for $k=1$, we have the estimates
   \begin{multline} 
  \label{h3egularityforv}
  \|D^3 \bv\|_{L^2(\Omega)} \leq \frac{c_3(\Omega)}{\alpha} \left(\|\bbf\|_{H^1_{\lambda_1}(\Omega)}  \right. \\
    \left.  + \frac{1}{\alpha}\left(\|\sfA\|_{W^{1,\infty}_{\lambda_1}(\Omega)} + \|D^2\sfA\|_{L^3(\Omega)}\right)\left( \|\bbf\|_{L^2(\Omega)} + \frac{1}{\alpha}\|\sfA\|_{W^{1,\infty}_{\lambda_1}(\Omega)}\|\bbf\|_{H^{-1}(\Omega)}\right)\right),
   \end{multline}    
    and
    \begin{multline}
      \label{h2egularityforp}
      \|D^2p\|_{L^2(\Omega)} \leq \frac{c_3(\Omega)}{\alpha}\|\sfA\|_{L^\infty(\Omega)} \left(\|\bbf\|_{H^1_{\lambda_1}(\Omega)}  \right. \\
     \left. + \frac{1}{\alpha}\left(\|\sfA\|_{W^{1,\infty}_{\lambda_1}(\Omega)} + \|D^2\sfA\|_{L^3(\Omega)}\right)\left( \|\bbf\|_{L^2(\Omega)} + \frac{1}{\alpha}\|\sfA\|_{W^{1,\infty}_{\lambda_1}(\Omega)}\|\bbf\|_{H^{-1}(\Omega)}\right) \right),
    \end{multline}     
 while, for any integer $k\geq 2$,
 \begin{equation} 
  \label{hkplus2regularityforv}
  \|D^{k+2} \bv\|_{L^2(\Omega)} \leq \frac{c_{k+2}(\Omega)}{\alpha}R_k(\alpha, \lambda_1, \sfA, \bbf),
\end{equation}    
    and
    \begin{equation}
      \label{hkplus1regularityforp}
      \|D^{k+1}p\|_{L^2(\Omega)} \leq \frac{c_{k+2}(\Omega)}{\alpha}\|\sfA\|_{L^\infty(\Omega)}R_k(\alpha, \lambda_1, \sfA, \bbf),
    \end{equation}
    where 
    \begin{equation}
      \label{Rkdef}
      \begin{aligned}
      R_k(\alpha, & \lambda_1, \sfA, \bbf) = \|\bbf\|_{H^k_{\lambda_1}(\Omega)} + \sum_{j=2}^{k-1} \left( \sum_{i=1}^{k-j} \frac{1}{\alpha^{(k-j)/i}}\frac{1}{\lambda_1^{(k-j)/4i}}\|\sfA\|_{H^{i+2}_{\lambda_1}(\Omega)}^{(k-j)/i}\right)  \|\bbf\|_{H^j_{\lambda_1}(\Omega)}  \\
   &  + \left( \sum_{i=1}^{k-1} \frac{1}{\alpha^{(k-1)/i}}\frac{1}{\lambda_1^{(k-1)/4i}}\|\sfA\|_{H^{i+2}_{\lambda_1}(\Omega)}^{(k-1)/i}\right)\left( \|\bbf\|_{H^1_{\lambda_1}(\Omega)} \right. \\
  & \quad \quad + \frac{1}{\alpha}\left( \|\sfA\|_{W^{1,\infty}_{\lambda_1}(\Omega)} + \|D^2\sfA\|_{L^3(\Omega)}\right)\left(\|\bbf\|_{L^2(\Omega)} + \frac{1}{\alpha}\|\sfA\|_{W^{1,\infty}_{\lambda_1}(\Omega)} \|\bbf\|_{H^{-1}(\Omega)}\right),
    \end{aligned}
    \end{equation}
    and the $c_{k+2}(\Omega)$, for $k\geq 1$, are non-dimensional shape constants.
\end{thm}

\begin{proof}
  The proof is obtained in a way similar to the proof of the estimates obtained in Theorem \ref{h2regularity} for the regularity of $(\bv,p)$ in $H^2(\Omega)\times H^1(\Omega)$. The difference being that, since we do not have yet the desired regularity, we are not allowed to take the full higher-order derivative in the equation \eqref{generalizedStokes}. Instead we use the classical technique of \cite{adn2} of estimating the finite differences, which are then showed to converge to the desired higher-order derivative, obtaining the desired bounds along the way.
  
  We argue by induction that, for a given $k\in\NN_0$, the regularity has been obtained up to $(\bv,p)\in (V(\Omega)\cap H^{k+1}(\Omega)^3) \times H^k(\Omega)$, and we intend to prove that it holds in the Sobolev spaces of order $k+2$ and $k+1$, respectively. The case $k=0$ has been proved in Theorem \ref{h2regularity}, so we continue here with the induction process for $k\in \NN$. We need to show that $\bv\in H^{k+2}$ and $p\in H^{k+1}$.
  
  By taking the derivative of \eqref{generalizedStokes} up to order $k$, we find the equations
\begin{equation}
  \label{generalizedStokesforderivativegamma}
   \begin{cases}
      - \bnabla \cdot (\sfD(\bv_\bgamma)\sfA + \sfA\sfD(\bv_\bgamma)) - \bb_\bgamma + \bnabla p_\bgamma =  \bbf_\bgamma, & \text{in } \Omega,\\
      \bnabla \cdot \bv_\bgamma = 0, & \text{in } \Omega,
    \end{cases}
\end{equation}
for any multi-index $\bgamma\in \NN_0^3$ with $|\bgamma|=k$, where 
\[ \bv_\bgamma = D^\bgamma\bv, \qquad p_\bgamma = D^\bgamma p, \qquad \bbf_\bgamma = D^\bgamma \bbf,
\]
and, by recursion,
\begin{equation}
  \label{bgammadef}
  \bb_\bgamma = \partial_{x_j}\bb_{\bgamma'} + \bB(\partial_{x_j}\sfA,\bv_{\bgamma'}) = \partial_{x_j}\bb_{\bgamma'} + \bnabla \cdot (\sfD(\bv_{\bgamma'})(\partial_{x_j}\sfA) + (\partial_{x_j}\sfA)\sfD(\bv_{\bgamma'})),
\end{equation}
where $\bgamma'\in \NN_0^3$ is another multi-index, one-degree lower than $\bgamma$, given by $\bgamma = \bgamma' + \be_j$, with the convention that $\bb_{\bgamma'} = 0$, if $|\bgamma'|=0$. Since we are assuming that $(\bv,p)\in V(\Omega)\cap H^{k+1}(\Omega)^3 \times H^k(\Omega)$, the first equation in \eqref{generalizedStokesforderivativegamma} holds in $H^{-1}(\Omega)^3$ and the divergence-free equation in \eqref{generalizedStokesforderivativegamma} holds in $L^2(\Omega)$.
  
If we were allowed to take one more derivative of the equation \eqref{generalizedStokesforderivativegamma}, we would see that
\begin{equation}
  \label{generalizedStokesforhighorderderivative}
   \begin{cases}
      - \bnabla \cdot (\sfD(\partial_{x_i}\bv_\bgamma)\sfA + \sfA\sfD(\partial_{x_i}\bv_\bgamma)) - \bnabla \cdot (\sfD(\bv_\bgamma)(\partial_{x_i}\sfA) + (\partial_{x_i}\sfA)\sfD(\bv_\bgamma)) - \partial_{x_i} \bb_\bgamma \\ 
      \qquad \qquad \qquad + \bnabla (\partial_{x_i}p_\bgamma)  = \partial_{x_i}\bbf_\bgamma, & \text{in } \Omega, \\
      \bnabla \cdot (\partial_{x_i}\bv_\bgamma) = 0, & \text{in } \Omega,
    \end{cases}
\end{equation}  
for any $i=1,2,3$. The equation \eqref{generalizedStokesforhighorderderivative} has the same form of the equation \eqref{generalizedStokesforderivative} used for the $H^2$ estimate of $\bv$, if viewing the additional term $\partial_{x_i}\bb_\bgamma$ as an extra forcing term. Naturally, the $H^{k+2}$ estimate would follow a similar argument, by localization of the equation on the interior and on patches of the boundary and rectification of the boundary. The estimates in the interior for the second-order derivatives of $D^\bgamma\bv$ would be essentially the same as those for the second-order derivatives of $\bv$, by testing the localized equation with $\eta_0\partial_{x_i}\bv_\bgamma$. The estimates on the boundary would change slightly, because $\bv_\bgamma$ does not vanish on the boundary. Nevertheless, any combination of tangential derivatives vanishes on the boundary and that would be sufficient. 

However, we do not have yet sufficient regularity to derive \eqref{generalizedStokesforhighorderderivative}, so we must also prove this regularity.  The idea is to work with \eqref{generalizedStokesforderivativegamma} and use the finite-difference operators defined in \eqref{finitedifference} to obtain $H^1$-estimates for $\eta_m\delta_i^h\bv_\bgamma$, uniformly with respect to $h$, and then, at the limit as $h\rightarrow 0$, arrive at the regularity and an estimate for the partial derivative $\eta_m\partial_{x_i}\bv_\bgamma$ of $\bv_\bgamma$ and for $\eta_m\partial_{x_i}p_\bgamma$. This is done for all the derivatives in the interior and for all the combinations of tangential derivatives near the boundary. The derivatives near the boundary that contain normal directions are obtained directly from the equation, by induction on the degree of normal derivatives (see below). Then, by combining all the partitions, we obtain a global estimate for $\bnabla\bv_\bgamma$ and $\bnabla p_\bgamma$, proving, along the way, that those derivatives have indeed the desired regularity.

More precisely, we consider a covering of $\Omega$ by open sets $\{U_m\}_{m=0}^M$, with $U_0$ in the interior of $\Omega$ and the remaining open sets covering the boundary, as done in the proof of Theorem \ref{h2regularity}, along with a partition of unity by functions $\{\eta_m\}_{m=0}^M$, subordinated to this covering. 

For the interior regularity, we test the weak formulation \eqref{generalizedStokesforderivativegamma} with the function $\bw = \delta_i^{-h}(\eta_0^2 \delta_i^h\bv_\bgamma)$, which does vanish on the boundary, and in fact it vanishes outside $U_0$, so that, after integrating by parts the leading order term and the pressure term, we find
  \begin{multline}
    \label{weakgeneralizedStokeswithpressureinteriorgammadifference} 
    \int_{U_0} (\sfD(\bv_\bgamma)\sfA + \sfA\sfD(\bv_\bgamma)) : \bnabla (\delta_i^{-h}(\eta_0^2\delta_i^h\bv_\bgamma)) \;\rd \bx - \int_{U_0}  \bb_\bgamma \cdot (\delta_i^{-h}(\eta_0^2\delta_i^h\bv_\bgamma)) \;\rd \bx\\
    - \int_{U_0}  p_\bgamma \bnabla \cdot (\delta_i^{-h}(\eta_0^2\delta_i^h\bv_\bgamma)) \;\rd \bx = \int_{U_0} \bbf_\bgamma \cdot (\delta_i^{-h}(\eta_0^2\delta_i^h\bv_\bgamma)) \;\rd \bx.
  \end{multline}
By changing variables we rewrite the first term in the form
\[ \int_{U_0} (\sfD(\bv_\bgamma)\sfA + \sfA\sfD(\bv_\bgamma)) : \bnabla (\delta_i^{-h}(\eta_0^2\delta_i^h\bv_\bgamma)) \;\rd \bx = \int_{U_0} \delta_i^h(\sfD(\bv_\bgamma)\sfA + \sfA\sfD(\bv_\bgamma)) : \bnabla (\eta_0^2\delta_i^h\bv_\bgamma) \;\rd \bx.
\]

The product rule works for difference operators just like it does for derivatives:
  \begin{equation}
    \label{finitedifferenceonlinearpart}
    \delta_i^h(\sfD(\bv_\bgamma)\sfA + \sfA\sfD(\bv_\bgamma))  = \sfD(\delta_i^h\bv_\bgamma)\sfA +\sfA\sfD(\delta_i^h\bv_\bgamma) + \sfD(\bv_\bgamma)(\delta_i^h\sfA) + (\delta_i^h\sfA)\sfD(\bv_\bgamma).
  \end{equation}
Moreover, we can write
\begin{equation}  
    \label{breakingetasquare} 
    \bnabla(\eta_0^2\delta_i^h\bv_\bgamma) = \eta_0(\bnabla\eta_0)\otimes (\delta_i^h\bv_\bgamma) + \eta_0 \bnabla(\eta_0\delta_i^h\bv_\bgamma),
\end{equation}
and
\begin{equation}
    \label{breakingDetadelta}
    \sfD(\eta_0\delta_i^h\bv_\bgamma) = \frac{1}{2}(\bnabla (\eta_0\delta_i^h\bv_\bgamma) + (\bnabla (\eta_0\delta_i^h\bv_\bgamma))^\tr) = \sfS((\bnabla\eta_0) \otimes (\delta_i^h\bv_\bgamma)) + \eta_0\sfD(\delta_i^h\bv_\bgamma),
  \end{equation}
  where $\sfS(\cdot)$ is the symmetrization operator defined in \eqref{defopS}. Thus, we rewrite the leading order term as
  \begin{equation}
    \label{lineartermdeltaihvgamma}
 \begin{aligned}
     \int_{U_0} & \delta_i^h(\sfD(\bv_\bgamma)\sfA + \sfA\sfD(\bv_\bgamma)) : \bnabla (\eta_0^2\delta_i^h\bv_\bgamma) \;\rd \bx \\
       & = \int_{U_0} (\sfD(\delta_i^h\bv_\bgamma)\sfA +\sfA\sfD(\delta_i^h\bv_\bgamma)) : \bnabla (\eta_0^2\delta_i^h\bv_\bgamma) \;\rd \bx \\
       & \qquad + \int_{U_0} (\sfD(\bv_\bgamma)(\delta_i^h\sfA) + (\delta_i^h\sfA)\sfD(\bv_\bgamma)) : \bnabla (\eta_0^2\delta_i^h\bv_\bgamma) \;\rd \bx \\
  & = \int_{U_0} \eta_0(\sfD(\delta_i^h\bv_\bgamma)\sfA +\sfA\sfD(\delta_i^h\bv_\bgamma)) : \left((\bnabla\eta_0)\otimes (\eta_0\delta_i^h\bv_\bgamma) + \bnabla (\eta_0\delta_i^h\bv_\bgamma) \right)\;\rd \bx \\
   & \qquad    + \int_{U_0} (\sfD(\bv_\bgamma)(\delta_i^h\sfA) + (\delta_i^h\sfA)\sfD(\bv_\bgamma)) : \bnabla (\eta_0^2\delta_i^h\bv_\bgamma) \;\rd \bx \\
  & = \int_{U_0} (\sfD(\eta_0\delta_i^h\bv_\bgamma)\sfA +\sfA\sfD(\eta_0\delta_i^h\bv_\bgamma)) : \left((\bnabla\eta_0)\otimes (\eta_0\delta_i^h\bv_\bgamma) + \bnabla (\eta_0\delta_i^h\bv_\bgamma) \right)\;\rd \bx \\
   & \qquad   - \int_{U_0} (\sfS((\bnabla\eta_0) \otimes (\delta_i^h\bv_\bgamma))\sfA +\sfA\sfS((\bnabla\eta_0) \otimes (\delta_i^h\bv_\bgamma))) : \left((\bnabla\eta_0)\otimes (\eta_0\delta_i^h\bv_\bgamma) + \bnabla (\eta_0\delta_i^h\bv_\bgamma) \right)\;\rd \bx \\
   & \qquad  + \int_{U_0} (\sfD(\bv_\bgamma)(\delta_i^h\sfA) + (\delta_i^h\sfA)\sfD(\bv_\bgamma)) : \bnabla (\eta_0^2\delta_i^h\bv_\bgamma) \;\rd \bx.
\end{aligned}
\end{equation}  

Since $\bv_\bgamma$ is divergence free, we rewrite the pressure term in \eqref{weakgeneralizedStokeswithpressureinteriorgammadifference} as
\[ -\int_{U_0}  p_\bgamma \bnabla \cdot (\delta_i^{-h}(\eta_0^2\delta_i^h\bv_\bgamma)) \;\rd \bx = - \int_{U_0}  p_\bgamma (\delta_i^{-h}(2\eta_0(\bnabla\eta_0)\cdot(\delta_i^h\bv_\bgamma))) \;\rd \bx.
\]
Thus, \eqref{weakgeneralizedStokeswithpressureinteriorgammadifference} becomes
\begin{equation}
\label{weakgeneralizedStokeswithpressureinteriorgammadifferenceb}
\begin{aligned}
  \int_{U_0} & (\sfD(\eta_0\delta_i^h\bv_\bgamma)\sfA +\sfA\sfD(\eta_0\delta_i^h\bv_\bgamma)) : \left( \bnabla (\eta_0\delta_i^h\bv_\bgamma) \right)\;\rd \bx \\
  & = \int_{U_0}  \bb_\bgamma \cdot (\delta_i^{-h}(\eta_0^2\delta_i^h\bv_\bgamma)) \;\rd \bx \\
  & \qquad + \int_{U_0} (\sfS((\bnabla\eta_0) \otimes (\delta_i^h\bv_\bgamma))\sfA +\sfA\sfS((\bnabla\eta_0) \otimes (\delta_i^h\bv_\bgamma))) : \left(\bnabla (\eta_0\delta_i^h\bv_\bgamma) \right)\;\rd \bx \\ 
   & \qquad   - \int_{U_0} (\sfD(\eta_0\delta_i^h\bv_\bgamma)\sfA +\sfA\sfD(\eta_0\delta_i^h\bv_\bgamma)) : \left((\bnabla\eta_0)\otimes (\eta_0\delta_i^h\bv_\bgamma) \right)\;\rd \bx  \\
   & \qquad   + \int_{U_0} (\sfS((\bnabla\eta_0) \otimes (\delta_i^h\bv_\bgamma))\sfA +\sfA\sfS((\bnabla\eta_0) \otimes (\delta_i^h\bv_\bgamma))) : \left((\bnabla\eta_0)\otimes (\eta_0\delta_i^h\bv_\bgamma) \right)\;\rd \bx \\
   & \qquad  - \int_{U_0} (\sfD(\bv_\bgamma)(\delta_i^h\sfA) + (\delta_i^h\sfA)\sfD(\bv_\bgamma)) : \bnabla (\eta_0^2\delta_i^h\bv_\bgamma) \;\rd \bx \\
   & \qquad + 2 \int_{U_0}  p_\bgamma (\delta_i^{-h}((\bnabla\eta_0)\cdot(\eta_0\delta_i^h\bv_\bgamma))) \;\rd \bx \\
   & \qquad + \int_{U_0} \bbf_\bgamma \cdot (\delta_i^{-h}(\eta_0^2\delta_i^h\bv_\bgamma)) \;\rd \bx.
\end{aligned}
\end{equation}
Notice how similar \eqref{weakgeneralizedStokeswithpressureinteriorgammadifferenceb} is to \eqref{weakgeneralizedStokeswithpressureinteriorb}, with $\bv$, $\bbf$ and $p$ replaced by $\bv_\bgamma$, $\bbf_\bgamma$ and $p_\bgamma$, the derivative $\partial_{x_i}$ replaced by the finite difference $\delta_i^h$, and the addition of the term with $\bb_\bgamma$. Using \eqref{finitedifferenceintestimate} and \eqref{finitedifferenceintestimateinfty}, we estimate \eqref{weakgeneralizedStokeswithpressureinteriorgammadifferenceb} just like \eqref{weakgeneralizedStokeswithpressureinteriorb}, except that we need to handle the extra term $\bb_\bgamma$. 

The idea is to consider $\bb_\bgamma$ as an extra forcing term. Once $\bb_\bgamma$ belongs to $L^2(\Omega)^3$, it can be treated just like $\bbf_\bgamma$. The fact that $\bb_\bgamma\in L^2(\Omega)^3$ will be discussed further on. 

Then, estimating the terms in \eqref{weakgeneralizedStokeswithpressureinteriorgammadifferenceb} like we did for \eqref{weakgeneralizedStokeswithpressureinteriorb}, we arrive at an interior estimate analogous to \eqref{interiorestimateh2c},
\begin{multline}
  \label{interiorestimatehkplus2a}
  \|\bnabla(\eta_0\delta_i^h\bv_\bgamma)\|_{L^2(\Omega)}^2 \leq \frac{c(\Omega)}{\alpha^2}\left( \|\bbf_\bgamma\|_{L^2(\Omega)}^2 + \|\bb_\bgamma\|_{L^2(\Omega)}^2 + \lambda_1\|p_\bgamma\|_{L^2(\Omega)}^2 \right. \\
  \left. + \left(\lambda_1\|\sfA\|_{L^\infty(\Omega)}^2 + \|\bnabla\sfA\|_{L^\infty(\Omega)}^2\right)\|\bnabla \bv_\bgamma\|_{L^2(\Omega)}^2\right),
\end{multline}
for $i=1,2,3$, where $c(\Omega)$ is a shape constant. With $\bbf_\bgamma=D^\bgamma\bbf$ and $\bb_\bgamma$ in $L^2(\Omega)^3$, $\sfA\in W^{1,\infty}(\Omega)^{3\times 3}$, and $\bv_\bgamma = D^\bgamma\bv \in H^1(\Omega)^3$, the right hand side above is bounded and independent of $h$. Thus, $\eta_0\delta_i^h\bv_\bgamma$ has a subsequence that converges weakly in $H^1(\Omega)^3$. Since, on the other hand, $\bv_\bgamma$ belongs to $H^1(\Omega)$ and, hence, $\delta_i^h\bv$ converges weakly to $\partial_{x_i}\bv$ in $L^2_\loc(\Omega)^3$, then $\eta_0\delta_i^h\bv_\bgamma$ must converge in $H^1(\Omega)^3$ precisely to $\eta_0\partial_{x_i}\bv_\bgamma$, with $\eta_0\partial_{x_i}\bv_\bgamma\in H^1(\Omega)^3$. Since this applies to any $i=1,2,3$ and $\eta_0$ is of class $\Ccal^2$, we find that $\eta_0\bv_\bgamma\in H^2(\Omega)^3$ and, passing to the limit as $h\rightarrow 0$ in \eqref{interiorestimatehkplus2a}, we find that
\begin{multline}
  \label{interiorestimatehkplus2b}
  \|D^2(\eta_0\bv_\bgamma)\|_{L^2(\Omega)}^2 \leq \frac{c(\Omega)}{\alpha^2}\left( \|\bbf_\bgamma\|_{L^2(\Omega)}^2 + \|\bb_\bgamma\|_{L^2(\Omega)}^2 + \lambda_1\|p_\bgamma\|_{L^2(\Omega)}^2 \right. \\
  \left. + \left(\lambda_1\|\sfA\|_{L^\infty(\Omega)}^2 + \|\bnabla\sfA\|_{L^\infty(\Omega)}^2\right)\|\bnabla \bv_\bgamma\|_{L^2(\Omega)}^2\right).
\end{multline}
  
  The estimates on the boundary are a bit more involved than in the case of low-order regularity. The difficulty is that, here, $\bv_\bgamma$ may not vanish on the boundary. Recall that in Theorem \ref{h2regularity} we have, instead of $\bv_\bgamma$, the vector field $\bv$, which does vanish on the boundary. In that case, we could take any tangential derivative of $\bv$ and still have it vanishing on the boundary.   Here, this procedure needs to be slightly modified, as follows.
  
  In general, the vector field $\bv_\bgamma$ does not vanish on the boundary because $\bgamma$ may contain directions which are not tangential to $\partial \Omega$. Fortunately, it suffices to start by allowing $\bgamma$ to contain only tangential derivatives. More precisely, if a multi-index $\bgamma$ only contains tangential directions (say $\bgamma\in \{1,2\}^k$ in the case the boundary has been rectified to $x_3=0$), then $\tilde\bv_\bgamma$ vanishes on the boundary. Hence, if $i=1$ or $2$, the finite difference $\delta_i^h\tilde\bv_\bgamma$ vanishes on the boundary, as well. Thus, by testing the localized and rectified equation against $\delta_i^{-h}(\eta_m\delta_i^h\tilde\bv_\bgamma)$, we obtain an $H^1$-estimate for $\eta_m\delta_i^h\tilde\bv_\bgamma$. This means that we obtain $L^2$ estimates for all the $(k+2)$-order derivatives with at most one derivative in the normal direction (the one coming from the $H^1$ norm). This leaves out all the $(k+2)$-order derivatives which have two or more derivatives in the normal direction. 
  
  Luckily, they can all be estimated directly from the localized and rectified version of the equation \eqref{generalizedStokesforderivativegamma}. This is done step by step. With $\tilde\bv_\bgamma$ having only tangential derivatives, we now obtain, directly and explicitly from the localized and rectified version of \eqref{generalizedStokesforderivativegamma}, an estimate for the second-order derivative of $\tilde\bv_\bgamma$ in the normal direction in terms of all the previously estimated second-order derivatives of $\bv_\bgamma$ with at most one normal direction. This gives an estimate for all the $(k+2)$-order derivatives with at most two derivatives in the normal direction. Then we allow $\tilde\bv_\bgamma$ to have at most one derivative in the normal direction and look again at \eqref{generalizedStokesforderivativegamma} to obtain estimates for all the $(k+2)$-order derivatives with at most three derivatives in the normal direction. We continue this way, until we find estimates for any $(k+2)$-order derivative. Coming back to the original variables, we obtain an estimate similar to \eqref{interiorestimatehkplus2a} and \eqref{estimateh2}, namely
\begin{multline}
  \label{estimateh2vgamma}
  \|D^2 (\eta_m\bv_\bgamma)\|_{L^2(\Omega)}^2 \leq \frac{c(\Omega)}{\alpha^2}\sum_{|\hat\bgamma|=k}\left( \|\bbf_{\hat\bgamma} + \bb_{\hat\bgamma}\|_{L^2(\Omega)}^2 + \lambda_1\|p_{\hat\bgamma}\|_{L^2(\Omega)}^2 \right. \\
  \left. +\left(\lambda_1\|\sfA\|_{L^\infty(\Omega)}^2 + \|\bnabla\sfA\|_{L^\infty(\Omega)}^2\right)\|D \bv_{\hat\bgamma}\|_{L^2(\Omega)}^2\right),
\end{multline}
where $c(\Omega)$ is a shape constant. Notice that, in the process described above, we need to consider all the multi-indices $\hat\bgamma$ of order $k$, i.e. all the lower-order derivatives, to obtain the desired bound. This is reflected in \eqref{estimateh2vgamma} in the summation over all the multi-indices with $|\hat\bgamma|=k$.

Combining all the localized estimates in the partition of unity we obtain a global bound for $\bv_\bgamma = \sum_m \eta_m \bv_\bgamma = \sum_m \eta_m D^\bgamma \bv$, for any $|\bgamma|=k$, which can be written as
\begin{multline}
  \label{estimatehkplus2}
  \|D^{k+2} \bv\|_{L^2(\Omega)}^2 \leq \frac{c(\Omega)}{\alpha^2}\left( \|D^k\bbf\|_{L^2(\Omega)}^2 + \sum_{|\bgamma|=k}\|\bb_\bgamma\|_{L^2(\Omega)}^2 + \lambda_1\|D^k p\|_{L^2(\Omega)}^2 \right. \\
    \left. +\left(\lambda_1\|\sfA\|_{L^\infty(\Omega)}^2 + \|\bnabla\sfA\|_{L^\infty(\Omega)}^2\right)\|D^{k+1} \bv\|_{L^2(\Omega)}^2\right),
\end{multline}
where $c(\Omega)$ is another shape constant.    
    
Now it is time to estimate $D^k p$ and $\bb_\bgamma$, with $|\bgamma|=k$, for $k\geq 1$. In the case $k=1$,  we have $\bgamma = \be_1$, $\be_2$, or $\be_3$, and, moreover, $\bb_{\bgamma'} = 0$ and $\bv_{\bgamma'} = \bv$. Thus, we have
\[ \bb_{\be_j} = \bB(\partial_{x_j}\sfA,\bv) = \bnabla \cdot (\sfD(\bv)(\partial_{x_j}\sfA) + (\partial_{x_j}\sfA)\sfD(\bv)).
\]
We estimate each $\bb_{\be_j}$ using the inequality \eqref{l6inequality} and the assumption that $\sfA\in W^{1,\infty}(\Omega)^{3\times 3}\cap W^{2,3}(\Omega)^{3\times 3}$, so that
\begin{multline}
 \label{estimatebbejinl2}
  \|\bb_{\be_j}\|_{L^2(\Omega)} \leq  c\left(\|\bnabla\sfA\|_{L^\infty(\Omega)}\|D^2 \bv\|_{L^2(\Omega)} + \|D^2\sfA\|_{L^3(\Omega)}\|\bnabla \bv\|_{L^6(\Omega)}\right)  \\
      \leq  c(\Omega)\left(\|\bnabla\sfA\|_{L^\infty(\Omega)} + \|D^2\sfA\|_{L^3(\Omega)}\right)\|D^2 \bv\|_{L^2(\Omega)},
\end{multline} 
for $j=1,2,3,$ and another shape constant $c(\Omega)$.  For $k\geq 2$, we estimate $\bb_\bgamma$ recursively,
  \begin{align*}
    \|\bb_\bgamma\|_{L^2(\Omega)} & \leq \|D \bb_{\bgamma'}\|_{L^2(\Omega)} + c\|D^2 \sfA\|_{L^3(\Omega)}\|D^k\bv\|_{L^6(\Omega)} + c \|\bnabla \sfA\|_{L^\infty(\Omega)}\|D^{k+1}\bv\|_{L^2(\Omega)}\\ 
    & \leq \|D^2\bb_{\bgamma''}\|_{L^2(\Omega)} + c \|D^3 \sfA\|_{L^2(\Omega)}\|D^{k-1}\bv\|_{L^\infty(\Omega)}\\
    & \qquad \qquad\qquad + c\|D^2 \sfA\|_{L^3(\Omega)}\|D^k\bv\|_{L^6(\Omega)} + c \|\bnabla \sfA\|_{L^\infty(\Omega)}\|D^{k+1}\bv\|_{L^2(\Omega)} \\
    & \leq \ldots \leq c\|D^{k+1}\sfA\|_{L^2(\Omega)}\|\bnabla \bv\|_{L^\infty(\Omega)} +  ... + c \|D^3 \sfA\|_{L^2(\Omega)}\|D^{k-1}\bv\|_{L^\infty(\Omega)}\\
    & \qquad  \qquad \qquad  + c\|D^2 \sfA\|_{L^3(\Omega)}\|D^k\bv\|_{L^6(\Omega)} + c \|\bnabla \sfA\|_{L^\infty(\Omega)}\|D^{k+1}\bv\|_{L^2(\Omega)},
  \end{align*}  
  where $\bgamma''$ is a multi-index one-degree lower than $\bgamma'$, and so on. This can be written in the form
  \begin{multline*}
    \|\bb_\bgamma\|_{L^2(\Omega)} \leq c\sum_{j=1}^{k-1} \|D^{k+2-j}\sfA\|_{L^2(\Omega)}\|D^{j}\bv\|_{L^\infty(\Omega)} \\
    + c\|D^2 \sfA\|_{L^3(\Omega)}\|D^k\bv\|_{L^6(\Omega)} + c \|\bnabla \sfA\|_{L^\infty(\Omega)}\|D^{k+1}\bv\|_{L^2(\Omega)}.
  \end{multline*}
  
  Using  the Agmon-type inequality \eqref{agmoninequality} and the inequality \eqref{l6inequality}, we see that
  \begin{multline}
    \label{estimatebgamma}
    \|\bb_\bgamma\|_{L^2(\Omega)} \leq  \frac{c(\Omega)}{\lambda_1^{1/4}}\sum_{j=1}^{k-1} \|D^{k+2-j}\sfA\|_{L^2(\Omega)}\|D^{j+2}\bv\|_{L^2(\Omega)} \\
    + c(\Omega)\left(\|D^2 \sfA\|_{L^3(\Omega)}+ \|\bnabla \sfA\|_{L^\infty(\Omega)}\right)\|D^{k+1}\bv\|_{L^2(\Omega)}.
  \end{multline}
  This is valid for any multi-index $\bgamma$ of order $k\geq 2$. In fact, formula \eqref{estimatebgamma} is valid even for $k=1$, since in this case the summation is empty and the estimate reduces to \eqref{estimatebbejinl2}. 
    
  Using estimate \eqref{estimatebgamma}, inequality \eqref{estimatehkplus2} yields
  \begin{multline}
  \label{estimatehkplus2b}
  \|D^{k+2} \bv\|_{L^2(\Omega)}^2 \leq \frac{c(\Omega)}{\alpha^2}\left( \|D^k\bbf\|_{L^2(\Omega)}^2 + \frac{c(\Omega)}{\lambda_1^{1/2}}\sum_{j=1}^{k-1} \|D^{k+2-j}\sfA\|_{L^2(\Omega)}^2\|D^{j+2}\bv\|_{L^2(\Omega)}^2 \right. \\
  \left. + \lambda_1\|D^k p\|_{L^2(\Omega)}^2 + \left(\|D^2 \sfA\|_{L^3(\Omega)}^2 + \|\sfA\|_{W^{1,\infty}_{\lambda_1}(\Omega)}^2\right)\|D^{k+1} \bv\|_{L^2(\Omega)}^2\right),
\end{multline}
  for $k\geq 1$.

The pressure term in \eqref{estimatehkplus2b} can be obtained by taking the derivative of \eqref{generalizedStokes} up to order $k-1$ (one degree lower than \eqref{generalizedStokesforderivativegamma}, i.e. with $\bgamma$ replaced by $\bgamma'$, with $|\bgamma'|=k-1$). In doing so, we find equations for $D^k p$ valid in $L^2(\Omega)^3$, which yield an estimate analogous to \eqref{basicpestimatefromStokes}, namely
\begin{multline}
  \label{estimatesdkplusoneofp}
  \|D^k p\|_{L^2(\Omega)} \leq \|D^{k-1}\bbf\|_{L^2(\Omega)} + \sum_{|\bgamma'|=k-1} \|\bb_{\bgamma'}\|_{L^2(\Omega)} + c\|\sfA\|_{L^\infty(\Omega)}\|D^{k+1}\bv\|_{L^2(\Omega)} \\
   + c\|\bnabla \sfA\|_{L^\infty(\Omega)}\|D^k\bv\|_{L^2(\Omega)}.
\end{multline}
Using the Poincar\'e inequality \eqref{poincareineq} and the estimate \eqref{estimatebgamma} with $\bgamma'$ instead of $\bgamma$ yields
\begin{multline}
  \label{estimatesdkplusoneofpb}
  \|D^k p\|_{L^2(\Omega)} \leq \|D^{k-1}\bbf\|_{L^2(\Omega)} + \frac{c(\Omega)}{\lambda_1^{1/4}}\sum_{j=1}^{k-2} \|D^{k+1-j}\sfA\|_{L^2(\Omega)}\|D^{j+2}\bv\|_{L^2(\Omega)} \\
    + c(\Omega)\left(\|D^2 \sfA\|_{L^3(\Omega)}+ \|\bnabla \sfA\|_{L^\infty(\Omega)}\right)\|D^k\bv\|_{L^2(\Omega)} + c\|\sfA\|_{L^\infty(\Omega)}\|D^{k+1}\bv\|_{L^2(\Omega)}.
\end{multline}
Inserting \eqref{estimatesdkplusoneofpb} into \eqref{estimatehkplus2b} and using the corresponding estimate \eqref{estimatebgamma} with $\bgamma'$ instead of $\bgamma$ and using the Poincar\'e inequality \eqref{poincareineq} yields
\begin{multline}
  \label{estimatehkplus2c}
  \|D^{k+2} \bv\|_{L^2(\Omega)}^2 \leq \frac{c(\Omega)}{\alpha^2}\left(\|\bbf\|_{H^k_{\lambda_1}(\Omega)}^2 + \frac{1}{\lambda_1^{1/2}}\sum_{j=1}^{k-1} \|\sfA\|_{H^{k+2-j}_{\lambda_1}(\Omega)}^2\|D^{j+2}\bv\|_{L^2(\Omega)}^2 \right. \\
  \left. + \left(\|D^2 \sfA\|_{L^3(\Omega)}^2 + \|\sfA\|_{W^{1,\infty}_{\lambda_1}(\Omega)}^2\right)\|D^{k+1} \bv\|_{L^2(\Omega)}^2\right).
\end{multline}
Due to the Sobolev embeddings in our three-dimensional space, we treat the cases $k=1$ and $k\geq 2$ differently. For $k=1$, we find, from \eqref{estimatehkplus2c}, that
\begin{equation}
  \label{estimateh3b}
   \|D^3 \bv\|_{L^2(\Omega)}^2 \leq \frac{c(\Omega)}{\alpha^2}\left(\|\bbf\|_{H^1_{\lambda_1}(\Omega)}^2 + \left(\|D^2 \sfA\|_{L^3(\Omega)}^2 + \|\sfA\|_{W^{1,\infty}_{\lambda_1}(\Omega)}^2\right)\|D^2 \bv\|_{L^2(\Omega)}^2\right).
\end{equation}
Estimate \eqref{estimateh3b} together with the estimate \eqref{h2regularityforh2v} for $D^2 \bv$ yields
\begin{multline}
    \label{estimateh3c}   
        \|D^3\bv\|_{L^2(\Omega)}^2 \leq \frac{c(\Omega)}{\alpha^2}\left( \|\bbf\|_{H^1_{\lambda_1}(\Omega)}^2  \right. \\
    \left. + \frac{1}{\alpha^2}\left(\|\sfA\|_{W^{1,\infty}_{\lambda_1}(\Omega)}^2 + \|D^2\sfA\|_{L^3(\Omega)}^2\right)\left( \|\bbf\|_{L^2(\Omega)}^2 + \frac{1}{\alpha^2}\|\sfA\|_{W^{1,\infty}_{\lambda_1}(\Omega)}^2\|\bbf\|_{H^{-1}(\Omega)}^2\right)\right),
      \end{multline}
which proves \eqref{hkplus2regularityforv} in the case $k=1$. Using \eqref{estimatesdkplusoneofpb} with $k=2$ together with \eqref{estimateh3b} we find that
\begin{multline}
    \label{h2regularityforp}
    \|D^2 p\|_{L^2(\Omega)} \leq \frac{c(\Omega)}{\alpha}\|\sfA\|_{L^\infty(\Omega)}\left( \|\bbf\|_{H^1_{\lambda_1}(\Omega)}  \right. \\
    \left. + \frac{1}{\alpha}\left(\|\sfA\|_{W^{1,\infty}_{\lambda_1}(\Omega)} + \|D^2\sfA\|_{L^3(\Omega)}\right)\left( \|\bbf\|_{L^2(\Omega)} + \frac{1}{\alpha}\|\sfA\|_{W^{1,\infty}_{\lambda_1}(\Omega)}\|\bbf\|_{H^{-1}(\Omega)}\right)\right),
\end{multline}
which proves the estimate \eqref{hkplus1regularityforp} for the pressure in the case $k=1$.

In the case $k\geq 2$, since $H^3(\Omega)$ is included in both $W^{2,3}(\Omega)$ and $W^{1,\infty}(\Omega)$, the last term in \eqref{estimatehkplus2c} is bounded by the term with $j=k-1$ in the summation. Thus, we rewrite \eqref{estimatehkplus2c} simply as
\begin{equation}
  \label{estimatehkplus2d}
  \|D^{k+2} \bv\|_{L^2(\Omega)}^2 \leq \frac{c(\Omega)}{\alpha^2}\left(\|\bbf\|_{H^k_{\lambda_1}(\Omega)}^2 + \frac{1}{\lambda_1^{1/2}}\sum_{j=1}^{k-1} \|\sfA\|_{H^{k+2-j}_{\lambda_1}(\Omega)}^2\|D^{j+2}\bv\|_{L^2(\Omega)}^2 \right). 
\end{equation}
for a different constant $c(\Omega)$.  Estimate \eqref{estimatehkplus2d} is of the form
\begin{equation}
  \label{vrecursion}
   v_{k+2} \leq c(\Omega)\left(d_k +  \sum_{j=1}^{k-1} a_{k+2-j}v_{j+2}\right),
\end{equation}
for $k\geq 2$, where 
\[ v_j = \|D^j\bv\|_{L^2(\Omega)}^2, \quad d_j = \frac{1}{\alpha^2} \|\bbf\|_{H^j_{\lambda_1}(\Omega)}^2, \quad a_j = \frac{1}{\alpha^2}\frac{1}{\lambda_1^{1/2}}\|\sfA\|_{H^j_{\lambda_1}(\Omega)}^2,
\]
for $j\geq 2$. We include the case $k=1$ in \eqref{vrecursion} by setting $d_1 = v_3$. 
Then, we show by induction that
\begin{equation}
  \label{recursioninequalitycomb}
  v_{k+2} \leq c(\Omega) \left(d_k + \sum_{j=1}^{k-1} \left( \sum_{\substack{\alpha_1, \ldots, \alpha_{k-j}\in \NN_0 \\ \alpha_1 + \ldots + (k-j)\alpha_{k-j} = k-j}} c_{\alpha_1,\ldots,\alpha_{k-j}}^k a_3^{\alpha_1}\cdots a_{k-j+2}^{\alpha_{k-j}}\right)d_j \right),
\end{equation}
for $k\geq 1$, where $c_{\alpha_1,\ldots,\alpha_{k-j}}^k$ are suitable numerical constants given recursively by
\begin{align}
  \label{combinatoricconstants1} 
  & c_{0,\ldots, 0, 1}^k = 1, \\
  \label{combinatoricconstants2} 
  & c_{\alpha_1,\ldots,\alpha_{k-j -1},0}^k = \sum_{\ell=j+1}^{k-1} \sum_{\substack{\tilde\alpha_1, \ldots, \tilde\alpha_{\ell-j}\in \NN_0 \\ \tilde\alpha_1 + \ldots + (\ell-j)\tilde\alpha_{\ell-j} = \ell-j}} c_{\tilde\alpha_1,\ldots,\tilde\alpha_{\ell-j}}^\ell,
\end{align}
for all $k\geq 1$ and $j=1,\ldots, k-1$. Moreover, in the formula \eqref{combinatoricconstants2}, the exponents $\alpha_1, \ldots, \alpha_{k-j-1}$ are related to $\ell$ and $\tilde\alpha_1, \ldots, \tilde\alpha_{\ell-j}$ by the (one-to-one) correspondences
\begin{equation}
  \label{alpharelations}
  \alpha_m = \begin{cases} 
      \tilde\alpha_m, & 1 \leq m \leq \ell-j, \; m\neq k-\ell, \\
      \tilde\alpha_{k-\ell} + 1, &  m = k-\ell, \text{ if } k-\ell \leq \ell - j, \\
      1, &  m = k-\ell, \text{ if } k-\ell > \ell - j, \\
      0, &  \ell-j < m \leq k-j, \; m\neq k-\ell,
    \end{cases}
\end{equation}
and
\begin{equation}
  \label{alpharelations2}
  \tilde\alpha_m = \begin{cases} 
      \alpha_m, & 1 \leq m \leq \ell-j, \; m\neq k-\ell, \\
      \alpha_{k-\ell} - 1, &  m = k-\ell, \text{ if } k-\ell \leq \ell - j. \\
    \end{cases}
\end{equation}
More precisely, for $k\geq 1$ and $j=1, \ldots, k-1$, the exponents $\tilde\alpha_1, \ldots, \tilde\alpha_{\ell-j}\in \NN_0$ satisfy the criteria $\tilde\alpha_1 + \ldots + (\ell-j)\tilde\alpha_{\ell-j} = \ell-j$, for $\ell\in \{j+1,\ldots, k_1\}$, if, and only if, the exponents $\alpha_1,\ldots \alpha_{k-j-1}\in \NN_0$ satisfy the criteria $\alpha_1 + \ldots + (k-j-1)\alpha_{k-j-1} = k-j$..

Indeed, this is trivially true for $k=1$. For $k\geq 2$, assuming it is true up to $k-1$, we find
\begin{align*}
 v_{k+2} & \leq c(\Omega)\left(d_k  \right. \\
   &  \qquad  \left. + \sum_{j=1}^{k-1} a_{k+2-j}\left(d_j + \sum_{\ell=1}^{j-1} \left( \sum_{\substack{\tilde\alpha_1, \ldots, \tilde\alpha_{j-\ell}\in \NN_0 \\ \tilde\alpha_1 + \ldots + (j-\ell)\tilde\alpha_{j-\ell} = j-\ell}} c_{\tilde\alpha_1,\ldots,\tilde\alpha_{j-\ell}}^ja_3^{\tilde\alpha_1}\cdots a_{j-\ell+2}^{\tilde\alpha_{j-\ell}}\right)d_\ell\right)\right) \\
   & = c(\Omega)\left(d_k +  \sum_{j=1}^{k-1} a_{k+2-j}d_j \right. \\
    & \qquad  \qquad \left. + \sum_{\ell=1}^{k-1} \sum_{j=\ell+1}^{k-1} a_{k+2-j} \left( \sum_{\substack{\tilde\alpha_1, \ldots, \tilde\alpha_{j-\ell}\in \NN_0 \\ \tilde\alpha_1 + \ldots + (j-\ell)\tilde\alpha_{j-\ell} = j-\ell}} c_{\tilde\alpha_1,\ldots,\tilde\alpha_{j-\ell}}^j a_3^{\tilde\alpha_1}\cdots a_{j-\ell+2}^{\tilde\alpha_{j-\ell}}\right)d_\ell\right) \\
   & = c(\Omega)\left(d_k +  \sum_{j=1}^{k-1} a_{k+2-j}d_j \right. \\
    & \qquad  \qquad \left. + \sum_{j=1}^{k-1} \sum_{\ell=j+1}^{k-1} a_{k+2-\ell} \left( \sum_{\substack{\tilde\alpha_1, \ldots, \tilde\alpha_{\ell-j}\in \NN_0 \\ \tilde\alpha_1 + \ldots + (\ell-j)\tilde\alpha_{\ell-j} = \ell-j}} c_{\tilde\alpha_1,\ldots,\tilde\alpha_{\ell-j}}^\ell a_3^{\tilde\alpha_1}\cdots a_{\ell-j+2}^{\tilde\alpha_{\ell-j}}\right)d_j\right).
\end{align*}
Putting $d_j$ in evidence we rewrite the estimate as
\begin{multline}
  \label{recursioninequalitycombpre}
  v_{k+2} \leq c(\Omega)d_k \\
   + c(\Omega)\sum_{j=1}^{k-1} \left(a_{k+2-j} + \sum_{\ell=j+1}^{k-1} a_{k+2-\ell} \left( \sum_{\substack{\tilde\alpha_1, \ldots, \tilde\alpha_{\ell-j}\in \NN_0 \\ \tilde\alpha_1 + \ldots + (\ell-j)\tilde\alpha_{\ell-j} = \ell-j}} c_{\tilde\alpha_1,\ldots,\tilde\alpha_{\ell-j}}^\ell a_3^{\tilde\alpha_1}\cdots a_{\ell-j+2}^{\tilde\alpha_{\ell-j}}\right)\right)d_j.
\end{multline}
We argue that
\begin{multline}
  \label{combinatorics}
  a_{k+2-j} + \sum_{\ell=j+1}^{k-1} a_{k+2-\ell} \left( \sum_{\substack{\tilde\alpha_1, \ldots, \tilde\alpha_{\ell-j}\in \NN_0 \\ \tilde\alpha_1 + \ldots + (\ell-j)\tilde\alpha_{\ell-j} = \ell-j}} c_{\tilde\alpha_1,\ldots,\tilde\alpha_{\ell-j}}^\ell a_3^{\tilde\alpha_1}\cdots a_{\ell-j+2}^{\tilde\alpha_{\ell-j}}\right) \\
  = \sum_{\substack{\alpha_1, \ldots, \alpha_{k-j}\in \NN_0 \\ \alpha_1 + \ldots + (k-j)\alpha_{k-j} = k-j}} c_{\alpha_1,\ldots,\alpha_{k-j}}^ka_3^{\alpha_1}\cdots a_{k-j+2}^{\alpha_{k-j}},
\end{multline}
with the numerical constants satisfying \eqref{combinatoricconstants1} and \eqref{combinatoricconstants2} and the exponents satisfying \eqref{alpharelations} and \eqref{alpharelations2}.

Indeed, the first term $a_{k+2-j}$ is in the summation term in the right hand side with $\alpha_{k+j} = 1$ and the remaining exponents $\alpha_\ell =0$, $1\leq \ell < k-j$, with the associated numerical constant being 
\[ c_{0,\ldots, 0, 1}^k=1. 
\]
This is the only possible combination that contains the ``highest-order'' term $a_{k+2-j}$. 

For $\ell = j+1, \ldots, k-1$ and any combination $\tilde\alpha_1, \ldots, \tilde\alpha_{\ell-j}\in \NN_0$ with $\tilde\alpha_1 + \ldots + (\ell-j)\tilde\alpha_{\ell-j} = \ell-j$, we have
\[ \tilde\alpha_1 + \ldots + (\ell-j)\tilde\alpha_{\ell-j} + (k-\ell) = \ell-j + k-\ell = k-j.
\]
Notice that $k-\ell$ may, or may not, be less than $\ell-j$. Thus, we take
\begin{equation}
  \label{alphaaddone}
  \tilde\alpha_{k-\ell} = \begin{cases} \alpha_{k-\ell}+1, & \text{if } k-\ell \leq \ell-j, \\ 1, & \text{ otherwise}, \end{cases}
\end{equation}
and keep the remaining exponents the same, i.e. $\alpha_m = \tilde\alpha_m$ if $m\leq \ell-j$, $m\neq k-\ell$, and $\alpha_m = 0$, if $\ell-j < m \leq k-1$, $m\neq k-\ell$. In doing so, the exponents $\alpha_1, \ldots \alpha_{k-j}$ satisfy
\[ \alpha_1 + \ldots + (k-j)\alpha_{k-j} = k-j, 
\]
with $\alpha_{k-j}=0$. Hence, this combination of exponents is in the right hand side as well, and, from \eqref{combinatorics}, we see that the associated constant $c_{\alpha_1,\ldots, \alpha_{k-j}}^k$ is given according to \eqref{combinatoricconstants2}.

On the other hand, if the exponents $\alpha_1, \ldots, \alpha_{k-j}\in \NN_0$ satisfy $\alpha_1 + \ldots + (k-j)\alpha_{k-j} = k-j$, then either $\alpha_{k-j}=1$ or $0$. If $\alpha_{k-j}=1$, the remaining exponents are zero, so that this term is in the left hand side of \eqref{combinatorics} and the associated numerical constant is $c_{0,\ldots,1}=1$. 

In the case $\alpha_{k-j} = 0$, then $\alpha_{k-\ell} \neq 0$, for one or more $j + 1 \leq \ell \leq k-1$. For any such $\ell$, we see that
\[ \alpha_1 + \ldots + (k-\ell)(\alpha_{k-\ell}-1)  = k-j - (k-\ell) = \ell - j.
\]
Thus, we define $\tilde\alpha_m = \alpha_m$, for $1\leq m \leq \ell-j$, with $m\neq k-\ell$, and $\tilde\alpha_m = \alpha_m-1$, if $m=k-\ell\leq \ell-j$ (if $k-\ell> \ell-j$, there is no need to define the corresponding $\tilde\alpha_m$; see \eqref{alpharelations2}). In doing so, the exponents $\tilde\alpha_1, \ldots, \tilde\alpha_{\ell-j}$ satisfy the condition to be in the summation term in the left hand side of \eqref{combinatorics}. 

This completes the proof of \eqref{combinatorics}, with the numerical constant satisfying \eqref{combinatoricconstants1} and \eqref{combinatoricconstants2} and the exponents satisfying \eqref{alpharelations} and \eqref{alpharelations2}. Combining \eqref{recursioninequalitycombpre} and \eqref{combinatorics} we deduce \eqref{recursioninequalitycomb}.

The estimate \eqref{recursioninequalitycomb} is the sharpest we can get from \eqref{vrecursion}. But in order to have a more readable and manageable inequality, we proceed to estimate the summation terms in \eqref{recursioninequalitycomb}. We divide the combinatorial condition $\sum_{i=1}^{k-j} i\alpha_i = k-j$ by $k-j$ to see that
\[ \sum_{i=1}^{k-j} \frac{i\alpha_i}{k-j} = 1,
\]
with $0 \leq i\alpha_i/(k-j) \leq 1$, for all $i=1,\ldots, k-j$. Using Young's inequality, we find that
\begin{equation}
  \label{combinatorialtermestimate}
  \prod_{i=1}^{k-j} a_{i+2}^{\alpha_i} \leq \sum_{i=1}^{k-j} \frac{i\alpha_i}{k-j} a_{i+2}^{(k-j)/i} \leq \sum_{i=1}^{k-j} a_{i+2}^{(k-j)/i}.
\end{equation}
Now, applying estimate \eqref{combinatorialtermestimate} to \eqref{recursioninequalitycomb} gives us
\begin{equation}
  \label{recursioninequalityest}
  v_{k+2} \leq c(\Omega) \left(d_k + c_k \sum_{j=1}^{k-1} \left( \sum_{i=1}^{k-j} a_{i+2}^{(k-j)/i}\right)d_j \right),
\end{equation}
where
\begin{equation}
  \label{defck}
  c_k = \sum_{\substack{\alpha_1, \ldots, \alpha_{k-j}\in \NN_0 \\ \alpha_1 + \ldots + (k-j)\alpha_{k-j} = k-j}} c_{\alpha_1,\ldots,\alpha_{k-j}}^k.
\end{equation}
Going back to the definitions of $v_j$, $d_j$ and $a_j$, not forgetting that $d_1=v_3$, we find that
\begin{multline}
  \label{recursioninequalityestfull}
  \|D^{k+2}\bv\|_{L^2(\Omega)}^2 \leq c_k(\Omega) \left( \frac{1}{\alpha^2}\|\bbf\|_{H^k_{\lambda_1}(\Omega)}^2 + \frac{1}{\alpha^2}\sum_{j=2}^{k-1} \left( \sum_{i=1}^{k-j} \frac{1}{\alpha^{2(k-j)/i}}\frac{1}{\lambda_1^{(k-j)/2i}}\|\sfA\|_{H^{i+2}_{\lambda_1}(\Omega)}^{2(k-j)/i}\right)  \|\bbf\|_{H^j_{\lambda_1}(\Omega)}^2 \right. \\
  \left. + \left( \sum_{i=1}^{k-1} \frac{1}{\alpha^{2(k-1)/i}}\frac{1}{\lambda_1^{(k-1)/2i}}\|\sfA\|_{H^{i+2}_{\lambda_1}(\Omega)}^{2(k-1)/i}\right)\|D^3\bv\|_{L^2(\Omega)}^2\right),
\end{multline}
for suitable non-dimensional shape constants $c_k(\Omega)$, for any $k\geq 2$. Using the estimate \eqref{estimateh3c} for $v_3$ we arrive at 
\begin{equation} 
  \label{hkplus2regularityforvfull}
  \begin{aligned}
   \|D^{k+2} & \bv\|_{L^2(\Omega)}^2 \leq \frac{c_k(\Omega)}{\alpha^2} \left( \|\bbf\|_{H^k_{\lambda_1}(\Omega)}^2 +  \sum_{j=2}^{k-1} \left( \sum_{i=1}^{k-j} \frac{1}{\alpha^{2(k-j)/i}}\frac{1}{\lambda_1^{(k-j)/2i}}\|\sfA\|_{H^{i+2}_{\lambda_1}(\Omega)}^{2(k-j)/i}\right)\|\bbf\|_{H^j_{\lambda_1}(\Omega)}^2 \right. \\
   & + \left( \sum_{i=1}^{k-1} \frac{1}{\alpha^{2(k-1)/i}}\frac{1}{\lambda_1^{(k-1)/2i}}\|\sfA\|_{H^{i+2}_{\lambda_1}(\Omega)}^{2(k-1)/i}\right)\left( \|\bbf\|_{H^1_{\lambda_1}(\Omega)}^2 \right. \\
  & \left. \quad \quad + \frac{1}{\alpha^2}\left( \|\sfA\|_{W^{1,\infty}_{\lambda_1}(\Omega)}^2 + \|D^2\sfA\|_{L^3(\Omega)}^2\right)\left(\|\bbf\|_{L^2(\Omega)}^2 + \frac{1}{\alpha^2}\|\sfA\|_{W^{1,\infty}_{\lambda_1}(\Omega)}^2 \|\bbf\|_{H^{-1}(\Omega)}^2\right)\right).
  \end{aligned}
\end{equation}  
This proves \eqref{Rkdef}. The estimate \eqref{hkplus1regularityforp} for the pressure follows from \eqref{estimatesdkplusoneofpb} and \eqref{hkplus2regularityforv}.
\end{proof}
\medskip

We end this section with some remarks. 


 \begin{rmk}
   The estimates \eqref{weaksolutionweakgeneralizedStokesestimaterepeat},  \eqref{h2regularityforh2v}, \eqref{h3egularityforv}, and \eqref{hkplus2regularityforv} for the velocity field (and similarly for the pressure estimates \eqref{weaksolutionweakgeneralizedStokesestimatewithpressure}, \eqref{h2regularityforh1p}, \eqref{hkplus2regularityforv}, and \eqref{hkplus1regularityforp}) can be put into a single expression of the form
   \begin{equation} 
     \label{hkplustworegularityshort}
     \|D^{k+2}\bv\|_{L^2(\Omega)} \leq \frac{c(\Omega)}{\alpha} \sum_{j=-1}^k \lambda_1^{(k-j)/2}a_{k,j}(\alpha,\lambda_1,\sfA) \|\bbf\|_{H^j(\Omega)},
   \end{equation}
   for all integers $k\geq -1$, where the coefficients $a_{k,j}$ are non-dimensional quantities. In the case $k\geq 2$, these quantities are given by
   \begin{equation}
   \begin{aligned}
      & a_{k,k}(\alpha,\lambda_1,\sfA) = 1, \\
      & a_{k,j}(\alpha,\lambda_1,\sfA) =  \sum_{i=1}^{k-j} \frac{1}{\alpha^{(k-j)/i}}\frac{1}{\lambda_1^{(2i+1)(k-j)/4i}}\|\sfA\|_{H^{i+2}_{\lambda_1}(\Omega)}^{(k-j)/i}, \quad 1 \leq j \leq k-1, \\
      & a_{k,0}(\alpha,\lambda_1,\sfA) = \frac{1}{\alpha\lambda_1^{1/2}} \left(\|\sfA\|_{W^{1,\infty}_{\lambda_1}(\Omega)} + \|D^2\sfA\|_{L^3(\Omega)}\right) a_{k,1}(\alpha,\lambda_1,\sfA), \\
     & \quad = \frac{1}{\alpha\lambda_1^{1/2}}\left(\|\sfA\|_{W^{1,\infty}_{\lambda_1}(\Omega)} + \|D^2\sfA\|_{L^3(\Omega)}\right)\sum_{i=1}^{k-1} \frac{1}{\alpha^{(k-j)/i}}\frac{1}{\lambda_1^{(2i+1)(k-j)/4i}}\|\sfA\|_{H^{i+2}_{\lambda_1}(\Omega)}^{(k-j)/i}, \\
     & a_{k,-1}(\alpha,\lambda_1,\sfA) = \frac{1}{\alpha\lambda_1^{1/2}}\|\sfA\|_{W^{1,\infty}_{\lambda_1}(\Omega)}a_{k,0}(\alpha,\lambda_1,\sfA),  \\
     & \quad =  \frac{1}{\alpha^2\lambda_1}\|\sfA\|_{W^{1,\infty}_{\lambda_1}(\Omega)}\left(\|\sfA\|_{W^{1,\infty}_{\lambda_1}(\Omega)} + \|D^2\sfA\|_{L^3(\Omega)}\right)\sum_{i=1}^{k-1} \frac{1}{\alpha^{(k-j)/i}}\frac{1}{\lambda_1^{(2i+1)(k-j)/4i}}\|\sfA\|_{H^{i+2}_{\lambda_1}(\Omega)}^{(k-j)/i},
   \end{aligned}
   \end{equation}
   while for $k=-1, 0, 1$, we have
   \begin{equation}
     \begin{aligned}
      & a_{-1,-1}(\alpha,\lambda_1,\sfA) = a_{0,0}(\alpha,\lambda_1,\sfA) = a_{1,1}(\alpha,\lambda_1,\sfA) = 1, \\
      & a_{0,-1}(\alpha,\lambda_1,\sfA) = \frac{1}{\alpha\lambda_1^{1/2}}\|\sfA\|_{W^{1,\infty}_{\lambda_1}(\Omega)}, \\
      & a_{1,0}(\alpha,\lambda_1,\sfA) = \frac{1}{\alpha\lambda_1^{1/2}}\left(\|\sfA\|_{W^{1,\infty}_{\lambda_1}(\Omega)} + \|D^2\sfA\|_{L^3(\Omega)}\right), \\
      & a_{1,-1}(\alpha,\lambda_1,\sfA) = \frac{1}{\alpha^2\lambda_1}\|\sfA\|_{W^{1,\infty}_{\lambda_1}(\Omega)}\left(\|\sfA\|_{W^{1,\infty}_{\lambda_1}(\Omega)} + \|D^2\sfA\|_{L^3(\Omega)}\right).
    \end{aligned}
  \end{equation}
  Notice, from \eqref{generalizedStokes}, that $\bbf$ has the physical dimension (scaling) of $\alpha \lambda_1\bv$, while $D^j$ has the physical dimension of $\lambda_1^{j/2}$, so it is interesting to rewrite \eqref{hkplustworegularityshort} as
   \begin{equation} 
     \label{hkplustworegularityshort2}
     \frac{1}{\lambda_1^{(k+2)/2}}\|D^{k+2}\bv\|_{L^2(\Omega)} \leq c(\Omega)\sum_{j=-1}^k a_{k,j}(\alpha,\lambda_1,\sfA) \frac{ \|\bbf\|_{H^j(\Omega)}}{\alpha \lambda_1^{(j+2)/2}},
   \end{equation}    
 \end{rmk} 
 
 \begin{rmk}
 We may use the interpolation inequality \eqref{interpolationineq} to slightly rewrite the summation in the right hand side of estimate \eqref{combinatorialtermestimate}, which ends up in the estimate \eqref{recursioninequalityest} and, consequently, in the final estimates \eqref{hkplus2regularityforv} and \eqref{hkplus1regularityforp} with \eqref{Rkdef}. Using \eqref{interpolationineq}, we see that (up to a multiplicative shape constant that will be accounted for at the end of the argument but which is omitted here for notational simplicity)
\[ a_{i+2} \leq a_3^{\theta_i}a_{k-j+2}^{1-\theta_i},
\]
where $\theta_i$ is given by
\begin{equation}
  \label{deftheta}
  i+2 = 3\theta_i + (1-\theta_i)(k-j+2).
\end{equation}
We use Young's inequality to write
\begin{equation}
  \label{aiafterinterpolationandyoung}
  a_{i+2}^{(k-j)/i} \leq a_3^{\theta_i (k-j)/i}a_{k-j+2}^{(1-\theta_i)(k-j)/i} \leq \frac{1}{p}a_3^{\theta_i (k-j)p/i} + \frac{1}{q}a_{k-j+2}^{(1-\theta_i)(k-j)q/i},
\end{equation}
for any $1\leq p, q \leq \infty$, with $p^{-1} + q^{-1} = 1$. Choosing $q$ such that 
\begin{equation}
  \label{youngdefq}
  \frac{(1-\theta_i)(k-j)q}{i} = 1,
\end{equation}
we find that
\[ \frac{1}{p} = 1 - \frac{(1-\theta_i)(k-j)}{i} = \frac{i - (1-\theta_i)(k-j)}{i}.
\]
Thus, one easily checks, using the value of $\theta_i$ given in \eqref{deftheta}, that
\begin{equation}
  \label{youngfindp}
  \frac{\theta_i (k-j)p}{i} = \frac{\theta_i (k-j)}{i - (1-\theta_i)(k-j)} = k-j.
\end{equation}
Combining \eqref{aiafterinterpolationandyoung}, \eqref{youngdefq} and \eqref{youngfindp} together, we see that
\[ a_{i+2}^{(k-j)/i} \leq a_3^{k-j} + a_{k-j+2}.
\]
Thus, the right hand side of \eqref{combinatorialtermestimate} is estimated as
\begin{equation}
  \label{prodtoyoungineq}
  \sum_{i=1}^{k-j} a_{i+2}^{(k-j)/i} \leq c(\Omega)(k-j) \left(a_3^{k-j} + a_{k-j}\right),
\end{equation}
where now we include the suitable constant $c(\Omega)$ coming from the interpolations. Applying estimate \eqref{prodtoyoungineq} to \eqref{recursioninequalitycomb} gives us
\begin{equation}
  \label{recursioninequalityestalt}
  v_{k+2} \leq c(\Omega) \left(d_k + kc_k \sum_{j=1}^{k-1} \left( a_3^{k-j} + a_{k-j}\right)d_j \right),
\end{equation}
for another suitable shape constant $c(\Omega)$, and where $c_k$ is given in \eqref{defck}. Going back to the definitions of $v_j$, $d_j$ and $a_j$, and using the estimate \eqref{estimateh3c} for $d_1=v_3$ we find that 
\begin{equation} 
  \label{hkplus2regularityforvfullalt}
  \begin{aligned}
   \|D^{k+2} & \bv\|_{L^2(\Omega)}^2 \leq \frac{c_k(\Omega)}{\alpha^2} \left( \|\bbf\|_{H^k_{\lambda_1}(\Omega)}^2 \right. \\
  & \left. + \sum_{j=2}^{k-1} \left( \frac{1}{\alpha^{2(k-j)}}\frac{1}{\lambda_1^{(k-j)/2}}\|\sfA\|_{H^3_{\lambda_1}(\Omega)}^{2(k-j)} +  \frac{1}{\alpha^2}\frac{1}{\lambda_1^{1/2}}\|\sfA\|_{H^{k-j}_{\lambda_1}(\Omega)}^2\right)  \|\bbf\|_{H^j_{\lambda_1}(\Omega)}^2 \right. \\
   & + \left( \frac{1}{\alpha^{2(k-1)}}\frac{1}{\lambda_1^{(k-1)/2}}\|\sfA\|_{H^3_{\lambda_1}(\Omega)}^{2(k-1)} +  \frac{1}{\alpha^2}\frac{1}{\lambda_1^{1/2}}\|\sfA\|_{H^{k-1}_{\lambda_1}(\Omega)}^2\right)\left( \|\bbf\|_{H^1_{\lambda_1}(\Omega)}^2 \right. \\
  & \left. \quad \quad + \frac{1}{\alpha^2}\left( \|\sfA\|_{W^{1,\infty}_{\lambda_1}(\Omega)}^2 + \|D^2\sfA\|_{L^3(\Omega)}^2\right)\left(\|\bbf\|_{L^2(\Omega)}^2 + \frac{1}{\alpha^2}\|\sfA\|_{W^{1,\infty}_{\lambda_1}(\Omega)}^2 \|\bbf\|_{H^{-1}(\Omega)}^2\right)\right),
  \end{aligned}
\end{equation}  
for suitable non-dimensional shape constants $c_k(\Omega)$, for any $k\geq 2$. A similar estimate holds for the pressure.
 \end{rmk}
 
\begin{rmk}  
   In the case $k\geq 2$, in Theorem \ref{hkplustworegularity}, solving explicitly the recurrence relation in the estimate \eqref{hkplus2regularityforv} for $D^{k+2}\bv$ is not a trivial matter. This relation is a so-called non-homogeneous linear recurrence relation and its solution depends on finding the roots of an intricate polynomial (see \cite{GreeneKnuth1982}). Luckily, we have found the exact solution for the recurrence relation and managed to estimate it in a suitable way (if \eqref{vrecursion} were an equality, than the equality in \eqref{recursioninequalitycomb} would be an exact solution).
\end{rmk}

\begin{rmk}
  \label{dontneedsomuchregularityonA}
  As discussed in the paragraph before Theorem \ref{hkplustworegularity}, all the works in the literature assume more regularity on the coefficients of the linear term of the equation than we need in our theorem. One can see the reason for that by perusing the proof of \cite[Theorem 15.1]{friedman1969} (see also \cite[Theorem 17.2]{friedman1969}). In \cite[Theorem 15.1]{friedman1969}, an equation of the form $\sum_{|\alpha|=m} a_\alpha D^\alpha u = f$ is considered, and the assumed regularity on $f$ and the coefficients $a_\alpha$ guarantee that the term $f_i = \partial_{x_i} f - A_i$ belongs to $H^{k-1}$, with $A_i$ defined in \cite[eq. (15.22)]{friedman1969} ($A_i$ contains derivatives of $a_\alpha$ and $u$, analogous to our term $\bb_\bgamma$, in \eqref{bgammadef}). Since $A_i$ contains terms that are products of the first-order derivatives of the solution $u$ with second-order derivatives of the coefficients of the linear term and of second-order derivatives of the solution $u$ with first-order derivatives of the coefficients of the linear term, the assumptions that $u\in H^{k+1}$ and that the coefficients are in $\Ccal^{k+1}$ guarantee that $f_i\in H^{k-1}$, in any dimensions. This, in turn, is used to bootstrap and show that $\partial_{x_i}u\in H^k$, hence $u\in H^{k+2}$. In our case, however, we work specifically in the three-dimensional case, so we can exploit some useful Sobolev embeddings, such as $H^{1/2}$ into $L^3$, $H^1$ into $L^6$, and $[H^1,H^2]_{1/2}$ into $L^\infty$. In doing so, we find that $f_i\in H^{k-1}$ even with the coefficients in $W^{1,\infty}\cap H^{k+1}$ (not necessarily in $W^{k+1,\infty}$). This is precisely what is done in the proof of Theorem \ref{hkplustworegularity}. Hence, we improve the classical result for any integer $k\geq 1$, in space dimension three. Analogous improvements can also be obtained in other dimensions.
\end{rmk}

\section{The viscoelastic Stokes problem}
\label{viscoelasticsec}
In this section, we apply the results of Section \ref{secgenstokes} to the viscoelastic Stokes problem \eqref{viscoelasticStokesorg} (or \eqref{viscoelasticStokes}). We first verify the conditions on $\sfB$ that guarantee that $\sfA = \Acal(\sfB)$, given according to \eqref{defAcal}, is a uniformly positive definite symmetric tensor, in the sense of Definition \ref{defuniformpositivedef}. Next we investigate under which regularity assumptions on $\sfB$ the tensor $\sfA=\Acal(\sfB)$ has the regularity conditions required in the results of Section \ref{secgenstokes}. Next we combine all these results to obtain the desired existence, uniqueness, and regularity properties of the viscoelastic Stokes problem.

\subsection{Conditions for ellipticity of the linear part of the viscoelastic Stokes problem}
\label{conditionsellipticity}

In applications, $\sfB$ is the left Cauchy-Green tensor $\sfF\sfF^\tr$, so we restrict our analysis to the case in which $\sfB$ is a positive definite symmetric tensor. We want to find conditions on $\mu_1$, $\mu_2$, $\mu_3$ and $\sfB$ for which $\Acal(\sfB)$, given by \eqref{defAcal}, is also a positive define symmetric tensor, which is one of the conditions needed in Theorems \ref{weaksolutionweakgeneralizedStokeswithpressure}, \ref{h2regularity} and \ref{hkplustworegularity}. It is clear that $\Acal(\sfB)$ is also symmetric, so we only need to investigate when it is positive definite.

In order to study this issue, it is natural to consider the real-valued functions
\begin{equation}
  \label{defg}
  g_\bx(\lambda) = \mu_1(\bx) + \mu_2(\bx)\lambda + \frac{\mu_3(\bx)}{\lambda}, \quad \text{for } \lambda > 0, \; \bx\in \Omega.
\end{equation}
Notice that if $\lambda$ is an eigenvalue of $\sfB(\bx)$, then $g_\bx(\lambda)$ is an eigenvalue of $\Acal(\sfB(\bx))$, with the same eigenspace. Hence, the positivity of $\Acal(\sfB(\bx))$ is directly related to the positivity of $g_\bx(\lambda)$, for every eigenvalue of $\sfB(\bx)$. For the sake of notational simplicity, given a symmetric matrix $\sfB$ in three dimensions, we denote by $\lambda_1(\sfB) \leq \lambda_2(\sfB) \leq \lambda_3(\sfB)$ the three eigenvalues of $\sfB$, in increasing order. Then, it is straightforward to deduce that the uniform positive definiteness of the tensor $\Acal(\sfB)$, in the sense of Definition \ref{defuniformpositivedef}, follows if, and only if, $g_\bx(\lambda_i(\sfB(\bx)))$ is strictly positive uniformly on $\Omega$, for every $i=1,2,3$.

We summarize the above facts in the following result. 
\begin{prop}
  \label{propessinfglambdacondition}
  Let $\sfB=\sfB(\bx)$ be a tensor on $\Omega$ and suppose that, for almost every $\bx$ in $\Omega$, $\sfB(\bx)$ is symmetric and positive definite with eigenvalues $\lambda_i(\sfB(\bx))>0$, $i=1,2,3$. Then, $\Acal(\sfB)$ is a symmetric tensor and the eigenvalues of $\Acal(\sfB(\bx))$ are given by $g_\bx(\lambda_i(\sfB(\bx))$, $i=1,2,3$, with the same corresponding eigenspaces of $\sfB(\bx)$. Moreover, $\sfB$ is $\Acal$-positive, in the sense of Definition \ref{defproper}, if, and only if, 
  \begin{equation}
    \label{essinfglambda}
    \alpha \define \essinf_{\bx\in \Omega} \left\{\min_{i=1,2,3} \left(\mu_1(\bx) + \mu_2(\bx)\lambda_i(\sfB(\bx)) + \frac{\mu_3(\bx)}{\lambda_i(\sfB(\bx))} \right)\right\} > 0.
  \end{equation}
  In this case, \eqref{unifellip} holds with $\sfA = \Acal(\sfB)$ and $\alpha$ as given in \eqref{essinfglambda}.
\end{prop}
\medskip

Typically, the parameters $\mu_1$, $\mu_2$, and $\mu_3$ may depend on some invariants of the left Cauchy-Green stress tensor $\sfB$ and/or on the temperature and other quantities, when such couplings are relevant. In some models, however, they are taken as constants. In the cases they are constant, we investigate in more details when the condition \eqref{essinfglambda} is satisfied. 

When these parameters do not change with $\bx$, the function in \eqref{defg} is the same for every $\bx$, i.e. $g=g_\bx$. Since $\sfB$ is a positive definite symmetric tensor, all its eigenvalues are positive, so, given a set of real-valued parameters $\mu_1, \mu_2, \mu_3$, we look for set of positive values of $\lambda$ for which $g$ is positive, i.e.
  \begin{equation}
    \label{defLambda}
     \Lambda_{\mu_1,\mu_2,\mu_3} = \left\{\lambda > 0; \;g(\lambda) = \mu_1 + \mu_2\lambda + \frac{\mu_3}{\lambda} > 0\right\}.
  \end{equation}

Since $g$ is a continuous functions, the following result is easy to obtain and its proof is omitted.
  
\begin{prop}
  \label{propmuscenarios}
  Let $\mu_1, \mu_2, \mu_3\in \RR$ be constants and let $\sfB$ be a positive definite symmetric tensor on a domain $\Omega\subset\RR^3$. Suppose there exists a set $\Lambda_0$ which is compactly included in the set $\Lambda_{\mu_1,\mu_2,\mu_3}$ defined in \eqref{defLambda} and such that $\lambda_i(\sfB(\bx))\in \Lambda_0$ almost everywhere on $\Omega$, for every $i=1, 2, 3$. Then, condition \eqref{essinfglambda} of Proposition \ref{propessinfglambdacondition} is fulfilled, and $\sfB$ is $\Acal$-positive in the sense of Definition \ref{defproper}.
\end{prop}
\medskip

  The set  $\Lambda_{\mu_1,\mu_2,\mu_3}$ can be explicitly and easily characterized in terms of the parameters. For that purpose, it is useful to introduce the polynomial $p(\lambda) = \lambda g(\lambda) = \mu_2\lambda^2 + \mu_1\lambda +\mu_3$, which has the same sign as $g$ for $\lambda$ positive. The roots of $p$ are real when $\mu_1^2 \geq 4\mu_2\mu_3$ and are given by
  \begin{equation}
    \label{proots}
    \lambda_{\pm} = \frac{-\mu_1 \pm \sqrt{\mu_1^2 - 4\mu_2\mu_3}}{2\mu_2}.
  \end{equation}
  
   One fundamental condition on the parameters comes when considering the state of the material with no displacement, in which case $\sfF$ is the identity tensor, and so is $\sfB$. Hence, all the eigenvalues of $\sfB(\bx)$ are equal to one, with $g(1)= \mu_1+\mu_2+\mu_3$. Thus, we recover the classic thermodynamic condition
  \begin{equation}
    \label{sumofmuspositive}
    \mu_1 + \mu_2 + \mu_3 > 0.
  \end{equation}
  Assuming condition \eqref{sumofmuspositive}, there are a few different scenarios that may happen, as follows.
  
  \begin{lem}
    \label{lemmuscenarios}
    Let $\mu_1, \mu_2, \mu_3\in \RR$ be constants, and consider the function $g$ and the set $\Lambda_{\mu_1,\mu_2,\mu_3}$ given in \eqref{defLambda}. Suppose that \eqref{sumofmuspositive} holds. Then, we have the following scenarios.
  \renewcommand{\labelenumi}{(\roman{enumi})}
  \begin{enumerate}
    \item If $\mu_2, \mu_3> 0$ and $\mu_1> -2\sqrt{\mu_2\mu_3}$, then $\Lambda_{\mu_1,\mu_2,\mu_3} = (0,\infty)$.
    \item If $\mu_2, \mu_3> 0$ and $\mu_1\leq -2\sqrt{\mu_2\mu_3}$, then $0 <\lambda_- \leq \lambda_+$, and $\Lambda_{\mu_1,\mu_2,\mu_3} = (0,\lambda_-)\cup (\lambda_+,\infty)$.
    \item If $\mu_3>0$, $\mu_2=0$, and $\mu_1\geq 0$, then $\Lambda_{\mu_1,\mu_2,\mu_3} = (0,\infty)$.
    \item If $\mu_3>0$, $\mu_2=0$, and $\mu_1<0$, then $\lambda_0 = -\mu_3/\mu_1>0$ and $\Lambda_{\mu_1,\mu_2,\mu_3} = (0, \lambda_0)$.
    \item If $\mu_3>0$ and $\mu_2<0$, then $\Lambda_{\mu_1,\mu_2,\mu_3} = (0, \lambda_+)$.
    \item If $\mu_3=0$ and $\mu_1,\mu_2\geq 0$, then $\Lambda_{\mu_1,\mu_2,\mu_3} = (0,\infty)$.
    \item If $\mu_3=0$, $\mu_2>0$, and $\mu_1<0$, then $\Lambda_{\mu_1,\mu_2,\mu_3} = (\lambda_+,\infty)$.
    \item If $\mu_3=0$, $\mu_2<0$, and $\mu_1>0$, then $\Lambda_{\mu_1,\mu_2,\mu_3} = (0,\lambda_+)$.
    \item If $\mu_3<0$ and $\mu_2>0$, then $\Lambda_{\mu_1,\mu_2,\mu_3} = (\lambda_+,\infty)$.
    \item If $\mu_3,\mu_2<0$ and $\mu_1> 2\sqrt{\mu_2\mu_3}$, then $\Lambda_{\mu_1,\mu_2,\mu_3} = (\lambda_-,\lambda_+)$.
    \item If $\mu_3<0$, $\mu_2=0$, and $\mu_1>0$, then $\lambda_0 = -\mu_3/\mu_1>0$ and $\Lambda_{\mu_1,\mu_2,\mu_3} = (\lambda_0,\infty)$.
  \end{enumerate}
\end{lem}
\medskip

We end this section with a few remarks.
\begin{rmk}[Small Deformations]
  \label{rmksmalldeformations}
  Without any deformation, the tensor $\sfF$ is the identity $\sfI$, and so is the left Cauchy-Green tensor $\sfB=\sfF\sfF^\tr = \sfI$. In this case, $\sfA=\Acal(\sfI)$ is a diagonal tensor, with identical diagonal elements given by $\mu_1+\mu_2+\mu_3$. Under the assumption \eqref{sumofmuspositive}, $\sfA$ is uniformly elliptic.  Then, by continuity, it is clear that, for small perturbations, the eigenvalues of $\sfB$ are near the identity and, hence, the eigenvalues of $\sfA=\Acal(\sfB)$ are still strictly positive. More precisely, one has that there exists $\delta>0$ such that if $\sfB\in L^\infty(\Omega)^{3\times 3}$ is a positive definite symmetric tensor such that $\|\sfB-\sfI\|_{L^\infty(\Omega)} \leq \delta$, then $\sfB$ is still an $\Acal$-positive tensor. The value of $\delta$ can be estimated in terms of $\mu_1,\mu_2,\mu_3$ depending on the case described in Proposition \ref{propmuscenarios}. In particular, the material may, or may not, allow arbitrarily large deformations.
\end{rmk}


\begin{rmk}[Salt dynamics]
In the time-dependent problem associated with the constitutive law \eqref{sigmaconstitutive}, we consider an Eulerian formulation in which $\sfB$ is a time-dependent variable of the system. In the initial value problem for this system, without any initial deformation, $\sfB$ is taken to be the identity tensor at the initial time. In this case, the associate tensor $\sfA=\Acal(\sfB)$ is initially a diagonal operator with diagonal element equal to $\mu_1+\mu_2+\mu_3>0$ (see Remark \ref{rmksmalldeformations}). Using energy estimates of the associated evolution equations and the regularization property of problem \eqref{generalizedStokes} proved here, we obtain a~priori estimates for $\sfB$ in $H^3(\Omega)^{3\times 3}\subset W^{1,\infty}(\Omega)$ and for $\bv$ in $H^4(\Omega)^3$, showing that $\sfB$ remains in $H^3(\Omega)^{3\times 3}$ and stays $\Acal$-positive for a short time, while $\bv$ remains in $H^4(\Omega)^3$. In doing so, we prove that the system is well-posed locally in time. This result will be presented elsewhere.
\end{rmk}

\subsection{Regularity of $\Acal(\sfB)$}
  \label{functionalestimatesforAcalofB}

In view of the regularity conditions on $\sfA$ needed for the existence, uniqueness, and regularity results given in Section \ref{secgenstokes}, we want to obtain the corresponding regularity properties of $\sfA=\Acal(\sfB)$ in terms of $\sfB$. This is the aim of the current section. Along the way, some useful estimates are derived.

One step is to represent the inverse $\sfB^{-1}$ in terms of $\sfB^2$, using the representation 
\[ p_\sfB(\lambda) = \lambda^3 - \rm{I}_\sfB \lambda^2 + \rm{II}_\sfB \lambda - \rm{III}_\sfB 
\]
for the characteristic polynomial of a tridimensional matrix $\sfB(\bx)$, at each point $\bx$. In this representation,
\begin{align*} 
  \rm{I}_\sfB & = \Tr \sfB, \\
  \rm{II}_\sfB & =  \frac{1}{2}\left( (\Tr \sfB)^2 - \Tr(\sfB^2)\right), \\
  \rm{III}_\sfB & = \frac{1}{6}\left((\Tr \sfB)^3 - 3(\Tr \sfB) (\Tr(\sfB^2)) + 2 \Tr(\sfB^3)\right) = \det \sfB,
\end{align*}
where $\Tr(\cdot)$ denotes the trace of a matrix or tensor, and $\det(\cdot)$, its determinant. We use the Cayley-Hamilton Theorem, which asserts that
\begin{equation}
  \label{cayleyhamilton} 
  p_\sfB(\sfB) = \sfB^3 - \rm{I}_\sfB \sfB^2 + \rm{II}_\sfB \sfB - \rm{III}_\sfB \sfI = 0, \quad \text{ on } \Omega.
\end{equation}
Hence,
  \begin{equation} 
    \label{cayleyhamiltonbinv}
    \sfB^{-1} = \frac{1}{\det \sfB}(\sfB^2 - \rm{I}_\sfB \sfB + \rm{II}_\sfB \sfI).
  \end{equation}
  
  Using this representation, we first obtain an $L^\infty$ regularity result for $\Acal(\sfB)$.  
\begin{prop}
  \label{proplinftyboundonacalb}
  If $\sfB\in L^\infty(\Omega)^{3\times 3}$ and $(\det B)^{-1}\in L^\infty(\Omega)$, then, $\sfB^{-1}\in L^\infty(\Omega)^{3\times 3}$. If, moreover, $\mu_1, \mu_2, \mu_3\in L^\infty(\Omega)$, then $\Acal(\sfB)\in L^\infty(\Omega)^{3\times 3}$.
\end{prop}

\begin{proof}
Clearly, $\sfB^2$, $\Tr \sfB$, and $\Tr \sfB^2$ are essentially bounded, with 
\[  \|\sfB^2\|_{L^\infty(\Omega)} \leq 3\|\sfB\|_{L^\infty(\Omega)}^2, \quad \|\Tr\sfB\|_{L^\infty(\Omega)} \leq 3\|\sfB\|_{L^\infty(\Omega)}, \quad \|\Tr\sfB^2\|_{L^\infty(\Omega)} \leq 9\|\sfB\|_{L^\infty(\Omega)}^2,
\]
and, hence,
\[ \|\rm{I}_{\sfB}\sfB\|_{L^\infty(\Omega)} \leq 3\|\sfB\|_{L^\infty(\Omega)}^2, \quad \|\rm{II}_{\sfB}\sfI\|_{L^\infty(\Omega)} \leq 9\|\sfB\|_{L^\infty(\Omega)}^2.
\]
Using these estimates in \eqref{cayleyhamiltonbinv} implies that $\sfB^{-1}\in L^\infty(\Omega)$, with
  \begin{equation}
    \label{estimateBinvlinfty}
    \|\sfB^{-1}\|_{L^\infty(\Omega)} \leq 15\|(\det\sfB)^{-1}\|_{L^\infty(\Omega)} \|\sfB\|_{L^\infty(\Omega)}^2.
  \end{equation}
  Then, using \eqref{estimateBinvlinfty} in \eqref{defAcal}, we obtain
    \begin{equation}
    \label{estimateAcalBlinfty}
    \|\Acal(\sfB)\|_{L^\infty(\Omega)} \leq \|\mu_1\|_{L^\infty(\Omega)} + \|\mu_2\|_{L^\infty(\Omega)}\|\sfB\|_{L^\infty(\Omega)} + 15 \|\mu_3\|_{L^\infty(\Omega)}\|(\det\sfB)^{-1}\|_{L^\infty(\Omega)}\|\sfB\|_{L^\infty(\Omega)}^2,
  \end{equation}
  which proves that $\Acal(\sfB)\in L^\infty(\Omega)^{3\times 3}$.
\end{proof}
\medskip

For the sake of simplicity and with the applications in mind, we restrict the next estimates to the relevant case $\det\sfB=1$. We then look for $W^{1,\infty}(\Omega)$ regularity.
\begin{prop}
  \label{propwoneinftyboundonacalb}
  If $\sfB\in W^{1,\infty}(\Omega)^{3\times 3}$ and $\det B = 1$ almost everywhere on $\Omega$, then, $\sfB^{-1}\in W^{1,\infty}(\Omega)^{3\times 3}$. If, moreover, $\mu_1, \mu_2, \mu_3\in W^{1,\infty}(\Omega)$, then $\Acal(\sfB)\in W^{1,\infty}(\Omega)^{3\times 3}$.
\end{prop}

\begin{proof}
That $\sfB^{-1}$ and $\Acal(\sfB)$ belong to $L^\infty(\Omega)^{3\times 3}$ follows from Proposition \ref{proplinftyboundonacalb}. We only need to consider the derivatives of such tensors.

Since the trace is a linear operator and taking into account that $\det\sfB=1$, the derivatives $\partial_{x_i}\left(\sfB^{-1}\right)$ are easily found from \eqref{cayleyhamiltonbinv} to be
\begin{equation}
  \label{partialxiofinvB}
  \partial_{x_i} \left(\sfB^{-1}\right) = 2\sfB\partial_{x_i}\sfB - \Tr(\partial_{x_i}\sfB) \sfB + \Tr(\sfB)\partial_{x_i}\sfB + \left( \Tr(\sfB)\Tr(\partial_{x_i}\sfB) -  \Tr(\sfB\partial_{x_i}\sfB)\right) \sfI.
\end{equation}
for every $i=1,2,3.$ Thus,
\[ \|\partial_{x_i} \left(\sfB^{-1}\right)\|_{L^\infty(\Omega)} \leq 20\|\sfB\|_{L^\infty(\Omega)}\|\partial_{x_i}\sfB\|_{L^\infty(\Omega)},
\]
This yields 
  \begin{equation}
    \label{estimateBinvwoneinfty}
    \|D(\sfB^{-1})\|_{L^\infty(\Omega)} \leq 20\|\sfB\|_{L^\infty(\Omega)} \|D\sfB\|_{L^\infty(\Omega)}.
  \end{equation}
From \eqref{defAcal} we find that
\begin{multline*}
  \|\partial_{x_i}\Acal(\sfB)\|_{L^\infty(\Omega)} \leq \|\partial_{x_i}\mu_1\|_{L^\infty(\Omega)} + \|\partial_{x_i}\mu_2\|_{L^\infty(\Omega)}\|\sfB\|_{L^\infty(\Omega)} + \|\mu_2\|_{L^\infty(\Omega)} \|\partial_{x_i}\sfB\|_{L^\infty(\Omega)} \\
  + \|\partial_{x_i}\mu_3\|_{L^\infty(\Omega)}\|\sfB^{-1}\|_{L^\infty(\Omega)} + \|\mu_3\|_{L^\infty(\Omega)}\|\partial_{x_i}(\sfB^{-1})\|_{L^\infty(\Omega)}.
\end{multline*}
which, together with \eqref{estimateBinvlinfty} and \eqref{estimateBinvwoneinfty}, yields   
\begin{multline}
    \label{estimateAcalBwoneinfty}
    \|D\Acal(\sfB)\|_{L^\infty(\Omega)} \leq \|D\mu_1\|_{L^\infty(\Omega)} + \|D\mu_2\|_{L^\infty(\Omega)}\|\sfB\|_{L^\infty(\Omega)} + \|\mu_2\|_{L^\infty(\Omega)}\|D\sfB\|_{L^\infty(\Omega)} \\
    + 9 \|D\mu_3\|_{L^\infty(\Omega)}\|\sfB\|_{L^\infty(\Omega)}^2 +20\|\mu_3\|_{L^\infty(\Omega)}\|\sfB\|_{L^\infty(\Omega)}\|D\sfB\|_{L^\infty(\Omega)}.
  \end{multline}
  This proves the desired regularity for $\Acal(\sfB)$.
\end{proof}

Now we look for a $W^{2,3}(\Omega)$ estimate.
\begin{prop}
  \label{propw23boundonacalb}
  If $\sfB\in W^{2,3}(\Omega)^{3\times 3}$ and $\det B = 1$ almost everywhere on $\Omega$, then, $\sfB^{-1}\in W^{2,3}(\Omega)^{3\times 3}$. If, moreover, $\mu_1, \mu_2, \mu_3\in W^{1,\infty}(\Omega)\cap W^{2,3}(\Omega)$, then $\Acal(\sfB)\in W^{2,3}(\Omega)^{3\times 3}$.
\end{prop}

\begin{proof}
From \eqref{cayleyhamiltonbinv}, and using that $\det \sfB =1$ almost everywhere, we see that
\[ \|\sfB^{-1}\|_{L^3(\Omega)} \leq c \|\sfB\|_{L^6(\Omega)}^2,
\] 
for some universal constant $c$. Using \eqref{partialxiofinvB}, we find that
\[ \|D\sfB^{-1}\|_{L^3(\Omega)} \leq c \|\sfB\|_{L^6(\Omega)}\|D\sfB\|_{L^6(\Omega)},
\]
for a different universal constant $c$. Taking the $\partial_{x_j}$ derivative of \eqref{partialxiofinvB} we find that
\begin{multline}
  \label{partialxixjofinvB}
  \partial_{x_i x_j}^2 \left(\sfB^{-1}\right) = 2\partial_{x_j}\sfB\partial_{x_i}\sfB + 2\sfB\partial_{x_i x_j}^2\sfB \\
  - \Tr(\partial_{x_ix_j}^2\sfB) \sfB - \Tr(\partial_{x_i}\sfB) \partial_{x_j} \sfB + \Tr(\partial_{x_j}\sfB)\partial_{x_i}\sfB + \Tr(\sfB)\partial_{x_i x_j}^2\sfB \\
  + \left( \Tr(\partial_{x_j}\sfB)\Tr(\partial_{x_i}\sfB) + \Tr(\sfB)\Tr(\partial_{x_i x_j}^2\sfB) -  \Tr(\partial_{x_j}\sfB\partial_{x_i}\sfB) -  \Tr(\sfB\partial_{x_i x_j}^2\sfB)\right) \sfI,
\end{multline}
for every $i=1,2,3.$ Then, using H\"older's inequality, we find that
  \begin{equation}
    \label{estimateBinvw23}
    \|D^2(\sfB^{-1})\|_{L^3(\Omega)} \leq c \left(\|D\sfB\|_{L^6(\Omega)}^2 +  \|\sfB\|_{L^\infty(\Omega)}\|D^2\sfB\|_{L^3(\Omega)}\right),
  \end{equation}
  for another suitable universal constant $c$. These estimates show that $\sfB^{-1}$ belongs to $W^{2,3}(\Omega)^{3\times 3}$. Now, from \eqref{defAcal} and the regularity obtained for $\sfB^{-1}$, we readily see that $\Acal(\sfB)\in W^{2,3}(\Omega)^{3\times 3}$, with, in particular,
  \begin{multline}
    \label{estimateAcalBw23}
    \|D^2\Acal(\sfB)\|_{L^3(\Omega)} \leq c\left(\|D^2\mu_1\|_{L^3(\Omega)} 
      + \|D^2\mu_2\|_{L^3(\Omega)}\|\sfB\|_{L^\infty(\Omega)} 
      + \|D\mu_2\|_{L^\infty(\Omega)}\|D\sfB\|_{L^3(\Omega)} \right. \\ + \|\mu_2\|_{L^\infty(\Omega)}\|D\sfB\|_{L^3(\Omega)} 
+ \|D^2\mu_3\|_{L^3(\Omega)}\|\sfB^{-1}\|_{L^\infty(\Omega)} \\ \left. + \|D\mu_3\|_{L^\infty(\Omega)}\|D(\sfB^{-1})\|_{L^3(\Omega)} + \|\mu_3\|_{L^\infty(\Omega)}\|D^2(\sfB^{-1})\|_{L^3(\Omega)}\right),
  \end{multline}
  for a suitable universal constant $c$.
\end{proof}

Finally we look for $H^m(\Omega)$ estimates, for any integer $m\geq 2$.
\begin{prop}
\label{prophmboundonacalb}
If $\sfB\in H^m(\Omega)$ and $\det B = 1$ almost everywhere on $\Omega$, where $m\in \NN$, $m\geq 2$, then, $\sfB^{-1}\in H^m(\Omega)^{3\times 3}$. If, moreover, $\mu_1, \mu_2, \mu_3\in W^{1,\infty}(\Omega)
\cap H^m(\Omega)^{m\times m}$, then $\Acal(\sfB)\in H^m(\Omega)^{3\times 3}$.
\end{prop}

\begin{proof}
Since $\sfB\in H^m(\Omega)^{3\times 3} \subset L^4(\Omega)^{3\times 3}$, for $m\geq 2$, it is straightforward to see that
  \begin{equation}
    \label{estimateBinvl2}
    \|\sfB^{-1}\|_{L^2(\Omega)} \leq c \|\sfB\|_{L^4(\Omega)}^2,
  \end{equation}
for a suitable universal constant $c$, so that $\sfB^{-1}$ is in $L^2(\Omega)^{3\times 3}$. Using H\"older's inequality in \eqref{partialxiofinvB} we see that
  \begin{equation}
    \label{estimateBinvh1}
    \|D(\sfB^{-1})\|_{L^2(\Omega)} \leq c\|\sfB\|_{L^{12}(\Omega)}\|D\sfB\|_{L^{12/5}(\Omega)}^2,
  \end{equation}
for another universal constant $c$. Since $H^m(\Omega)\subset L^{12}(\Omega)\cap W^{1,12/5}(\Omega)$, this shows that $\sfB^{-1}\in H^1(\Omega)^{3\times 3}$. From \eqref{partialxixjofinvB}, we see that
\begin{equation}
  \label{estimateBinvh2}
  \|D^2(\sfB^{-1})\|_{L^2(\Omega)} \leq c \left(\|D\sfB\|_{L^4(\Omega)}^2 +  \|\sfB\|_{L^\infty(\Omega)}\|D^2\sfB\|_{L^2(\Omega)}\right),
  \end{equation}
  for another suitable universal constant $c$. Since in three dimensions we have the embedding $H^m\subset L^\infty$, for $m\geq 2$, it follows that $\sfB^{-1}\in H^2(\Omega)^{3\times 3}$ as well. This proves the case $m=2$. For $m = 3, 4, \ldots$, we use Leibniz's rule and estimate $D^m(\sfB^{-1})$ simply by
  \[ \|D^m(\sfB^{-1})\|_{L^2(\Omega)} \leq c\sum_{i=1}^{[m/2]} \|D^i\sfB\|_{L^\infty(\Omega)} \|D^{m-i}\sfB\|_{L^2(\Omega)},
  \]
  where $[m/2]$ is the largest integer smaller than or equal to $m/2$. When $m\geq 3$ and the space is three-dimensional, the embedding $H^m(\Omega)\subset W^{i,\infty}$ holds for any integer $i\leq m/2$. Thus, $\sfB^{-1}\in H^m(\Omega)^{3\times 3}$.
  
  For the regularity of $\Acal(\sfB)$, we see that
  \begin{equation}
    \label{estimateAcalBinl2}
    \|\Acal(\sfB)\|_{L^2(\Omega)} \leq c\left(\|\mu_1\|_{L^2(\Omega)} + \|\mu_2\|_{L^\infty(\Omega)}\|\sfB\|_{L^2(\Omega)} + \|\mu_3\|_{L^\infty(\Omega)}\|\sfB^{-1}\|_{L^2(\Omega)}\right),
  \end{equation}
  for a universal constant $c$, and
  \begin{multline}
    \label{estimateAcalBinhm}
    \|D^m\Acal(\sfB)\|_{L^2(\Omega)} \leq c\left(\|D^m\mu_1\|_{L^2(\Omega)} + \|\mu_2\|_{L^\infty(\Omega)}\|D^m\sfB\|_{L^2(\Omega)} + \|D\mu_2\|_{L^\infty(\Omega)}\|D^{m-1}\sfB\|_{L^2(\Omega)} \right. \\
      + \|\mu_3\|_{L^\infty(\Omega)}\|D^m(\sfB^{-1})\|_{L^2(\Omega)} + \|D\mu_3\|_{L^\infty(\Omega)}\|D^{m-1}(\sfB^{-1})\|_{L^2(\Omega)} \\
      \left. + \sum_{i=2}^m \|D^i\mu_2\|_{L^2(\Omega)}\|D^{m-i}\sfB\|_{L^\infty(\Omega)} + \sum_{i=2}^m \|D^i\mu_3\|_{L^2(\Omega)}\|D^{m-i}(\sfB^{-1})\|_{L^\infty(\Omega)} \right),
  \end{multline}
  for another suitable constant $c$. From the embedding $H^m(\Omega) \subset W^{m-i,\infty}$, for $m\geq 2$ and $2\leq i \leq m$, valid in three dimensions, it follows from \eqref{estimateAcalBinl2} and \eqref{estimateAcalBinhm} that $\Acal(\sfB)\in H^m(\Omega)$.
\end{proof}

\begin{rmk}
  If the regularities of $\Acal(\sfB)$ in $W^{2,3}(\Omega)^{3\times 3}$ and in $H^m(\Omega)^{3\times 3}$ were the sole purposes of Propositions \ref{propw23boundonacalb} and \ref{prophmboundonacalb}, respectively, then, of course, one could ask for different regularity assumptions on $\sfB$ and on the parameters $\mu_1, \mu_2$, and $\mu_3$. But since these estimates will be used in conjunction with Theorem \ref{hkplustworegularity}, which requires that $\sfA = \Acal(\sfB)$ belong to $W^{1,\infty}(\Omega)^{3\times 3}$, then we are constrained by starting with the assumptions required in Proposition \ref{propwoneinftyboundonacalb}.
\end{rmk}

\begin{rmk}
  More general conditions for the regularity of $\Acal(\sfB)$ can be obtained by imposing conditions not only on $\sfB$ but also on $\sfB^{-1}$ itself. However, for practical purposes, since in the evolutionary system \eqref{LiueulerB} the tensor $\sfB$ is one of the unknowns, it is useful to derive conditions on $\sfB$ alone.
\end{rmk}

\subsection{Existence and regularity results for the viscoelastic Stokes problem}
\label{existuniqregviscoelasticsec}

The existence, uniqueness, and regularity properties for the viscoelastic Stokes problem \eqref{viscoelasticStokes} are direct consequences of the corresponding results for the generalized Stokes problem \eqref{generalizedStokes}, provided the tensor $\sfB$ is $\Acal$-positive, which guarantees that the corresponding tensor $\sfA=\Acal(\sfB)$ is uniformly positive definite, and provided the tensor $\Acal(\sfB)$ is sufficiently regular. The main condition for $\sfB$ to be $\Acal$-positive is given in Proposition \ref{propessinfglambdacondition}, in Section \ref{conditionsellipticity}, while the regularity of $\Acal(\sfB)$ is investigated in Propositions \ref{proplinftyboundonacalb}, \ref{propwoneinftyboundonacalb}, \ref{propw23boundonacalb}, and \ref{prophmboundonacalb}, in Section \ref{functionalestimatesforAcalofB}. The aim of the current section is to put these things together. With that in mind, we state, without proof, the following results, corresponding to Theorems \ref{weaksolutionweakgeneralizedStokeswithpressure}, \ref{h2regularity}, and \ref{hkplustworegularity}, respectively.

\begin{thm}[Existence and uniqueness]
  \label{weaksolutionweakviscoelasticdStokeswithpressure}
  Let $\Omega\subset \RR^3$ be a bounded Lipschitz domain, $\sfB\in L^\infty(\Omega)^{3\times 3}$ be a uniformly positive definite symmetric tensor on $\Omega$ with $(\det \sfB)^{-1}\in L^\infty(\Omega)$, $\mu_1$, $\mu_2$, and $\mu_3 \in L^\infty(\Omega)$ be scalar fields, and $\bbf\in H^{-1}(\Omega)^3$. Assume that \eqref{essinfglambda} holds. Then, there exist a unique vector field $\bv\in V(\Omega)$ and a scalar field $p\in L^2(\Omega)$ that solve the weak formulation \eqref{weakviscoelasticStokeswithpressure} of the viscoelastic Stokes problem. Moreover, the field $p$ is unique up to a constant, and the estimates
  \begin{equation}
    \label{weaksolutionweakgeneralizedStokesestimaterepeatwithB}
    \|\bnabla \bv\|_{L^2(\Omega)} \leq \frac{1}{\alpha} \|\bbf\|_{H^{-1}(\Omega)},
  \end{equation}
  and  
  \begin{equation}
    \label{weaksolutionweakgeneralizedStokesestimatewithpressurewithB}
    \|p\|_{L^2(\Omega)/\RR} \leq \frac{c(\Omega)}{\alpha} \|\Acal(\sfB)\|_{L^\infty(\Omega)}\|\bbf\|_{H^{-1}(\Omega)},
  \end{equation}
  hold, where $c(\Omega)$ is a shape constant.
\end{thm}

\begin{thm}[Regularity]
  \label{h2regularityforviscoelasticStokes}
  Under the hypotheses of Theorem \ref{weaksolutionweakviscoelasticdStokeswithpressure}, suppose moreover that $\Omega$ has a boundary of class $\Ccal^2$, $\sfB\in W^{1,\infty}(\Omega)^{3\times 3}$ with $\det \sfB = 1$, $\mu_1$, $\mu_2$, and $\mu_3\in W^{1,\infty}(\Omega)$, and $\bbf \in L^2(\Omega)^3$. Then, the solution $(\bv, p)$ of the weak formulation \eqref{weakviscoelasticStokeswithpressure} of the viscoelastic Stokes problem is such that $\bv \in V(\Omega)\cap H^2(\Omega)$ and $p\in H^1(\Omega)$, with 
  \begin{equation}
    \label{h2regularityforh2vwithB}
    \|D^2 \bv\|_{L^2(\Omega)} \leq \frac{c(\Omega)}{\alpha}\left( \|\bbf\|_{L^2(\Omega)} + \frac{1}{\alpha}\|\Acal(\sfB)\|_{W^{1,\infty}_{\lambda_1}(\Omega)}\|\bbf\|_{H^{-1}(\Omega)}\right),
  \end{equation}
  and
  \begin{equation}
    \label{h2regularityforh1pwithB}
    \|\bnabla p\|_{L^2(\Omega)} \leq \frac{c(\Omega)}{\alpha}\|\Acal(\sfB)\|_{L^\infty(\Omega)}\left( \|\bbf\|_{L^2(\Omega)} + \frac{1}{\alpha}\|\Acal(\sfB)\|_{W^{1,\infty}_{\lambda_1}(\Omega)}\|\bbf\|_{H^{-1}(\Omega)}\right),
  \end{equation}
  where $\lambda_1$ is the Poincar\'e constant, and $c(\Omega)$ is a shape constant.  Furthermore, $(\bv,p)$ solves the viscoelastic Stokes system \eqref{viscoelasticStokesorg} almost everywhere on $\Omega$, with the first two equations in \eqref{viscoelasticStokesorg} holding in $L^2(\Omega)^3$ and $H^1(\Omega)^3$, respectively, and with the last equation holding for the trace of $\bv$ on $\partial\Omega$ in the space $H^{3/2}(\partial \Omega)$.
\end{thm}

\begin{thm}[High-order regularity]
  \label{hkplustworegularityforviscoelasticStokes}
  Under the hypotheses of Theorem \ref{weaksolutionweakviscoelasticdStokeswithpressure}, suppose moreover that $\Omega$ has a boundary of class $\Ccal^{k+2}$, $\sfB\in W^{1,\infty}(\Omega)^{3\times 3}\cap W^{2,3}(\Omega)^{3\times 3} \cap H^{k+1}(\Omega)^{3\times 3}$ with $\det \sfB = 1$, $\mu_1$, $\mu_2$, and $\mu_3\in W^{1,\infty}(\Omega)\cap W^{2,3}(\Omega)\cap H^{k+1}(\Omega)$, and $\bbf \in H^k(\Omega)^3$. Then, the solution $(\bv, p)$ of the weak formulation \eqref{weakviscoelasticStokeswithpressure} of the viscoelastic Stokes problem is such that $\bv \in V(\Omega)\cap H^{k+2}(\Omega)$ and $p\in H^{k+1}(\Omega)$. Furthermore, $(\bv,p)$ solves the viscoelastic Stokes problem \eqref{viscoelasticStokesorg} almost everywhere on $\Omega$, with the first two equations in \eqref{viscoelasticStokesorg} holding in $H^{k-1}(\Omega)^3$ and $H^k(\Omega)$, respectively, and with the last equation holding for the trace of $\bv$ on $\partial\Omega$, in the space $H^{k+1/2}(\partial \Omega)$. Finally, depending on the value of $k$, the estimates \eqref{h2egularityforp}, \eqref{h2egularityforp}, \eqref{hkplus2regularityforv}, and \eqref{hkplus1regularityforp} hold with $\sfA$ replaced by $\Acal(\sfB)$.
\end{thm}

\begin{rmk}
  The estimates needed for $\Acal(\sfB)$ in Theorems \ref{weaksolutionweakviscoelasticdStokeswithpressure}, \ref{h2regularityforviscoelasticStokes}, and \ref{hkplustworegularityforviscoelasticStokes} can be taken to be those obtained in Propositions \ref{proplinftyboundonacalb}, \ref{propwoneinftyboundonacalb}, \ref{propw23boundonacalb}, and \ref{prophmboundonacalb}, but, of course, other estimates may be worked out if desired.
\end{rmk}

\begin{rmk}
  In a more general form, one obtains results analogous to Theorems \ref{weaksolutionweakviscoelasticdStokeswithpressure}, \ref{h2regularityforviscoelasticStokes}, and \ref{hkplustworegularityforviscoelasticStokes} without assuming explicit regularity conditions on $\sfB$ and instead assuming directly that the tensor $\sfB$ is such that $\Acal(\sfB)$ satisfies the assumptions for $\sfA$ in Theorems \ref{weaksolutionweakgeneralizedStokeswithpressure}, \ref{h2regularity}, and \ref{hkplustworegularity}, respectively. This is particularly useful when the material can be modelled with $\mu_3=0$ or in special situations or with special symmetries in which the regularity of $\Acal(\sfB)$ might be improved, or even if $\Acal(\sfB)$ depends on $\sfB$ on a different way.
\end{rmk}

\begin{rmk}
The same kind of results (in Theorems \ref{weaksolutionweakgeneralizedStokeswithpressure}, \ref{h2regularity} and \ref{hkplustworegularity}, for the generalized Stokes problem, and in Theorems \ref{weaksolutionweakviscoelasticdStokeswithpressure}, \ref{h2regularityforviscoelasticStokes}, and \ref{hkplustworegularityforviscoelasticStokes}, for the viscoelastic Stokes problem) hold under different homogenous boundary condition, such as fully periodic conditions or combinations of no-slip and periodic boundary conditions such as in a periodic channel. Non-homogenous boundary conditions can also be treated, provided they are the trace of functions with sufficient regularity. The details are omitted. 
\end{rmk}

\section{Conclusions}

We considered in detail a particular class of viscoelastic materials, modelled by system \eqref{viscoelasticStokesorg}. Our main interest was in modelling the deformation of different types of salts, needed for simulations of the exploitation of pre-salt oil fields. 

We obtained existence, uniqueness and regularity results for system \eqref{viscoelasticStokesorg}, based on natural conditions on the parameters $\sfB$ and $\mu_1, \mu_2,\mu_3$, and on the boundary $\Omega$. In particular, we identified the condition \eqref{essinfglambda} as the main condition for the uniform ellipticity of the problem, as well as the regularity conditions given in Section \ref{viscoelasticsec}.

For the existence, uniqueness and low-order regularity (up to $H^2$), we rewrote the system into a form suitable for the application of previously known results and verified all the necessary conditions.

For higher-order regularity (in $H^k$, $k\geq 3$), we exploited the Sobolev estimates valid in three dimensions and improved the regularity needed on the coefficients of the differential operator as compared with previously available results for similar elliptic systems.

We also obtained sharper and more explicit estimates for the regularity of the solution in terms of the parameters of the problem. These estimates are important for the well-posedness of the corresponding evolutionary problem to be presented in a forthcoming paper.

In regards to the ellipticity condition \eqref{essinfglambda}, which is of practical importance for checking the validity of a numerical simulation, we identified, in the case of constant coefficients $\mu_1, \mu_2, \mu_3$,  different intervals of the positive real line in which the eigenvalues of the left Cauchy-Green deformation tensor are allowed to belong to (see Lemma \ref{lemmuscenarios}). If the deformation is such that at least one of the eigenvalues escapes those intervals, the system is no longer elliptic. This explains the breakdown of our own simulations seen in some situations.

\end{document}